\documentclass{amsart}
\usepackage{color}

\usepackage{amsmath} 
\usepackage{amssymb}
\usepackage{mathrsfs}

\newtheorem{theorem}{Theorem}[section] 
\newtheorem{claim}[theorem]{Claim}
\newtheorem{cd}[theorem]{Claim/Definition}

\newtheorem{conclusion}[theorem]{Conclusion}
\newtheorem{observation}[theorem]{Observation}

\theoremstyle{definition}
\newtheorem{definition}[theorem]{Definition}
\newtheorem{dc}[theorem]{Definition/Claim}

\newtheorem{discussion}[theorem]{Discussion}
\newtheorem{convention}[theorem]{Convention}

\newtheorem{hypothesis}[theorem]{Hypothesis}

\theoremstyle{remark}
\newtheorem{remark}[theorem]{Remark}
\newtheorem{question}[theorem]{Question}
\newtheorem{notation}[theorem]{Notation}
\newtheorem{context}[theorem]{Context}

\newcommand{\tp}{{\rm tp}}
\newcommand{\tr}{{\rm tr}}
\newcommand{\rnd}{{\rm rnd}}
\newcommand{\card}{{\rm card}}
\newcommand{\dm}{{\rm dm}}
\newcommand{\Ceb}{{\rm Cb}}
\newcommand{\uq}{{\rm uq}}
\newcommand{\cn}{{\rm cn}}
\newcommand{\ut}{{\rm ut}}

\newcommand{\st}{{\rm st}}

\newcommand{\Car}{{\rm Car}}
\newcommand{\oor}{{\rm or}}
\newcommand{\vr}{{\rm vr}}

\newcommand{\wg}{{\rm wg}}

\newcommand{\Ax}{{\rm Ax}}
\newcommand{\na}{{\rm na}}

\newcommand{\wk}{{\rm wk}}

\newcommand{\aaa}{{\rm a}}

\newcommand{\vq}{{\rm vq}}
\newcommand{\qr}{{\rm qr}}

\newcommand{\NF}{{\rm NF}}
\newcommand{\PC}{{\rm PC}}
\newcommand{\EM}{{\rm EM}}
\newcommand{\LST}{{\rm LST}}
\newcommand{\arity}{{\rm arity}}

\newcommand{\seq}{{\rm seq}}
\newcommand{\rct}{{\rm rct}}
\newcommand{\good}{{\rm good}}

\newcommand{\acl}{{\rm acl}}
\newcommand{\reg}{{\rm reg}}
\newcommand{\dcl}{{\rm dcl}}
\newcommand{\inv}{{\rm inv}}
\newcommand{\ess}{{\rm ess}}
\newcommand{\jct}{{\rm jct}}
\newcommand{\stp}{{\rm stp}}
\newcommand{\at}{{\rm at}}
\newcommand{\Ord}{{\rm Ord}}

\newcommand{\bas}{{\rm bas}}

\newcommand{\brim}{{\rm brim}}
\newcommand{\Depth}{{\rm Depth}}
\newcommand{\full}{{\rm full}}
\newcommand{\bs}{{\rm bs}}
\newcommand{\bd}{{\rm bd}}
\newcommand{\eq}{{\rm eq}}

\newcommand{\id}{{\rm id}}

\newcommand{\ortp}{{\rm ortp}}

\newcommand{\ict}{{\rm ict}}
\newcommand{\fin}{{\rm fin}}

\newcommand{\Min}{{\rm Min}}
\newcommand{\Dom}{{\rm Dom}}
\newcommand{\NDOP}{{\rm NDOP}}
\newcommand{\Rang}{{\rm Rang}}

\newcommand{\rest}{{\restriction}}
\newcommand{\dom}{{\rm dom}}

\newcommand{\wilog}{{\rm without loss of generality}}
\newcommand{\Wilog}{{\rm Without loss of generality}}

\newcommand{\then}{{\underline{then}}}
\newcommand{\when}{{\underline{when}}}
\newcommand{\Then}{{\underline{Then}}}

\newcommand{\If}{{\underline{if}}}
\newcommand{\Iff}{{\underline{iff}}}
\newcommand{\mn}{{\medskip\noindent}}
\newcommand{\sn}{{\smallskip\noindent}}
\newcommand{\bn}{{\bigskip\noindent}}

\newcommand{\cL}{{\mathscr L}}

\newcommand{\gy}{{\mathfrak y}}
\newcommand{\gx}{{\mathfrak x}}
\newcommand{\ga}{{\mathfrak a}}

\newcommand{\gn}{{\mathfrak n}}
\newcommand{\gC}{{\mathfrak C}}

\newcommand{\gK}{{\mathfrak K}}
\newcommand{\gk}{{\mathfrak k}}
\newcommand{\cC}{{\mathscr C}}

\newcommand{\bbE}{{\mathbb E}}

\newcommand{\cE}{{\mathscr E}}

\newcommand{\bbL}{{\mathbb L}}

\newcommand{\cP}{{\mathscr P}}

\newcommand{\gs}{{\mathfrak s}}

\newcommand{\cS}{{\mathscr S}}
\newcommand{\cT}{{\mathscr T}}
\newcommand{\gt}{{\mathfrak t}} 
\newcommand{\cU}{{\mathscr U}}

\newcommand{\cf}{{\rm cf}}
\newcommand{\pr}{{\rm pr}}

\newcount\skewfactor
\def\mathunderaccent#1#2 {\let\theaccent#1\skewfactor#2
\mathpalette\putaccentunder}
\def\putaccentunder#1#2{\oalign{$#1#2$\crcr\hidewidth
\vbox to.2ex{\hbox{$#1\skew\skewfactor\theaccent{}$}\vss}\hidewidth}}

\newbox\noforkbox \newdimen\forklinewidth
\setbox0\hbox{$\textstyle\bigcup$}
\setbox1\hbox to \wd0{\hfil\vrule width \forklinewidth depth \dp0
                        height \ht0 \hfil}
\setbox\noforkbox\hbox{\box1\box0\relax}
\def\unionstick{\mathop{\copy\noforkbox}\limits}
\def\nonfork#1#2_#3{#1\unionstick_{\textstyle #3}#2}
\def\nonforkin#1#2_#3^#4{#1\unionstick_{\textstyle #3}^{\textstyle #4}#2}     
\setbox0\hbox{$\textstyle\bigcup$}
\setbox1\hbox to \wd0{\hfil{\sl /\/}\hfil}
\setbox2\hbox to \wd0{\hfil\vrule height \ht0 depth \dp0 width
                                \forklinewidth\hfil}
\newbox\doesforkbox
\setbox\doesforkbox\hbox{\box1\box0\relax}
\def\nunionstick{\mathop{\copy\doesforkbox}\limits}

\def\fork#1#2_#3{#1\nunionstick_{\textstyle #3}#2}
\def\forkin#1#2_#3^#4{#1\nunionstick_{\textstyle #3}^{\textstyle #4}#2}

\newenvironment{PROOF}[2][\proofname.]
   {\begin{proof}[#1]\renewcommand{\qedsymbol}{\textsquare$_{\rm #2}$}}
   {\end{proof}}

\usepackage[hidelinks]{hyperref}

\begin{document}
\makeatletter\def\shfiuwefootnote{\gdef\@thefnmark{}\@footnotetext}\makeatother\shfiuwefootnote{Version 2023-04-10. See \url{https://shelah.logic.at/papers/E108/} for possible updates.}

\title {Stable frames and weights \\
 E108}
\author {Saharon Shelah}
\address{Einstein Institute of Mathematics\\
Edmond J. Safra Campus, Givat Ram\\
The Hebrew University of Jerusalem\\
Jerusalem, 91904, Israel\\
 and \\
 Department of Mathematics\\
 Hill Center - Busch Campus \\ 
 Rutgers, The State University of New Jersey \\
 110 Frelinghuysen Road \\
 Piscataway, NJ 08854-8019 USA}
\email{shelah@math.huji.ac.il}
\urladdr{http://shelah.logic.at}
\thanks{The author thanks Alice Leonhardt for the beautiful typing.
The author would like to thank the Israel Science Foundation for
partial support of this research (Grant No. 242/03). 
First Typed - 2002/Sept/10.
Was paper [839] till winter 203 when by an editor request
it was divided to [839], [1248], [1249]}





\date{December 23, 2015}

\begin{abstract}

\noindent
Was paper 839 in the author-s  
list till winter 2023 when it was divided 
to three.

\underline{Part I}:  We would like to generalize imaginary elements, 
weight of $\ortp(a,M,N),{\mathbf P}$-weight, ${\mathbf P}$-simple types, etc. from
\cite[Ch.III,V,\S4]{Sh:c} to the context of good frames.  This requires
allowing the vocabulary to have predicates and function symbols of
infinite arity, but it seemed that we do not suffer any real loss.

\noindent
\underline{Part II}:  Good frames were suggested in \cite{Sh:h} as 
the (bare bones) right parallel among a.e.c. to superstable (among
elementary classes).  Here we consider $(\mu,\lambda,\kappa)$-frames as
candidates for being the right parallel to 
the class of $|T|^+$-saturated models of a stable theory (among
elementary classes).  
A loss as compared to the superstable case is that going up
by induction on cardinals is problematic (for cardinals of small
cofinality).  But this arises only when we try to lift.  For this
context we investigate the dimension.

\noindent
\underline{Part III}:  In the context of Part II, we consider the main
gap problem for the parallel of somewhat saturated model; showing we
are not worse than in the first order case.
\end{abstract}

\maketitle
\numberwithin{equation}{section}
\setcounter{section}{-1}
\newpage

\centerline {Anotated Content}
\bigskip

\noindent
Part I, pg.\pageref{part1}
\bigskip

\noindent
\S0 \quad Introduction, pg.\pageref{0}
\bigskip

\noindent
\S1 \quad Weight and $\mathbf P$-weight, pg.5 (labels w(dot),wp(dot) and
without dot), pg.\pageref{1}
\mn
\begin{enumerate}
\item[${{}}$]  [For $\gs$ a good $\lambda$-frame with some additional
  properties we define placed and $\mathbf P$-weight.]
\end{enumerate}
\bigskip

\noindent
\S2 \quad Imaginary elements, an $\ess-(\mu,\lambda)$-a.e.c. and
frames, (labels $m(dot,e(dot),b)$, pg.\pageref{2}
\mn
\begin{enumerate}
\item[${{}}$]  [Define an $\ess-(\mu,\lambda)$-a.e.c. allowing
infinitary functions.  Then get $\gs$ with type bases.]
\end{enumerate}
\bigskip

\noindent
\S3 \quad $\mathbf P$-simple types, pg.18 (label \ref{c2}, on (dot)),
pg. \pageref{3}
\bigskip

\noindent
Part II \quad Generalizing stable classes, pg.\pageref{part2}
\bigskip

\noindent
\S4 \quad Introduction, pg.\pageref{4} 
\bigskip

\noindent
\S5 \quad Axiomatize a.e.c. without full continuity, (label f), pg.\pageref{5}
\mn
\begin{enumerate}
\item[${{}}$]  [Smooth out: generalize \cite[\S1]{Sh:600}.]
\end{enumerate}

\S(5A) \quad a.e.c., pg.\pageref{5A}

\S(5B) \quad Basic Notions, (label \ref{f31} on), pg.\pageref{5B}

\S(5C) \quad Liftings, (labels \ref{f46} on), pg.\pageref{5C}
\bigskip

\noindent
\S6 \quad PR frames, (labels pr(dot)), pg.\pageref{6}
\mn
\begin{enumerate}
\item[${{}}$]  [Seems better with NF, here, so earlier; 
\sn
\begin{enumerate}
\item[$(a)$]  dominated appear
\sn
\item[$(b)$]  missing reference
\sn
\item[$(c)$]  ``$P$ based on $\mathbf a$", see I, but by
\sn
\item[$(d)$]  use places $K_{(M,A)}$ or monsters $\gC$.]
\end{enumerate}
\end{enumerate}
\bigskip

\noindent
Part III, pg.\pageref{part3}
\bigskip

\noindent
\S7 \quad Introduction, (7.1-7.4 give labels!), pg.\pageref{7}
\bigskip

\noindent
\S8 \quad Analysis of dimension for $\mathbf P$, (label g(dot)), pg.\pageref{8}
\mn
\begin{enumerate}
\item[${{}}$]  [Question: use NF?  Place $\mathbf a$ or $\ga$?
\sn
\item[${{}}$]  $(a) \quad K_A$ or $K_{(M,A)}$ or $K_{M,\infty}$
\sn
\item[${{}}$]  $(b) \quad$ use monster $\gC$ or play...
\sn
\item[${{}}$]  $(c) \quad$ define $M <_{\gk_A} N$
\sn
\item[${{}}$]  $(d) \quad$ m.d. candidate (multi-dimensional)
\sn
\item[${{}}$]  $(e) \quad (< \kappa)$-based.]
\end{enumerate}
\bigskip

\noindent
Part IV, pg.\pageref{part4}
\bigskip

\noindent
\S9 \quad Strong stability: weak form of superstability [continue
\cite{Sh:863}, (label uq(dot)), pg.\pageref{9}
\bigskip

\noindent
\S10 \quad Decomposition (more on main gap + decomposition, continue
\S7,\S8), (label h(dot)), pg.\pageref{10}
\bigskip

\noindent
\S11 \quad Decompositions, (label dc(dot)), pg.\pageref{11}
\bigskip

\newpage

\centerline {Part I: Beautiful frames: weight and
  simplicity} \label{part1}

\section {Introduction} \label{0} 

We consider here the directions listed in the abstract\footnote{As
we have started this in 2002 and have not worked on it for long, we
intend to make public what is in reasonable state.}

In part I we assume ${\gs}$ is a good $\lambda$-frame with some
extra properties from \cite{Sh:705}, e.g. as in the assumption of
\cite[\S12]{Sh:705}, so we shall assume knowledge of \cite{Sh:705}, and
the basic facts on good $\lambda$-frames from \cite{Sh:600}. 

We can look at results from \cite{Sh:c} which were 
not regained in beautiful $\lambda$-frames.  Well, of
course, we are far from the main gap for the original ${\gs}$ 
(\cite[Ch.XIII]{Sh:c}) and there are
results which are obviously more strongly connected to elementary
classes, particularly ultraproducts.  This leaves us with parts of
type theory: semi-regular types, weight, ${\mathbf P}$-simple
\footnote{The motivation is for suitable $\mathbf P$ (e.g. a single
regular type)  that on the one hand stp$(a,A) \pm
\mathbf P \Rightarrow \stp(a/E,A)$ is $\mathbf P$-simple for some
equivalence relation definable over $A$ 
and on the other hand if $\stp(a_i,A)$ is $\mathbf P$-simple for
 $i < \alpha$ then $\Sigma\{w(a_i,A) \cup \{a_j:j<i\}):i < \alpha\}$
 does not depend on the order in which we list the $a_i$'s.  Note that
 $\mathbf P$ here is ${\cP}$ there.} types, ``hereditarily orthogonal
to $\mathbf P$" (the last two were defined and investigated in 
\cite[Ch.V,\S0 + Def4.4-Ex4.15]{Sh:a},
\cite[Ch.V,\S0,pg.226,Def4.4-Ex4.15,pg.277-284]{Sh:c}). 

Note that ``a type $q$ is $p$-simple (or ${\mathbf P}$-simple)" and ``$q$
is hereditarily orthogonal to $p$ (or ${\mathbf P}$)" are essentially
the \footnote{Note, ``foreign to $\mathbf P$" and ``hereditarily
orthogonal to $\mathbf P$" are equivalent.  Now ($\mathbf P = \{p\}$ for
ease)
\mn
\begin{enumerate}
\item[$(a)$]   $q(x)$ is $p(x)$-simple when for some set $A$, in
${\gC}$ we have $q({\gC}) \subseteq \acl(A \cup \bigcup p_i({\gC}))$
\sn
\item[$(b)$]   $q(x)$ is $p(x)$-internal when for some set $A$, in
${\gC}$ we have $q({\gC}) \subseteq \dcl(A \cup p({\gC}))$.
\end{enumerate}
\mn
Note
\mn
\begin{enumerate}
\item[$(\alpha)$]   internal implies simple
\sn
\item[$(\beta)$]  if we aim at computing weights it is better to
stress acl as it covers more
\sn 
\item[$(\gamma)$]   but the difference is minor and
\sn
\item[$(\delta)$]   in existence it is better to stress dcl, also it
is useful that $\{F \restriction (p({\gC}) \cup q({\gC}):F$ an
automorphism of ${\gC}$ over $p({\gC}) \cup \Dom(p)\}$
is trivial when $q(x)$ is $p$-internal but not so for $p$-simple
(though form a pro-finite group).
\end{enumerate}} 
``internal" and ``foreign" in Hrushovshi's profound works.
\newpage

\section {I, Weight and ${\mathbf P}$-weight} \label{1}

Recalling \cite{Sh:600}, \cite{Sh:705}
\begin{context}
\label{a2}
1) ${\gs}$ is a full good$^+ \,\lambda$-frame,
with primes, $K^{3,\vq}_{\gs} = K^{3,\qr}_{\gs},\bot = 
{\underset \wk \bot}$ and $p \bot M \Leftrightarrow
p {\underset \aaa \perp} M$, note that as ${\gs}$ is full,
${\cS}^{\bs}_{\gs}(M) = {\cS}^{\na}_{\gs}(M)$; also $\gk_{\gs} =
\gk[\gs] = (K^{\gs},\le_{\gk_{\gs}})$ is the a.e.c.

\noindent
2) ${\gC}$ is an ${\gs}$-monster so it is 
$K^{\gs}_{\lambda^+}$-saturated over $\lambda$ and $M <_{\gs} {\gC}$
means $M \le_{{\gk}[{\gs}]} {\gC}$ and $M \in K_{\gs}$.  
As $\gs$ is full, it has regulars.
\end{context}

\begin{observation}
\label{a5}
${\gs}^{\reg}$ satisfies all the above except being full.
\end{observation}

\begin{proof}  See \cite[10.18=L10.p19tex]{Sh:705} and Definition
\cite[10.17=L10.p18tex]{Sh:705}. 
\end{proof}

\begin{claim}
\label{a8}
1) If $p \in {\cS}^{\bs}_{\mathfrak s}(M)$ 
\then \, we can find $b,N$ and a finite $\mathbf J$ such that:
\mn
\begin{enumerate}
\item[$\circledast$]  $(a) \quad M \le_{\gs} N$
\sn
\item[${{}}$]   $(b) \quad \mathbf J \subseteq N$ 
is a finite independent set in $(M,N)$
\sn
\item[${{}}$]   $(c) \quad c \in \mathbf J \Rightarrow \ortp(c,M,N)$ is
  regular, recalling $\ortp$ stands for orbital type
\sn
\item[${{}}$]   $(d) \quad (M,N,\mathbf J) \in K^{3,\qr}_{\gs}$
\sn
\item[${{}}$]   $(e) \quad b \in N$ realizes $p$.
\end{enumerate}
\mn
2) We can add, if $M$ is brimmed, that
\mn
\begin{enumerate}
\item[$(f)$]   $(M,N,b) \in K^{3,\pr}_{\gs}$.
\end{enumerate}
\mn
3) In (2), $|\mathbf J|$ depends only on $(p,M)$.

\noindent
4) If $M$ is brimmed, \then\, we can work in $\gs(\brim)$ and get the
same $\|\mathbf J\|$ and $N$ (so $N \in K_{\gs}$) brimmed.
\end{claim}

\begin{PROOF}{\ref{a8}}  
1) By induction on $\ell < \omega$, we try to choose 
$N_\ell,a_\ell,q_\ell$ such that:
\mn
\begin{enumerate}
\item[$(*)$]   $(a) \quad N_0 = M$
\sn
\item[${{}}$]   $(b) \quad N_\ell \le_{\gs} N_{\ell +1}$
\sn
\item[${{}}$]  $(c) \quad q_\ell \in {\cS}_{\gs}(N_\ell)$, so possibly
  $q_\ell \notin \cS^{\na}_{\gs}(N_\ell)$
\sn
\item[${{}}$]   $(d) \quad q_0 = p$
\sn
\item[${{}}$]   $(e) \quad q_{\ell +1} \restriction N_\ell = q_\ell$
\sn
\item[${{}}$]   $(f) \quad q_{\ell +1}$ forks over $N_\ell$ so now necessarily
$q_\ell \notin {\cS}^{\na}_{\gs}(N_\ell)$
\sn
\item[${{}}$]   $(g) \quad (N_\ell,N_{\ell +1},a_\ell) \in
K^{3,\pr}_{\gs}$
\sn
\item[${{}}$]   $(h) \quad r_\ell = 
\ortp(a_\ell,N_\ell,N_{\ell +1})$ is regular
\sn
\item[${{}}$]   $(i) \quad r_\ell$  either is $\perp M$
or does not fork over $M$. 
\end{enumerate}
\mn
If we succeed to carry the induction for all $\ell < \omega$ let $N =
\cup\{N_\ell:\ell < \omega\}$; as this is a countable chain,
there is $q \in {\cS}(N)$ such that 
$\ell < \omega \Rightarrow q \restriction N_\ell = q$ and as $q$ is
not algebraic (because each $q_n$ is not), and ${\gs}$ is full, clearly $q \in
{\cS}_{\gs}(N)$; but $q$ contradicts 
the finite character of non-forking.  So for
some $n \ge 0$ we are stuck, but this cannot occur if $q_n \in 
{\cS}^{\na}_{\gs}(N_n)$.
[Why?  By \ref{a5}, equivalently ${\gs}^{\reg}$ has enough regulars
and then we can apply \cite[8.3=L6.1tex]{Sh:705}.]
So for some $b \in N_n$ we have $q_n = \ortp(b,N_n,N_n)$, 
i.e., $b$ realizes $q_n$ hence it realizes $p$.

Let $\mathbf J = \{a_\ell:\ortp(a_\ell,N_\ell,N_{\ell +1})$ does not
fork over $N_0\}$.  By \cite[6.2]{Sh:705} we have 
$(M,N_n,\mathbf J) = (N_0,N_n,\mathbf J) \in K^{3,\vq}_{\gs}$
hence $\in K^{3,\qr}_{\gs}$ by \cite{Sh:705} so we are done.

\noindent 
2) Let $N,b,\mathbf J$ be as in part (1) with $|\mathbf J|$ minimal.  We
can find $N' \le_{\gs} N$ such that $(M,N',b) \in K^{3,\pr}_{\gs}$ and we 
can find $\mathbf J'$ such that $\mathbf J' \subseteq N'$ is 
independent regular in $(M,N')$ and maximal under those demands.  
Then we can find $N'' \le_{\gs} N'$ such that $(M,N'',\mathbf J') \in
K^{3,\qr}_{\gs}$.  If $\ortp_{\gs}(b,N'',N') \in {\cS}^{\na}_{\gs}(N'')$ 
is not orthogonal to
$M$ we can contradict the maximality of $\mathbf J'$ in $N'$ as in the
proof of part (1), 
so $\ortp_{\gs}(b,N'',N') \perp M$ (or $\notin {\cS}^{\na}_{\gs}(N)$).
Also \wilog \, $(N'',N',b) \in
K^{3,\pr}_{\gs}$, so by \cite{Sh:705} we 
have $(M,N',\mathbf J') \in K^{3,\qr}_{\gs}$.
Hence there is an isomorphism $f$ from $N'$ onto $N''$ which is the
identity of $M \cup \mathbf J'$ (by the uniqueness for
$K^{3,\qr}_{\gs}$).  So using $(N',f(b),\mathbf J')$ for
$(N,b,\mathbf J)$ we are done.

\noindent
3) If not, we can find $N_1,N_2,\mathbf J_1,\mathbf J_2,b$ such that $M
\le_{\gs} N_\ell \le_{\gs} N$ and the quadruple $(M,N_\ell,\mathbf J_\ell,b)$ is
as in (a)-(e)+(f) of part (1)+(2) for $\ell =1,2$.  Assume toward
contradiction that $|\mathbf J_1| \ne |\mathbf J_2|$ so \wilog \, $|\mathbf
J_1| < |\mathbf J_2|$. 

By ``$(M,N_\ell,b) \in K^{3,\pr}_{\gs}$" \wilog \, $N_2 \le_{\gs} N_1$.

By \cite[10.15=L10b.11tex(3)]{Sh:705} for some 
$c \in \mathbf J_2 \backslash \mathbf J_1,\mathbf J_1 \cup
\{c\}$ is independent in $(M,N_1)$, contradiction to $(M,N,\mathbf J_1)
\in K^{3,\vq}_{\gs}$ by \cite[10.15=L10b.11tex(4)]{Sh:705}.  

\noindent
4) Similarly.
\end{PROOF}

\begin{definition}
\label{a11}
1) For $p \in {\cS}^{\bs}_{\gs}(M)$, let the weight of $p,w(p)$
be the unique natural number such that: if $M \le_{\gs} M',M'$ is 
brimmed, $p' \in {\cS}^{\bs}_{\gs}(M')$ is a non-forking extension of $p$
\then \, it is the unique $|\mathbf J|$ from Claim \ref{a8}(3), it is 
a natural number. 

\noindent
2) Let $w_{\gs}(a,M,N) = w(\ortp_{\gs}(a,M,N))$.
\end{definition}

\begin{claim}
\label{w.3}
1) If $p \in {\cS}^{bs}(M)$ regular, \then \, $w(p) =1$. 

\noindent
2) If $\mathbf J$ is independent in $(M,N)$ and $c \in N$, \then \, for
some $\mathbf J' \subseteq \mathbf J$ with $\le w_{\gs}(c,M,N)$
elements, $\{c\} \cup (\mathbf J \backslash \mathbf J')$ is independent in
$(M,N)$.
\end{claim}

\begin{PROOF}{\ref{w.3}}
Easy by now.
\end{PROOF}

\noindent
Note that the use of $\gC$ in Definition \ref{a17} is for
transparency only and can be avoided, see \ref{a26} below.
\begin{definition}
\label{a17}
1) We say that ${\mathbf P}$ is an
$M^*$-based family (inside ${\gC}$) \when \,:
\mn
\begin{enumerate}
\item[$(a)$]   $M^* <_{\gk[\gs]} {\gC}$ and $M^* \in K_{\gs}$
\sn
\item[$(b)$]   ${\mathbf P} \subseteq \cup\{\cS^{\bs}_{\gs}(M):
M \le_{\gk[\gs]} {\gC}$ and $M \in K_{\gs}\}$
\sn
\item[$(c)$]   ${\mathbf P}$ is preserved by automorphisms of ${\gC}$
  over $M^*$.
\end{enumerate}
\mn
2) Let $p \in {\cS}^{\bs}_{\gs}(M)$ where $M \le_{\gk[\gs]}{\gC}$
\mn
\begin{enumerate}
\item[$(a)$]   we say that $p$ is hereditarily orthogonal to 
${\mathbf P}$ (or ${\mathbf P}$-foreign) \when \,: 

if $M \le_{\gs} N \le_{{\gk}[{\gs}]} {\gC},q \in {\cS}^{\bs}_{\gs}(N),
q \restriction M=p$, \then \, $q$ is orthogonal to ${\mathbf P}$
\sn
\item[$(b)$]   we say that $p$ is ${\mathbf P}$-regular \when \, $p$ is
regular, not orthogonal to ${\mathbf P}$ and if $q \in 
{\cS}^{\bs}_{\gs}(M'),M \le_{\gs} M' <_{{\gk}[{\gs}]} {\gC}$ and $q$
is a forking extension of $p$ \then \, $q$ is hereditarily orthogonal 
to ${\mathbf P}$
\sn
\item[$(c)$]   $p$ is weakly ${\mathbf P}$-regular \If \, it is regular
and is not orthogonal to some ${\mathbf P}$-regular $p'$.
\end{enumerate}
\mn
3) ${\mathbf P}$ is normal \when \, ${\mathbf P}$ is a set of regular types
and each of them is ${\mathbf P}$-regular. 

\noindent
4) For $q \in {\cS}^{\bs}_{\gs}(M),M <_{{\gk}[{\gs}]}
{\gC}$ let $w_{\mathbf P}(q)$ be defined as the natural number satisfying
the following
\mn
\begin{enumerate}
\item[$\circledast$]  if $M \le_{\gs} M_1 \le_{\gs}
M_2 \le_{\gs} {\gC},M_\ell$ is $(\lambda,*)$-brimmed, 
$b \in M_2,\ortp_{\gs}(b,M_1,M_2)$ is a non-forking extension 
of $q,(M_1,M_2,b) \in K^{3,\pr}_{\gs},(M_1,M_2,\mathbf J) \in 
K^{3,\qr}_{\gs}$ and $\mathbf J$ is regular in 
$(M_1,M_2)$, i.e. independent and $c \in \mathbf J \Rightarrow \ortp_{\gs}
(c,M_1,M_2)$ is a regular type \then \, 
$w_{\mathbf P}(q) = |\{c \in \mathbf J:\ortp_{\gs}(c,M_1,M)$ is weakly
${\mathbf P}$-regular$\}|$.
\end{enumerate}
\mn
5) We replace ${\mathbf P}$ by $p$ if ${\mathbf P} = \{p\}$, where $p \in
{\cS}^{\bs}(M^*)$ is regular (see \ref{a20}(1)).
\end{definition}

\begin{claim}
\label{a20}
1) If $p \in {\cS}^{\bs}_{\gs}(M)$ is regular \then \,
$\{p\}$ is an $M$-based family and is normal. 

\noindent
2) Assume $\mathbf P$ is an $M^*$-based family.
If $q \in {\cS}^{\bs}_{\gs}(M)$ and $M^* \le_{\gs} 
M \le_{{\gk}[{\gs}]}{\gC}$ \then \, $w_{\mathbf P}(q)$ is well defined (and is 
a natural number). 

\noindent
3) In Definition \ref{a17}(4) we can find $\mathbf J$ such that for
   every $c \in \mathbf J_1$ we have: $\ortp(c,M_1,M)$ is 
weakly-$\mathbf P$-regular
$\Rightarrow \ortp(c,M_1,M)$ is $\mathbf P$-regular.
\end{claim}

\begin{PROOF}{\ref{a20}}
 Should be clear. 
\end{PROOF}

\begin{discussion}
\label{a23}  It is tempting to try to generalize the notion of
${\mathbf P}$-simple (${\mathbf P}$-internal in Hrushovski's terminology)
and semi-regular.  An important property of those notions in the first
order case is that: e.g. if $(\bar a/A) \pm p,p$ regular then for some
equivalence relation $E$ definable over $A,\ortp(\bar a/E,A)$ is $\pm
p$ and is $\{p\}$-simple.
So assume that $p,q \in {\cS}^{\bs}_{\gs}(M)$ are not
orthogonal, and we can define an equivalence relation ${\cE}^{p,q}_M$ on
$\{c \in {\gC}:c$ realizes $p\}$, defined by

\begin{equation*}
\begin{array}{clcr}
c_1 {\cE}^{p,q}_M c_2 &\text{ \Iff \,\,for every } d \in {\gC} 
\text{ realizing } q \text{ we have} \\
  &\ortp_{\gs}(c_1 d,M,{\gC}) = \ortp_{\gs} (c_2d,M,{\gC}).
\end{array}
\end{equation*}

\mn
This may fail (the desired property) even in the first order case:  
suppose $p,q$ are definable over $a^* \in M$ (on
getting this, see later) and we have $\langle c_\ell:\ell \le n \rangle,\langle
M_\ell:\ell < n \rangle$ such that 
$\ortp(c_\ell,M_\ell,{\gC}) = p_\ell$ each $p_\ell$
is parallel to $p,c_\ell {\cE}^{p,q}_{M_\ell} c_{\ell +1}$ but 
$c_0,c_n$  realizes $p,q$ respectively and $\{c_0,c_n\}$ is 
independent over $M_0$.  Such a situation
defeats the attempt to define a ${\mathbf P}-\{q\}$-simple type $p/{\cE}$ as
in \cite[Ch.V]{Sh:c}.
 
In first order logic we can find a saturated $N$ and  $a^* \in N$ such that
$\ortp(M,\bigcup\limits_{\ell} M_\ell \cup\{c_0,\dotsc,c_n\})$ does not fork
over $a^*$ and use ``average on the type with an ultrafilter
$c$ over $q({\gC}) + a^*_t$" (for suitable $a^*_t$'s).  \underline{But} 
see below.
\end{discussion}

\begin{discussion}:  
\label{w20}
1) Assume ($\gs$ is full and) every $p \in {\cS}^{\na}_{\gs}(M)$ is 
representable by some $a_p \in M$ (in \cite{Sh:c}, e.g. the canonical
base $\Ceb(p)$).  
We can define for $\bar a,\bar b \in {}^{\omega >} {\gC}$ 
when $\ortp(\bar a,\bar b,{\gC})$ is stationary (and/or non-forking).  We 
should check the basic properties. See \S3.

\noindent
2) Assume $p \in {\cS}^{\bs}_{\gs}(M)$ is regular, 
definable over $\bar a^*$ (in the natural sense).  
We may wonder if the niceness of the 
dependence relation hold  for $p \restriction \bar a^*$?
\end{discussion}

If you feel that the use of a monster model is not natural in our
context, how do we ``translate" a set of types in ${\gC}^{\eq}$ 
preserved by every automorphism of ${\gC}$ which
is the identity on $A$? by using a ``place" defined by:
\begin{definition}
\label{a26}
1) A local place is a pair $\mathbf a = (M,A)$ such that $A \subseteq M
\in K_{\gs}$ (compare with \ref{j5}. 

\noindent
2) The places $(M_1,A_1),(M_2,A_2)$ are equivalent if $A_1 = A_2$ and
there are $n$ and $N_\ell \in K_{\gs}$ for $\ell \le n$
satisfying $A \subseteq N_\ell$
for $\ell = 0,\ldots,n$ such that $M_1 = N_0,M_2 = N_n$ and for each 
$\ell < n,N_\ell \le_{\gs}
N_{\ell +1}$ or $N_{\ell +1} \le_{\gs} N_\ell$.  We write
$(M_1,A_1) \sim (M_2,A_2)$ or $M_1 \sim_{A_1} M_2$.

\noindent
3) For a local place $\mathbf a = (M,A)$ let $K_{\mathbf a} = K_{(M,A)} =
\{N:(N,A) \sim (M,A)\}$; so in $(M,A)/ \sim$ we fix both $A$ as a set
and the type it realizes in $M$ over $\emptyset$.

\noindent
4) We call such class $K_{\mathbf a}$ a place.

\noindent
5) We say that ${\mathbf P}$ is an invariant set\footnote{Really a class}
of types in a place $K_{(M,A)}$ \when \,:
\mn
\begin{enumerate}
\item[$(a)$]   ${\mathbf P} \subseteq \{{\cS}^{\bs}_{\gs}(N):N \sim_A M\}$
\sn
\item[$(b)$]   membership in ${\mathbf P}$ is preserved by isomorphism
over $A$
\sn
\item[$(c)$]   if $N_1 \le_{\gs} N_2$ are both in $K_{(M,A)},p_2 \in
{\cS}^{\bs}_{\gs}(N_2)$ does not fork over $N_1$ then $p_2 \in {\mathbf P}
\Leftrightarrow p_2 \restriction N_1 \in {\mathbf P}$. 
\end{enumerate}
\mn
6) We say $M \in K_{\gs}$ is brimmed over $A$ \when \, for some $N$ we
have $A \subseteq N \le_{\gs} M$ and $M$ is brimmed over $N$. 
\end{definition}

\begin{cd}
\label{a32}
1) If $A \subseteq M \in K_{\gs}$ and $\mathbf P_0 \subseteq
   \cS^{\bs}_{\gs}(M)$ \then \, there is at most one invariant set
   $\mathbf P'$ of types in the place $K_{(M,A)}$ such that $\mathbf P^+
   \cap \cS^{\bs}_{\gs}(M) = \mathbf P_0$ and $M \le_{\gs} N \wedge p
   \in \mathbf P^+ \cap \cS^{\bs}_{\gs}(N) \Rightarrow$ ($p$ does not fork
   over $M$).

\noindent
2) If in addition $M$ is brimmed\footnote{$M$ is brimmed over $A$
  means that for some $M_1,A \subseteq M_1 \le_{\gs} M$ and $M$ is
  brimmed over $M_1$.}  over $A$ \then \, we can omit the
   last demand in part (1).

\noindent
3) If $\mathbf a = (M_1,A),(M_2,A) \in K_{\mathbf a}$ then $K_{(M_2,A)} =
K_{\mathbf a}$. 
\end{cd}

\begin{PROOF}{\ref{a32}}
Easy.
\end{PROOF}

\begin{definition}
\label{a35}
1) If in \ref{a32} there are such $\mathbf P$, we denote it by
   $\inv(\mathbf P_0) = \inv(\mathbf P_0,M)$.

\noindent
2) If $\mathbf P_0 = \{p\}$, then let $\inv(p) = \inv(p,M) = \inv(\{p\})$.

\noindent
3) We say $p \in \cS^{\bs}_{\gs}(M)$ does not split (or is definable)
   over $A$ \when \, $\inv(p)$ is well defined.
\end{definition}
\newpage

\section {I Imaginary elements, an essential-$(\mu,\lambda)$-a.e.c. and
frames} \label{2}

\subsection {Essentially a.e.c.}\
\bigskip

We consider revising the definition of a.e.c. ${\gk}$, by allowing
function symbols in $\tau_{\gk}$ with infinite number of places
while retaining local characters, e.g., if $M_n \le M_{n+1},M =
\cup\{M_n:n < \omega\}$ is uniquely determined.  Before this we
introduce the relevant equivalence relations.  In this context, we can 
give name to equivalence classes for equivalence relations on
infinite sequences.

\begin{definition}
\label{b2}
We say that ${\gk}$ is an essentially $[\lambda,\mu)$-a.e.c. or
  ess-$[\lambda,\mu)$-a.e.c. and we may write $(\mu,\lambda)$ instead
  of $[\lambda,\mu)$ \If \, 
($\lambda < \mu$ and) it is an object consisting of:
\mn
\begin{enumerate}
\item[$I(a)$]  a vocabulary $\tau = \tau_{\gk}$ which has
predicates and function symbols of possibly infinite arity but $\le \lambda$
\sn
\item[$(b)$]   a class of $K = K_{\gk}$ of $\tau$-models
\sn
\item[$(c)$]   a two-place relation $\le_{\gk}$ on $K$ 
\end{enumerate}
\mn
such that
\mn
\begin{enumerate}
\item[$II(a)$]   if $M_1 \cong M_2$ \then \, $M_1 \in K
\Leftrightarrow M_2 \in K$
\sn
\item[$(b)$]   if $(N_1,M_1) \cong (N_2,M_2)$ \then \, $M_1
\le_{\gk} N_1 \Leftrightarrow M_2 \le_{\gk} N_2$
\sn
\item[$(c)$]   every $M \in K$ has cardinality $\in [\lambda,\mu)$
\sn
\item[$(d)$]   $\le_{\gk}$ is a partial order on $K$
\sn
\item[$III_1$]   if $\langle M_i:i < \delta \rangle$ is 
$\le_{\gk}$-increasing and $|\bigcup\limits_{i < \delta} M_i| < 
\mu$ \then \, there is a unique $M \in K$ such that $|M| = 
\cup\{|M_i|:i < \delta\}$ and $i < \delta
\Rightarrow M_i \le_{\gk} M$
\sn
\item[$III_2$]   if in addition $i < \delta \Rightarrow M_i
\le_{\gk} N$ \then \, $M \le_{\gk} N$
\sn
\item[$IV$]   if $M_1 \subseteq M_2$ and $M_\ell \le_{\gk} N$
for $\ell =1,2$ \then \, $M_1 \le_{\gk} M_2$
\sn
\item[$V$]   if $A \subseteq N \in K$, \then \, there is $M$
satisfying $A \subseteq M \le_{\gk} N$ and $\|M\| \le \lambda +|A|$
(here it is enough to restrict ourselves to the case $|A| \le \lambda$).
\end{enumerate}
\end{definition}

\begin{definition}
\label{b5}
1) We say ${\gk}$ is an ess-$\lambda$-a.e.c. \If \, it is an
ess-$[\lambda,\lambda^+)$-a.e.c. 

\noindent
2) We say ${\gk}$ is an ess-a.e.c. \If \, there is $\lambda$ such
that it is an ess-$[\lambda,\infty)$-a.e.c., so $\lambda = 
\LST({\gk})$. 

\noindent
3) If $\gk$ is an ess-$[\lambda,\mu)$-a.e.c. and $\lambda \le
   \lambda_1 < \mu_1 \le \mu$ then let $K^{\gk}_{\lambda_1} =
   (K_{\gk})_{\lambda_1} = \{M \in K_{\gk}:\|M\| =
   \lambda_1\},K^{\gk}_{\lambda_1,\mu_1} = \{M \in K_{\gk}:\lambda_1
   \le \|M\| < \mu_1\}$.

\noindent
4) We define $\Upsilon^{\oor}_{\gk}$ as in
   \cite[0.8=L11.1.3A(2)]{Sh:734}.

\noindent
5) We may omit the ``essentially" \when \, $\arity(\tau_{\gk}) =
\aleph_0$ where $\arity(\gk) = \arity(\tau_{\gk})$ and for vocabulary
   $\tau,\arity(\tau) = \min\{\kappa$: every predicate and function
symbol have $\arity < \kappa\}$.
\end{definition}

\noindent
We now consider the claims on ess-a.e.c. 
\begin{claim}
\label{b8}
Let ${\gk}$ be an ess-$[\lambda,\mu)$-a.e.c.  

\noindent
1) The parallel of $\Ax(III)_1,(III)_2$ holds with
a directed family $\langle M_t:t \in I \rangle$.

\noindent
2) If $M \in K$ we can find $\langle M_{\bar a}:\bar a \in
{}^{\omega >}M \rangle$ such that:
\mn
\begin{enumerate}
\item[$(a)$]   $\bar a \subseteq M_{\bar a} \le_{\gk} M$
\sn
\item[$(b)$]   $\|M_{\bar a}\| = \lambda$
\sn
\item[$(c)$]   if $\bar b$ is a permutation of $\bar a$ then
$M_{\bar a} = M_{\bar b}$
\sn
\item[$(d)$]   if $\bar a$ is a subsequence of $\bar b$ then
$M_{\bar a} \le_{\gk} M_{\bar b}$.
\end{enumerate}
\mn
3) If $N \le_{\gk} M$ we can add in (2) that $\bar a \in
{}^{\omega >} N \Rightarrow M_{\bar a} \subseteq N$. 

\noindent
4) If for simplicity $\lambda_* = \lambda + \sup\{\Sigma\{|R^M|:R \in
\tau_{\gk}\} + \Sigma\{|F^M|:F \in \tau_{\gk}\}:M \in K_{\gk}$ has
cardinality $\lambda\}$ \then \, $K_{\gk}$ 
and $\{(M,N):N \le_{\gk} M\}$ essentially are $\PC_{\chi,\lambda_*}$-classes 
where $\chi = |\{M/\cong:M \in K^{\gk}_\lambda\}|$, noting that
$\chi \le 2^{2^\theta}$.  
That is, $\langle M_{\bar a}:\bar a\in {}^{\omega >} A\rangle$
satisfying clauses (b),(c),(d) of
part (2) such that $A = \cup\{|M_{\bar a}|:\bar a \in {}^{\omega >}
A\}$ represent a unique $M\in K_{\gk}$ with universe $A$ and similarly for
$\le_{\gk}$, (on the Definition of $\PC_{\chi,\lambda_*}$, see 
\cite[1.4(3)]{Sh:88r}).  Note that if in $\tau_{\gk}$ there are no two
distinct symbols with the same interpretation in every $M \in K_{\gk}$
then $|\tau){k_*}| \le 2^{2^\lambda}$.

\noindent
5) The results on omitting types in \cite{Sh:394} or
   \cite[0.9=L0n.8,0.2=0n.11]{Sh:734} hold, i.e., if
$\alpha < (2^{\lambda_*})^+ \Rightarrow
   K^{\gk}_{\beth_\alpha} \ne \emptyset$ \then \, $\theta \in
[\lambda,\mu) \Rightarrow K_\theta \ne \emptyset$ and there is an
$\EM$-model, i.e., $\Phi \in \Upsilon^{\oor}_{\gk}$ with $|\tau_\Phi| =
   |\tau_{\gk}| + \lambda$ and $\EM(I,\Phi)$ having cardinality
   $\lambda + |I|$ for any linear order $I$.

\noindent
6) The lemma on the equivalence of being
universal model homogeneous and of being saturated 
(see \cite[3.18=3.10]{Sh:300b} or \cite[1.14=L0.19]{Sh:600}) holds. 

\noindent
7) We can generalize the results of \cite[\S1]{Sh:600} on deriving an
ess-$(\infty,\lambda)$-a.e.c. from an ess $\lambda$-a.e.c.
\end{claim}

\begin{PROOF}{\ref{b8}}
The same proofs, on the generalization in \ref{b8}(7), see in \S5 below.
The point is that, in the term of \S5, our $\gk$ is a
$(\lambda,\mu,\kappa)$-a.e.c. (automatically with primes).
\end{PROOF}

\begin{remark}
\label{e.4d}
1) In \ref{b8}(4) we can decrease the bound on $\chi$ if we have more nice
definitions of $K_\lambda$, e.g., if $\arity(\tau) \le \kappa$ then $\chi
= 2^{(\lambda^{< \kappa}+|\tau|)}$ where $\arity(\tau) = \min\{\kappa$:
every predicate and function symbol of $\tau$ has $\arity < \kappa\}$.

\noindent
2) We may use above $|\tau_{\gs}| \le \lambda,\arity(\tau_{\gk}) 
= \aleph_0$ to get that $\{(M,\bar a)/\cong M \in K^{\gk}_\lambda,
\bar a \in {}^\lambda M$ list $M\}$
has cardinality $\le 2^\lambda$.  See also \ref{b56}.

\noindent
3) In \ref{b32} below, if we omit ``$\bbE$ is small" and $\lambda_1 =
   \sup\{|\seq(M)/\bbE_M|:M \in K^{\gk}_\lambda\}$ is $< \mu$ then
   $\gk_{[\lambda_1,\mu)}$ is an ess-$[\lambda_1,\mu)$-a.e.c.

\noindent
4) In Definition \ref{b2}, we may omit axiom V and define $\LST(\gk) \in
   [\lambda,\mu]$ naturally, and if $M \in K^{\gk}_\lambda \Rightarrow
   \mu > |\seq(M)/\bbE_M|$ \then \, in \ref{b32}(1) below we can omit
   ``$\bbE$ is small".

\noindent
5) Can we preserve in such ``transformation" the arity finiteness?  A 
natural candidate is trying to code $p \in {\cS}^{\bs}_{\gs}(M)$ 
by $\{\bar a:\bar a \in {}^{\omega >} M\}$ and
there are $M_0 \le_{\gs} M_1$ such that $M \le_{\gs} M_1$ and
$\ortp(a_\ell,M_0,M_1)$ is parallel to $p$ and $\bar a$ is independent in
$(M_0,M_1)\}$.  If e.g., $K_{\gs}$ is saturated this helps but still
we suspect it may fail. 

\noindent
6) What is the meaning of ess-$[\lambda,\mu)$-a.e.c.?
Can we look just at $\langle M_t:t \in I \rangle,I$ directed, $t \le_I
s \Rightarrow M_t \le_{\gs} M_{\gs} \in K_\lambda$?  But for
isomorphism types we take a kind of completion and so make more pairs
isomorphic but $\bigcup\limits_{t \in I} M_t$ does not determine $\bar M =
\langle M_t:t \in I \rangle$ and the completion may depend on this
representation.

\noindent
7) If we like to avoid this and this number is $\lambda'$, \then \,
 we should change the definition of $\seq(N)$ (see \ref{b14}(b)) 
to $\seq'(N) = \{\bar a:\ell g(\bar a) = \lambda$ and for some $M \le_{\gs} N$ from
$K^{\gk}_\lambda,\langle a_{1 + \alpha}:\alpha < \lambda\rangle$ list
 the members of $M$ and $a_0 \in \{\gamma:\gamma <  \mu_*\}$. 
\end{remark}
\bigskip

\subsection {Imaginary Elements and Smooth Equivalent Relations}\
\bigskip

Now we return to our aim of getting canonical base for orbital types.
\begin{definition}
\label{b14}
Let ${\gk} = (K_{\gk},\le_{\gk})$ be a $\lambda$-a.e.c. or just
ess-$[\lambda,\mu)$-a.e.c. (if ${\gk}_\lambda = {\gk}_{\gs}$ we may
write ${\gs}$ instead of ${\gk}_\lambda$, see \ref{b35}).  We say that 
$\bbE$ is a smooth ${\gk}_\lambda$-equivalence relation \when \,:
\mn
\begin{enumerate}
\item[$(a)$]   $\bbE$ is a function with domain $K_{\gk}$
mapping $M$ to $\bbE_M$
\sn
\item[$(b)$]   for $M \in K_{\gk},\bbE_M$ is an equivalence relation
  on a subset of $\seq(M) = \{\bar a:\bar a \in {}^\lambda M$ and 
$M \restriction \Rang(\bar a) \le_{\gk} M\}$ so $\bar a$ is not 
necessarily without repetitions; note that $\gk$ determines $\lambda$,
  pedantically when non-empty
\sn
\item[$(c)$]   if $M_1 \le_{\gk} M_2$ then $\bbE_{M_2}
\restriction \seq(M_1) = \bbE_{M_1}$
\sn
\item[$(d)$]   if $f$ is an isomorphism from $M_1 \in K_{\gs}$ onto 
$M_2$ \then \, $f$ maps $\bbE_{M_1}$ onto $\bbE_{M_2}$
\sn
\item[$(e)$]   if $\langle M_\alpha:\alpha \le \delta \rangle$ is
$\le_{\gs}$-increasing continuous \then \, $\{\bar a/\bbE_{M_\delta}:\bar a \in
\seq(M_\delta)\} = \{\bar a/\bbE_{M_\delta}:\bar a \in 
\bigcup\limits_{\alpha < \delta} \seq(M_\alpha)\}$. 
\end{enumerate}
\mn
2) We say that $\bbE$ is small \If \, each $\bbE_M$ has $\le \|M\|$
 equivalence classes.
\end{definition}

\begin{remark}
\label{b17}
1) Note that if we have $\langle \bbE_i:i < i^* \rangle$, each
$\bbE_i$ is a smooth ${\gk}_\lambda$-equivalence
relation and $i^* < \lambda^+$ \then \, we can find a smooth
${\gk}_\lambda$-equivalence relation $\bbE$ such that essentially the 
$\bbE_M$-equivalence classes are the $\mathbb E_i$-equivalence classes for
$i < i^*$; in detail: \wilog \, $i^* \le \lambda$ and $\bar a 
\bbE_M \bar b$ \Iff \, $\ell g(\bar a) = \ell g(\bar b)$
and
\mn
\begin{enumerate}
\item[$\circledast_1$]   $i(\bar a) = i(\bar b)$ and if $i(\bar a) < i^*$ then
$\bar a \restriction [1+i(\bar a)+1,\lambda) \bbE_{i(\bar a)}
\bar b \restriction [1+i(\bar b)+1,\lambda)$ where $i(\bar a) = 
\Min\{j:(j+1 < i(*)) \wedge a_0 \ne a_{1+j}$ or $j=\lambda\}$.
\end{enumerate}
\mn
2) In fact $i^* \le 2^\lambda$ is O.K., e.g. choose a function $\mathbf e$ 
from $\{e:e$ an equivalence relation on $\lambda$ to $i^*$ and for
$\bar a,\bar b \in \seq(M)$ we let $i(\bar a) = \mathbf e(\{(i,j):
a_{2i+1} = a_{2j+1}\}$ and
\mn
\begin{enumerate}
\item[$\circledast_2$]   $\bar a \in \bbE_M \bar b$ \Iff \,  
$i(\bar a) = i(\bar b)$ and $\langle a_{2i}:i < \lambda \rangle
\bbE_{i(\bar a)} \langle b_{2i}:i < \lambda \rangle$.
\end{enumerate}
\mn
3) We can redefine $\seq(M)$ as ${}^{\lambda \ge}M$, then have to make
minor changes above.
\end{remark}

\begin{definition}
\label{b20}
Let ${\gk}$ be a $\lambda$-a.e.c. or just ess-$[\lambda,\mu)$-a.e.c. 
and $\bbE$ a small smooth ${\gk}$-equivalence relation and the reader
may assume for simplicity that the vocabulary $\tau_{\gk}$ has only predicates.
Also assume $F_*,c_*,P_* \notin \tau_{\gk}$.  
We define $\tau_*$ and ${\gk}_* = {\gk} \langle \bbE \rangle =
(K_{\gk_*},\le_{\gk_*})$ as follows:
\mn
\begin{enumerate}
\item[$(a)$]   $\tau^* = \tau \cup \{F_*,c_*,P_*\}$ with $P_*$ a unary
predicate, $c_*$ an individual constant and $F_*$ a $\lambda$-place
function symbol
\sn
\item[$(b)$]   $K_{\gk_*}$ is the class of $\tau^*_{\gk}$-models 
$M^*$ such that for some model $M \in K_{\gk}$ we have:
\sn
\begin{enumerate}
\item[$(\alpha)$]   $|M| = P^{M^*}_*$
\sn
\item[$(\beta)$]    if $R \in \tau$ then $R^{M^*} = R^M$
\sn
\item[$(\gamma)$]   if $F \in \tau$ has arity $\alpha$ then
$F^{M^*} \restriction M = F^M$ and for any $\bar a \in {}^\alpha(M^*),\bar a
\notin {}^\alpha M$ we have $F^{M^*}(\bar a) = c^{M^*}_*$ (or allow partial
function or use $F^{M^*}(\bar a) = a_0$ when $\alpha > 0$ and
$F^{M^*}(\langle \rangle)$ when $\alpha = 0$, i.e. $F$ is an individual constant);
\sn
\item[$(\delta)$]   $F_*$ is a $\lambda$-place function symbol and:
\sn
\item[${{}}$]  $(i) \quad$ if $\bar a \in \seq(M)$ then
$F^{M^*}_*(\bar a) \in |M^*| \backslash |M| \backslash \{c^{M^*}_*\}$
\sn
\item[${{}}$]   $(ii) \quad$ if $\bar a,\bar b \in \Dom(\bbE)
  \subseteq \seq(M)$ 
then $F^{M^*}_*(\bar a) = F^{M^*}_*
(\bar b) \Leftrightarrow \bar a \bbE_M \bar b$
\sn
\item[${{}}$]   $(iii) \quad$ if $\bar a \in {}^\lambda(M^*)$ and
$\bar a \notin \Dom(\bbE) \subseteq \seq(M)$ then 
$F^{M^*}_*(\bar a) = c_*^{M^*}$
\sn
\item[$(\varepsilon)$]   $c^{M^*}_* \notin |M|$ and
if $b \in |M^*| \backslash |M| \backslash \{c^{M^*}_*\}$ 
then for some $\bar a \in \Dom(\bbE) \subseteq 
\seq(M)$ we have $F^{M^*}_*(\bar a) = b$
\end{enumerate}
\sn
\item[$(c)$]  $\le_{\gk_*}$ is the two-place relation on $K_{\gk_*}$
  defined by: $M^* \le_{\gk_*} N^*$ \If
\sn
\begin{enumerate}
\item[$(\alpha)$]   $M^* \subseteq N^*$ and
\sn
\item[$(\beta)$]   for some $M,N \in {\gk}$ as in clause (b)
we have $M \le_{\gk} N$.
\end{enumerate}
\end{enumerate}
\end{definition}

\begin{definition}
\label{b26}
1) In \ref{b20}(1) we call $M \in {\gk}$ a witness for $M^* \in
K_{\gk_*}$ if they are as in clause (b) above.

\noindent
2) We call $M \le_{\gk} N$ witness for $M^* \le_{{\gk}^*_\lambda} N^*$ if
they are as clause (c) above.
\end{definition}

\begin{discussion}
\label{b29}
Up to now we have restricted ourselves to vocabularies with each
predicate and function symbol of finite arity, and this restriction seems very
reasonable.  Moreover, it seems a priori that for a parallel to
superstable, it is quite undesirable to have infinite arity.  Still
our desire to have imaginary elements (in particular canonical basis
for types) forces us to accept them.  The price is that inthe class of
$\tau$-models the union of
increasing chains of $\tau$-models is not a well defined
$\tau$-model, more accurately we can show its existence, but not
smoothness; \underline{however} inside the class ${\gk}$ it will be.
\end{discussion}

\begin{claim}
\label{b32}
1) If ${\gk}$ is a $[\lambda,\mu)$-a.e.c. or just an 
ess-$[\lambda,\mu)$-a.e.c. and $\bbE$ a small smooth ${\gk}$-equivalence 
relation \then \, ${\gk} \langle \bbE \rangle$
is an ess-$[\lambda,\mu)$-a.e.c.  

\noindent
2) If ${\gk}$ has amalgamation and $\bbE$ is a small 
${\gk}$-equivalence class \then \, ${\gk} \langle \bbE \rangle$ has
amalgamation property.
\end{claim}

\begin{PROOF}{\ref{b32}}
The same proofs.  Left as an exercise to the reader.
\end{PROOF}
\bigskip

\subsection {Good Frames}\
\bigskip

\noindent
Now we return to good frames.
\begin{definition}
\label{b35}
1) We say that ${\gs}$ is a good ess-$[\lambda,\mu)$-frame \If \,
Definition \cite[2.1=L1.1tex]{Sh:600} is satisfied except that:
\mn
\begin{enumerate}
\item[$(a)$]   in clause (A), $\gK_{\gs} = (K_{\gs},\le_{\gs}),\gk$ is 
an ess-$[\lambda,\mu)$-a.e.c. and $\gK[\gs]$ is an
ess-$(\infty,\lambda)$-a.e.c. 
\sn
\item[$(b)$]   $K_{\gs}$ has a superlimit model in $\chi$ in every $\chi \in
  [\lambda,\mu)$
\sn
\item[$(c)$]  $K^{\gs}_\lambda/\cong$ has cardinality $\le 2^\lambda$,
  for convenience.
\end{enumerate}
\end{definition}

\begin{discussion}
\label{b40}
We may consider other relatives as our choice and mostly have similar
results.  In particular:
\mn
\begin{enumerate}
\item[$(a)$]   we can demand less: as in \cite[\S2]{Sh:842} we may
  replace $\cS^{\bs}_{\gs}$ by a formal version of $\cS^{\bs}_{\gs}$
\sn
\item[$(b)$]   we may demand goodness only for $\gs_\lambda$,
  i.e. $\gs$ restriction the class of models to $K^{\gs}_\lambda$ and
  have only the formal properties above so amalgamation and JEP are
  required only for models of cardinality $\lambda$.
\end{enumerate}
\end{discussion}

\begin{claim}
\label{b44}
All the definitions and results in \cite{Sh:600},
\cite{Sh:705} and \S1 here work for good ess-$[\lambda,\mu)$-frames.
\end{claim}

\begin{PROOF}{\ref{b44}}
No problem.  
\end{PROOF}

\begin{definition}
\label{b47}
If ${\gs}$ is a $[\lambda,\mu)$-frame or just an ess-$[\lambda,\mu)$-frame
and $\bbE$ a small smooth ${\gs}$-equivalence relation \then \, 
let ${\gt} = {\gs}\langle \bbE \rangle$ be defined by:
\mn
\begin{enumerate}
\item[$(a)$]   ${\gk}_{\gt} = {\gk}_{\gs}\langle \bbE \rangle$
\sn
\item[$(b)$]   ${\cS}^{\bs}_{\gt}(M^*) = \{\ortp_{{\gk}_{\gt}}(a,M^*,N^*):
M^* \le_{{\gk}_{\gt}} N^*$ and if $M \le_{\gk} N$ witness 
$M^*,N^* \in {\gk}_{\gt}$ then $a \in N \backslash M$ and 
$\ortp_{\gs}(a,M,N) \in {\cS}^{\bs}_{\gs}(M)\}$
\sn
\item[$(c)$]   non-forking similarly.
\end{enumerate}
\end{definition}

\begin{remark}
\label{b50}
We may add: if ${\gs}$ is\footnote{The reader may ignore this version.}
 an NF-frame we define ${\gt} = {\gs}\langle \bbE \rangle$ as an 
NF-frame similarly, see \cite{Sh:705}.
\end{remark}

\begin{claim}
\label{b53}
1) If ${\gs}$ is a good ess-$[\lambda,\mu)$-frame, $\bbE$ a small, 
smooth ${\gs}$-equivalence relation \then \, ${\gs}\langle \bbE
\rangle$ is a good ess-$[\lambda,\mu)$-frame. 

\noindent
2) In part (1) for every $\kappa,\dot I(\kappa,K^{{\gs}<\bbE>}) =
\dot I(\kappa,K^{\gs})$. 

\noindent
3) If ${\gs}$ has primes/regulars \then \, ${\gs} \langle
\bbE \rangle$ has.
\end{claim}

\begin{remark}
\label{b54}
We may add: if ${\gs}$ is an NF-frame \then \, so is ${\gs}\langle \bbE
\rangle$, hence $({\gs}\langle \bbE \rangle)^{\full}$ is a
full NF-frame; see \cite{Sh:705}. 
\end{remark}

\begin{PROOF}{\ref{b53}}
 Straightforward.  
\end{PROOF}

\noindent
Our aim is to change ${\gs}$ inessentially such that for every
$p \in {\cS}^{\bs}_{\gs}(M)$ there is a canonical base,
etc.  The following claim shows that in the context we have presented
this can be done.
\begin{claim}
\label{b56}
\underline{The imaginary elements Claim}  

Assume ${\gs}$ a good $\lambda$-frame or just a good
ess-$[\lambda,\mu)$-frame.
 
\noindent
1) If $M_* \in K_{\gs}$ and $p^* \in {\cS}^{\bs}_{\gs}(M_*)$,
 \then\footnote{note that there may well be an automorphism of $M^*$
 which maps $p^*$ to some $p^{**} \in \cS^{\bs}_{\gs}(M^*)$ such that
 $p^{**} \ne p^*$.}  \,  
there is a small, smooth $\gk_{\gs}$-equivalence relation $\bbE =
\bbE_{\gs,M_*,p^*}$ and function $\mathbf F$ such that:
\mn
\begin{enumerate}
\item[$(*)$]   if $M_* \le_{\gs} N$ and $\bar a \in \seq(N)$ so
$M =: N \restriction \Rang(\bar a) \le_{\gs} N$ and
$M \cong M_*$, \then \,
\sn
\begin{enumerate}
\item[$(\alpha)$]  $\mathbf F(N,\bar a)$ 
is well defined iff $\bar a \in \Dom(\bbE_N)$ and then $\mathbf
F(N,\bar a)$ belongs to ${\cS}^{\bs}_{\gs}(N)$
\sn
\item[$(\beta)$]  $S \subseteq \{(N,\bar a,p):N \in K_{\gs},\bar a \in
  \Dom(\bbE_N)\}$ is the minimal class such that:
\sn
\item[${{}}$]  $(i) \quad$ if $\bar a \in \seq(M_*)$ and $p$ does not
  fork over $M_* \rest \Rang(\bar a)$ then 

\hskip25pt $(M_*,\bar a,p) \in S$
\sn
\item[${{}}$]  $(ii) \quad S$ is closed under isomorphisms
\sn
\item[${{}}$]  $(iii) \quad$ if $N_1 \le_{\gs} N_2,p_2 \in
  \cS^{\bs}_{\gs}(N_2)$ does not fork over $\bar a \in \seq(N_1)$ then

\hskip25pt  $(N_2,\bar a,p_2) \in S 
\Leftrightarrow (N_1,\bar a,p_2 \rest N_1) \in S$
\sn
\item[${{}}$]  $(iv) \quad$ if $\bar a_1,\bar a_2 \in \seq(N),p \in
  \cS^{\bs}_{\gs}(N)$ does not fork over $N \rest \Rang(\bar a_\ell)$

\hskip25pt  for $\ell=1,2$ then $(N_2,\bar a_1,p) \in S \Leftrightarrow
  (N_2,\bar a_2,p) \in S$
\sn
\item[$(\gamma)$]  $\mathbf F(N,\bar a)=p$ iff $(N,\bar a,p) \in S$ hence
 if $\bar a,\bar b \in \seq(N)$ then: 
$\bar a \bbE_N \bar b$ iff $\mathbf F(\bar a,N) = \mathbf F(\bar b,N)$.
\end{enumerate}
\end{enumerate}
\mn
2) There are unique small\footnote{for small we use stability in $\lambda$}
smooth $\bbE$-equivalence relation $\bbE$ called $\bbE_{\gs}$ and
function $\mathbf F$ such that:
\mn
\begin{enumerate}
\item[$(**)(\alpha)$]   $\mathbf F(N,\bar a)$ is well defined iff $N \in
K_{\gs}$ and $\bar a \in \seq(N)$
\sn
\item[$(\beta)$]   $\mathbf F(N,\bar a)$, when defined, 
belongs to ${\cS}^{\bs}_{\gs}(N)$
\sn
\item[$(\gamma)$]   if $N \in K_{\gs}$ and $p \in
  {\cS}^{\bs}_{\gs}(N)$ \then \, there is $\bar a \in \seq(N)$ 
such that $\Rang(\bar a) = N$ and $\mathbf F(N,\bar a)=p$
\sn
\item[$(\delta)$]  if $\bar a \in \seq(M)$ and $M \le_{\gs} N$ \then
\, $\mathbf F(N,\bar a)$ is 
(well defined and is) the non-forking extension of $\mathbf F(M,\bar a)$
\sn
\item[$(\varepsilon)$]  if $\bar a_\ell \in \seq(N)$ and
$\mathbf F(N,\bar a_\ell)$ is well defined 
for $\ell=1,2$ \then \, $\bar a_1 \bbE_N \bar a_2 \Leftrightarrow 
\mathbf F(N,\bar a_1) = \mathbf F(N,\bar a_2)$
\sn
\item[$(\zeta)$]   $\mathbf F$ commute with isomorphisms.
\end{enumerate}
\mn
3) For ${\gt} = {\gs} \langle \bbE \rangle$ where $\bbE$
as in part (2) and $M^* \in K_{\gt}$ as witnessed by $M \in
K_{\gs}$ and $p^* \in {\cS}^{\bs}_{\gt}(M^*)$ is 
projected to $p \in {\cS}^{\bs}_{\gs}(M)$ let $\bas(p^*) = \bas(p) 
= \mathbf F(\bar a,M^*)/\bbE$ whenever $\mathbf F(M,\bar a) = p$.  
That is, assume $M_\ell$ witness that
$M^*_\ell \in K_{\gt}$, for $\ell=1,2$ and $(M^*_1,M^*_2,a) \in
K^{3,\bs}_{\gt}$ then $(M_1,M_2,a) \in K^{3,\bs}_{\gs}$ and 
$p^* = \ortp_{\gt}(a,M^*_1,M^*_2),p = \ortp_{\gs}(a,M_1,M_2)$; \then \, 
in ${\gt}$:
\mn
\begin{enumerate}
\item[$(\alpha)$]   if $M^*_\ell \le_{\gs} M^*,p_\ell \in 
{\cS}^{\bs}_{\gt}(M^*_\ell)$, then $p^*_1\|p^*_2 \Leftrightarrow 
\bas(p^*_1) = \bas(p^*_2)$
\sn
\item[$(\beta)$]   $p^* \in {\cS}^{\bs}_{\gt}
(M^*)$ does not split over $\bas(p^*)$, see Definition \ref{a35}(3) or
\cite[\S2 end]{Sh:705}.
\end{enumerate}
\end{claim}

\begin{PROOF}{\ref{b56}}
1) Let $M^{**} \le_{\gs} M^*$ be of cardinality
$\lambda$ such that $p^*$ does not fork over $M^{**}$.  Let $\bar a^*
= \langle a_\alpha:\alpha < \lambda\rangle$ list the element of
$M^{**}$.

We say that $p_1 \in {\cS}^{\bs}_{\gs}(M_1)$ is a weak
copy of $p^*$ when there is a witness $(M_0,M_2,p_2,f)$ which means:
\mn
\begin{enumerate}
\item[$\circledast_1$]   $(a) \quad M_0 \le_{\gs} M_2$ and $M_1 \le_{\gs} M_2$
\sn
\item[${{}}$]  $(b) \quad$ if $\|M_1\| = \lambda$ then $\|M_2\| = \lambda$
\sn
\item[${{}}$]  $(c) \quad f$ is an isomorphism from $M^{**}$ onto $M_0$
\sn
\item[${{}}$]  $(d) \quad p_2 \in {\cS}^{\bs}_{\gs}(M_2)$ is a 
non-forking extension of $p_1$
\sn
\item[${{}}$]   $(e) \quad p_2$ does not fork over $M_0$
\sn
\item[${{}}$]   $(f) \quad f(p^* \restriction M^{**})$ is $p_2
\restriction M_0$.
\end{enumerate}
\mn
For $M_1 \in K^{\gs}_\lambda,p_1 \in \cS^{\bs}_{\gs}(M_1)$
which is a weak copy of $p^*$, we say that $\bar b$ explicate its being
a weak copy \when \, for some witness $(M_0,M_2,p_2,f)$ and $\bar c$
\mn
\begin{enumerate}
\item[$\circledast_2$]  $(a) \quad \bar b = \langle b_\alpha:\alpha
< \lambda \rangle$ list the elements of $M_1$
\sn 
\item[${{}}$]   $(b) \quad \bar c = \langle c_\alpha:\alpha <
\lambda\rangle$ list the element of $M_2$
\sn
\item[${{}}$]  $(c) \quad \{\alpha:b_{2 \alpha} = b_{2 \alpha+1}\}$
code the folowing sets
\sn
\begin{enumerate}
\item[${{}}$]  $(\alpha) \quad$ the isomorphic type of $(M_2,\bar c)$
\sn
\item[${{}}$]  $(\beta) \quad \{(\alpha,\beta):b_\alpha = c_\beta\}$
\sn
\item[${{}}$]   $(\gamma) \quad \{(\alpha,\beta):f(a^*_\alpha) = c_\beta\}$
\end{enumerate}
\end{enumerate}
\mn
Now
\mn
\begin{enumerate}
\item[$\circledast_3$]   if $p \in {\cS}^{\bs}_{\gs}(M)$ is a 
weak copy of $p^*$ then for some $\bar a \in \seq(M)$, there is a 
$M_1 \le_{\gs} M$ over which $p$ does not
fork such that $\bar a$ list $M_1$ and explicate $p \restriction M_1$
is a weak copy of $p^*$
\sn
\item[$\circledast_4$]   $(a) \quad$ if $M \in K^{\gs}_\lambda$ and $\bar
b$ explicate $p^* \in {\cS}^{\bs}_{\gs}(M)$ is a weak copy of $p^*$,

\hskip25pt \then \, from $\bar b$ and $M$ we can reconstruct $p_1$
\sn
\item[${{}}$]   $(b) \quad$ call it $p_{M,\bar b}$
\sn
\item[${{}}$]  $(c) \quad$ if $M \le_{\gs} N$ let $p_{N,\bar b}$ 
be its non-forking extension in ${\cS}^{\bs}_{\gs}(N)$ we also 

\hskip25pt call it $\mathbf F(N,\bar b)$.
\end{enumerate}
\mn
Now we define $\bbE$, so for $N \in K_{\gs}$ we define a
two-place relation $\bbE_N$
\mn
\begin{enumerate}
\item[$\circledast_5$]  $(\alpha) \quad \bbE_N$ is 
on $\{\bar a$: for some $M \le_{\gs}
N$ of cardinality $\lambda$ and $p \in {\cS}^{\bs}_{\gs}(M)$

\hskip25pt  which is a copy of $p^*$, the sequence
$\bar a$ explicates $p$ being 

\hskip25pt a weak copy of $p^*\}$
\sn
\item[$(\beta)$]   $\bar a_1 \bbE_N \bar a_2$ iff $(\bar a_1,\bar
a_2$ are as above and) $p_{N,\bar a_1} = p_{N,\bar a_2}$.
\end{enumerate}
\mn
Now
\mn
\begin{enumerate}
\item[$\odot_1$]   for $N \in K_{\gs},\bbE_N$ is an
equivalence relation on $\Dom(E_N) \subseteq \seq(N)$
\sn
\item[$\odot_2$]   if $N_1 \le_{\gs} N_2$ and $\bar a \in \seq(N_1)$
\then \, $\bar a \in \Dom(\bbE_{N_1}) \Leftrightarrow \bar a 
\in \Dom(\bbE_{N_2})$
\sn
\item[$\odot_3$]   if $N_1 \le_{\gs} N_2$ and $\bar a_1,\bar a_2
\in \Dom(\bbE_{N_1})$ then $\bar a_2 \bbE_{N_1} \bar a_2
\Leftrightarrow \bar a_1 \bbE_{N_2} \bar a_2$
\sn
\item[$\odot_4$]    if $\langle N_\alpha:\alpha \le \delta\rangle$
is $\le_{\gs}$-increasing continuous and $\bar a_1 \in 
\Dom(\bbE_{N_\delta})$ then for some $\alpha < \delta$ and $\bar a_2
\in \Dom(\bbE_{N_\alpha})$ we have $\bar a_1 \bbE_{N_\delta} \bar a_2$.
\end{enumerate}
\mn
[Why?  Let $\bar a_2$ list the elements of $M_1 \le_{\gs}
N_\delta$ and let $p = p_{N_\delta,\bar a_1}$ so $p \in 
{\cS}^{\bs}_{\gs}(N_\delta)$, hence for some $\alpha <
\delta,p$ does not fork over $M_\alpha$ hence for some $M'_1
\le_{\gs} M_\alpha$ of cardinality $\lambda$, the type $p$ does
not fork over $M'_1$.  Let $\bar a_2$ list the elements of $M'_1$ such
that it explicates $p \restriction M'_1$ being a weak copy of $p^*$.  
So clearly $\bar a_2 \in \Dom(\bbE_{N_\alpha}) \subseteq 
\Dom(\bbE_{N_\delta})$ and $\bar a_1 \bbE_{N_\delta} \bar a_2$.]

Clearly we are done.

\noindent
2) Similar only we vary $(M^*,p^*)$ but it suffices to consider
   $2^\lambda$ such pairs.

\noindent
3) Should be clear.  
\end{PROOF}

\begin{dc}
\label{b62}
Assume that ${\gs}$ is a good ess-$[\lambda,\mu)$-frame so \wilog \, is full.  
We can repeat the operations in \ref{b56}(3) and \ref{b53}(2), so 
after $\omega$ times we get ${\gt}_\omega$ which is full (that is 
${\cS}^{\bs}_{\gt_\omega}(M^\omega) = {\cS}^{\na}_{\gt_\omega}(M^\omega))$ 
and ${\gt}_\omega$ has canonical type-bases as witnessed by a 
function $\bas_{\gt_\omega}$, see Definition \ref{b65}.
\end{dc}

\begin{PROOF}{\ref{b54}}
Should be clear.
\end{PROOF}

\begin{definition}
\label{b65}
We say that ${\gs}$ has type bases \If \, there is a function
$\bas(-)$ such that:
\mn
\begin{enumerate}
\item[$(a)$]   if $M \in K_{\gs}$ and $p \in {\cS}^{\bs}_{\gs}(M)$
\then \, $\bas(p)$ is (well defined and is) an element of $M$
\sn
\item[$(b)$]   $p$ does not split over bas$(p)$, that is any
automorphism\footnote{there are reasonable stronger version, but it
follows that the function $\bas(-)$ satisfies them} of $M$ over
$\bas(p)$ maps $p$ to itself
\sn
\item[$(c)$]  if $M \le_{\gs} N$ and $p \in {\cS}^{\bs}_{\gs}(N)$
then:  $\bas(p) \in M$ \Iff \, $p$ does not fork over $M$
\sn
\item[$(d)$]  if $f$ is an isomorphism from $M_1 \in K_{\gs}$
onto $M_2 \in K_{\gs}$ and $p_1 \in {\cS}^{\bs}(M_1)$ then
$f(\bas(p_1)) = \bas(f(p_1))$.
\end{enumerate}
\end{definition}

\begin{remark}
\label{e12f}
In \S3 we can add:
\mn
\begin{enumerate}
\item[$(e)$]   strong uniqueness: if $A \subseteq M \le_{\gk(\gs)}
{\gC},p \in {\cS}(A,{\gC})$ well defined, 
\then \, for at most one $q \in {\cS}^{\bs}_{\gs}
(M)$ do we have: $q$ extends $p$ and $\bas(p) \in A$.  (needed for
non-forking extensions).
\end{enumerate}
\end{remark}

\begin{definition}
\label{b71}
We say that ${\gs}$ is equivalence-closed \when \,:
\mn
\begin{enumerate}
\item[$(a)$]   ${\gs}$ has type bases $p \mapsto \bas(p)$
\sn
\item[$(b)$]   if $\bbE_M$ is a definition of an equivalence
relation on ${}^{\omega >} M$ preserved by isomorphisms and
$\le_{\gs}$-extensions (i.e. $M \le_{\gs} N \Rightarrow \bbE_M =
\bbE_N \rest {}^{\omega >} M$) \then \, there
is a definable function $F$ from ${}^{\omega >} M$ to $M$ such that
$F^M(\bar a) = F^M(\bar b)$ iff $\bar a E_M \bar b$ (or work in ${\gC}$).
\end{enumerate}
\end{definition}

\noindent
To phrase the relation between $\gk$ and $\gk'$ we define. 
\begin{definition}
\label{b77}
Assume $\gk_1,\gk_2$ are ess-$[\lambda,\mu)$-a.e.c.

\noindent
1) We say $\mathbf i$ is an interpretation in $\gk_2$ \when \,
$\mathbf i$ consists of
\mn
\begin{enumerate}
\item[$(a)$]   a predicate $P^*_{\mathbf i}$
\sn
\item[$(b)$]   a subset $\tau_{\mathbf i}$ of $\tau_{\gk_2}$.
\end{enumerate}
\mn
2) In this case for $M_2 \in K_{\gk_2}$ le $M^{[\mathbf i]}_2$ be the
$\tau_{\mathbf i}$-model $M_1 = M^{[\mathbf i]}_2$ with
\sn
\begin{enumerate}
\item[$\bullet$]   universe $P^{M_2}_{\mathbf i}$
\sn
\item[$\bullet$]   $R^{M_1} = R^{M_2} \rest |M_1|$ for $R \in
\tau_{\mathbf i}$
\sn
\item[$\bullet$]   $F^{M_1}$ similarly, so $F^{M_2}$ can be a partial
function even if $F^{M_2}$ is full.
\end{enumerate}
\mn
3) We say that $\gk_1$ is $\mathbf i$-interpreted (or interpreted by
$\mathbf i$) in $\gk_2$ \when \,:
\mn
\begin{enumerate}
\item[$(a)$]   $\mathbf i$ is an interpretation in $\gk_1$
\sn
\item[$(b)$]   $\tau_{\gk_1} = \tau_{\mathbf i}$
\sn
\item[$(c)$]   $K_{\gk_1} = \{M^{[\mathbf i]}_2:M_2 \in K_{\gk_2}\}$
\sn
\item[$(d)$]   if $M_2 \le_{\gk_2} N_2$ then $M^{[\mathbf i]}_2
\le_{\gk_1} N^{[\mathbf i]}_2$
\sn
\item[$(e)$]   if $M_1 \le_{\gk_1} N_1$ and $N_1 = N^{[\mathbf i]}_2$,
so $N_2 \in K_{\gk_2}$ \then \, for some $M_2 \le_{\gk} N_2$ we have
$M_1 = M^{[\mathbf i]}_2$
\sn
\item[$(f)$]   if $M_1 \le_{\gk_1} N_1$ and $M_1 = M^{[\mathbf i]}_2$,
so $M_2 \in K_{\gk_2}$ \then \, possible replacing $M_2$ by a model
isomorphic to it over $M_1$, there is $N_2 \in K_{\gk_2}$ we have
$M_2 \le_{\gk_2} N_2$ and $N_1 = N^{[\mathbf i]}_2$.
\end{enumerate}
\end{definition}

\begin{definition}
\label{b80}
1) Assume $\gk_1$ is interpreted by $\mathbf i$ in $\gk_2$.  We say
   strictly interpreted when: if $M^{[\mathbf i]}_2 = N^{[\mathbf i]}_2$
   then $M_2,N_2$ are isomorphic over $M^{[\mathbf i]}_2$.

\noindent
2) We say $\gk_1$ is equivalent to $\gk_2$ if there are $n$
   and $\gk'_0,\dotsc,\gk'_n$ such that $\gk_1 = \gk'_0,\gk_2 =
   \gk'_n$ and for each $\ell < n,\gk_\ell$ is strictly interpreted in
   $\gk_{\ell +1}$ or vice versa.  Actually we can demand $n=2$ and
   $k_\ell$ is strictly interpreted in $\gk'_1$ for $\ell=1,2$.
\end{definition}

\begin{definition}
\label{b83}
As above for (good) ess-$[\lambda,\mu)$-frame.
\end{definition}

\begin{claim}
\label{b86}
Assume $\gs$ is a good $\ess-[\lambda,\mu)$-frame.  \Then \, there is
  $\gC$ (called a $\mu$-saturated for $K_{\gs}$) such that:
\mn
\begin{enumerate}
\item[$(a)$]   $\gC$ is a $\tau_{\gs}$-model of cardinality $\le \mu$
\sn
\item[$(b)$]    $\gC$ is a union of some $\le_{\gs}$-increasing
continuous sequence $\langle M_\alpha:\alpha < \mu\rangle$
\sn
\item[$(c)$]   if $M \in K_{\gs}$ so $\lambda \le \|M\| < \mu$ \then
\, $M$ is $\le_{\gs}$-embeddable into some $M_\alpha$ from clause (b)
\sn
\item[$(d)$]    $M_{\alpha +1}$ is brimmed  
over $M_\alpha$ for $\alpha
< \mu$.
\end{enumerate}
\end{claim}
\newpage

\section {I ${\mathbf P}$-simple types} \label{3}

We define the basic types over sets not necessary models.  Note that
in Definition \ref{c11}  
there is no real loss using $C$ of
cardinality $\in (\lambda,\mu)$, as we can replace $\lambda$ by
$\lambda_1 = \lambda + |C|$ and so replace $K_{\gk}$ to
$K^{\gk}_{[\lambda_1,\mu)}$. 

\begin{hypothesis}
\label{c2}
1) $\gs$ is a good ess-$[\lambda,\mu)$-frame, see Definition
   \ref{b35}.

\noindent
2) ${\gs}$ has type bases, see Definition \ref{b65}.

\noindent
3) $\gC$ denote some $\mu$-saturated model for $K_{\gs}$ of
cardinality $\le \mu$, see \ref{b86}.

\noindent
4) But $M,A,\ldots$ will be $<_{\gk(\gs)} \gC,\subseteq \gC$ respectively but
 of cardinality $< \mu$.
\end{hypothesis}

\begin{definition}
\label{c5}
Let $A \subseteq M \in K_{\gs}$.

\noindent
1) $\dcl(A,M) = \{a \in M$: if $M' \le_{\gs} M'',M \le_{\gs} M''$ 
and $A \subseteq M'$ then $a \in M'$ and for every automorphism $f$ 
of $M',f \restriction A = \id_A \Rightarrow f(a) = a\}$.

\noindent
2) $\acl(A,M)$ is defined similarly but only with the first demand.
\end{definition}

\begin{definition}
\label{e.13y}
1) For $A \subseteq M \in K_{\gs}$ let

\[
{\cS}^{\bs}_{\gs}(A,M) = \{q \in {\cS}^{\bs}_{\gs}(M):\bas(q) \in 
\dcl(A,{\gC})\}.
\]

\mn
2) We call $p \in {\cS}^{\bs}_{\gs}(A,M)$ regular \If \, $p$ as a
member of $\cS^{\bs}_{\gs}(M)$ is regular. 
\end{definition}

\begin{definition}
\label{c8}
1) $\bbE_{\gs}$ is as in Claim \ref{b56}(2).

\noindent
2) If $A \subseteq M \in K_{\gs}$ and $p \in \cS^{\bs}_{\gs}(M)$, 
then $p \in \cS^{\bs}_{\gs}(A,M)$ 
\Iff \, $p$ is definable over $A$, see \ref{a35}(3) 
 \Iff \, $\inv(p)$ from Definition \ref{a35} is $\subseteq A$ and well defined.
\end{definition}

\begin{definition}
\label{c11}
Let $A \subseteq {\gC}$. 

\noindent
1) We define a dependency relation on good$(A,{\gC}) = \{c \in \gC:$ for
some $M <_{\gk(\gs)} {\gC},A \subseteq M$ and $\ortp(c,M,{\gC})$ is
definable over some finite $\bar a \subseteq A\}$ as follows:
\mn
\begin{enumerate}
\item[$\circledast$]    $c$ depends on $\mathbf J$ in $(A,\gC)$ 
iff there is no $M <_{\gk(\gs)} {\gC}$ such
that $A \cup \mathbf J \subseteq M$ and $\ortp(c,M,{\gC})$ is the
non-forking extension of $\ortp(c,\bar a,{\gC})$ where $\bar a$
witnesses $c \in \text{ good}(A,{\gC})$. 
\end{enumerate}
\mn
2) We say that $C \in {}^{\mu >}[{\gC}]$ is good over $(A,B)$
\when \, there is a brimmed $M <_{\gk(\gs)} {\gC}$ 
such that $B \cup A \subseteq M$ and $\ortp(C,M,{\gC})$ (see
Definition \ref{a35}(3)) is
definable over $A$. (In the first order context we could say $\{c,B\}$ is
independent over $A$ but here this is problematic as
$\ortp(B,A,{\gC})$ 
is not necessary basic).

\noindent
3) We say $\langle A_\alpha:\alpha < \alpha^* \rangle$ is independent
over $A$ in ${\gC}$, see \cite[L8.8,6p.5(1)]{Sh:705} \If \, we can 
find $M,\langle M_\alpha:\alpha < \alpha^* \rangle$ such that:
\mn
\begin{enumerate}
\item[$\circledast(a)$]   $A \subseteq M \le_{\gk(\gs)} M_\alpha
<_{\gs} {\gC}$ for $\alpha < \alpha^*$
\sn
\item[$(b)$]  $M$ is brimmed
\sn
\item[$(c)$]   $A_\alpha \subseteq M_\alpha$
\sn
\item[$(d)$]  $\ortp(A_\alpha,M,{\gC})$ definable over $A$ (= does
not split over $A$)
\sn
\item[$(e)$]   $\langle M_\alpha:\alpha < \alpha^* \rangle$ is
independent over $M$.
\end{enumerate}
\mn
3A) Similarly ``over $(A,B)$". 

\noindent
4) We define locally independent naturally, that is every finite
   subfamily is independent.
\end{definition}

\begin{claim}
\label{c17}
Assume $a \in \gC,A \subseteq {\gC}$.

\noindent
1) $a \in \good(A,\gC)$ iff $a$ realizes $p \in \cS^{\bs}_{\gs}(M)$
   for some $M$ satisfying $A \subseteq M <_{\gk(\gs)} \gC$.
\end{claim}

\begin{claim}
\label{c20}
1) If $A_\alpha \subseteq {\gC}$ is good over 
$(A,\bigcup\limits_{i < \alpha} A_i)$ for $\alpha < \alpha^*,\alpha^*
< \omega$ then $\langle A_\alpha:\alpha < \alpha^* \rangle$ is
independent over $A$. 

\noindent
2) Independence is preserved by reordering. 

\noindent
3) If $p \in {\cS}^{\bs}_{\gs}(\bar a,{\gC})$ is 
regular \then \, on $p({\gC}) = \{c:c$ realizes $p\}$ the 
independence relation satisfies:
\mn
\begin{enumerate}
\item[$(a)$]   like (1)
\sn
\item[$(b)$]   if $b^1_\ell$ depends on $\{b^0_0,\dotsc,b^0_{n-1}\}$
for $\ell < k$ and $b^2$ depends on $\{b^1_\ell:\ell < k\}$ \then \, $b^2$
depends on $\{b^0_\ell:\ell < n\}$
\sn
\item[$(c)$]   if $b$ depends on $\mathbf J,\mathbf J \subseteq \mathbf
J'$ \then \, $b$ depends on $\mathbf J'$.
\end{enumerate}
\end{claim}

\begin{remark}
\label{c23}
1) However, we have not mentioned finite character but the local
independence satisfies it trivially.
\end{remark}

\begin{PROOF}{\ref{c20}}
Easy.
\end{PROOF}

\begin{definition}
\label{c26}
1) Assume $q \in {\cS}^{\bs}_{\gs}(M)$ and $p \in {\cS}^{\bs}_{\gs}
(\bar a,{\gC})$.  We say that $q$ is explicitly $(p,n)$-simple \when \,,:
\mn
\begin{enumerate}
\item[$\circledast$]   there are $b_0,\dotsc,b_{n-1},c$ such that
\footnote{clause (c) + (e) are replacements for $c$ is algebraic over
$\bar a + \{b_\ell:\ell < n\}$ and each $b_\ell$ is necessary}:
\sn
\begin{enumerate}
\item[$(a)$]  $b_\ell$ realizes $p$
\sn
\item[$(b)$]   $c$ realizes $q$
\sn
\item[$(c)$]   $b_\ell$ is not good \footnote{not good here is a
replacement to ``$\ortp(b_\ell,\bar a +c,{\gC})$ does not fork over
$\bar a$"} over $(\bar a,c)$ for $\ell < n$
\sn
\item[$(d)$]  $\langle b_\ell:\ell < n \rangle$ is independent
over $\bar a$
\sn
\item[$(e)$]  $\langle c,b_0,\dotsc,b_{n-1} \rangle$ is good over $\bar a$
\sn
\item[$(f)$]   if \footnote{this seems a reasonable choice here but we
can take others; this is an unreasonable choice for first order}
$c'$ realizes $q$ then $c=c'$ \Iff \, for
every $b \in p({\gC})$ we have: $b$ is good over $(\bar a,c)$ iff
$b$ is good over $(\bar a,,c')$.
\end{enumerate}
\end{enumerate}
\mn
1A) We say that $a$ is explicitly $(p,n)$-simple over $A$ \If \,
$\ortp(a,A,{\gC})$ is; similarly in the other definitions replacing
$(p,n)$ by $p$ means ``for some $n$". 

\noindent
2) Assume $q \in {\cS}^{\bs}_{\gs}(\bar a,{\gC})$ and ${\mathbf P}$ as 
in Definition \ref{a17}.  We say that $q$ is 
${\mathbf P}$-simple if we can find $n$ and
explicitly ${\mathbf P}$-regular types $p_0,\dotsc,p_{n-1} \in 
{\cS}^{\bs}_{\gs}(\bar a,{\gC})$ such that: each $c \in p({\gC})$ 
is definable by its type over $\bar a \cup \bigcup\limits_{\ell <n}
p_\ell({\gC})$,

\noindent
3) In part (1) we say strongly explicitly $(p,n)$-simple \If \ 
there are $k>a$ and $\langle \bar a^*_\ell:\ell < \omega \rangle$ 
and $r \in {\cS}^{\bs}_{\gs}(\bar a,{\gC})$ such that [on finitely many see
\ref{c26}(3B) below]:
\mn
\begin{enumerate}
\item[$(a)$]   $\{\bar a^*_\ell:\ell< \omega\} \in r({\gC})$ is
independent for any $c',c'' \in p({\gC})$, we have $c' = c''$
\Iff \, for infinitely many $m < \omega$ ($\equiv$ for all but
finitely many $\models$, see claim on average) for every $b \in
p({\gC})$ we have:
\sn
\begin{enumerate}
\item[$(*)$]   $b$ is good over $(\bar a,\bar a,a^*_m,c)$ iff $b$
is good over $(\bar a,a,a^*_m,c)$, (compare with \ref{c41},
\ref{c56}!).
\end{enumerate}
\end{enumerate}
\mn
3A) In part (1) we say weakly $(p,n)$-simple if in $\circledast$,
clause $(f)$ is replaced by
\mn
\begin{enumerate}
\item[$(f)'$]   if $b$ is good over $(\bar a,a^*_m)$ then
$c',c'$ realizes the same type over $\bar a a^*_m b$.
\end{enumerate}
\mn
3B) In part (1) we say $(p,n)$-simple \If \, for some $\bar a^* \in
{}^{\omega >} {\gC}$ good over $\bar a$ for every $c \in q({\gC})$ 
there are $b_0,\dotsc,b_{n-1} \in p({\gC})$ such that $c \in \dcl
(\bar a,\bar a^*,b_0,\dotsc,b_{n-1})$ and $\bar a \char 94 \langle
b_0,\dotsc,b_{n-1}\rangle$ is good over $\bar a$ if simple.

\noindent
4) Similarly in (2). 

\noindent
5) We define $g w_p(b,\bar a),p$ regular parallel to some $p' \in
{\cS}^{\bs}_{\gs}(\bar a)$ ($gw$ for general weight).  Similarly
for $g w_p(q)$.
\end{definition}

\noindent
We first list some obvious properties.
\begin{claim}
\label{c32}
1) If $c$ is ${\mathbf P}$-simple over $\bar a,\bar a \subseteq A 
\subset {\gC}$ \then \, $w_p(c,A)$ is finite. 

\noindent
2) The obvious implications.
\end{claim}

\begin{claim}
\label{c35}
1) [Closures of the simple $\bs$]. 

\noindent
2) Assume $p \in {\cS}^{\bs}_{\gs}(\bar a,{\gC})$.  If $\bar b_1,
\bar b_2$ are $p$-simple over $A$ \then \,
\mn
\begin{enumerate}
\item[$(a)$]   $\bar b_1 \char 94 \bar b_2$ is $p$-simple (of course,
$\ortp_{\gs}(\bar b_2 \bar b_2,\bar a,{\gC})$ is not necessary in
${\cS}^{\bs}_{\gs}(\bar a,{\gC})$ even if 
$\ortp_{\gs}(\bar b_\ell,\bar a,{\gC}) \in 
{\cS}^{\bs}_{\gs}(\bar a,{\gC})$ for $\ell = 1,2$)
\sn
\item[$(b)$]   also $\ortp(\bar b_2,\bar a b_1,{\gC})$ is 
${\mathbf P}$-simple.
\end{enumerate}
\mn
2) If $\bar b_\alpha$ is $p$-simple over $\bar a$ for $\alpha <
\alpha^*,\pi:\beta^* \rightarrow \alpha^*$ one to one into, \then \,
$\sum\limits_{\alpha < \alpha^*} g w_p(b_\alpha,\bar a_* \cup
\bigcup\limits_{\ell < \alpha} b_\alpha) = \sum\limits_{\beta < \beta^*}
gw(b_{\pi(\beta)}) \bar a \cup \bigcup\limits_{i < \beta} \bar b_{\pi(i)})$.
\end{claim}

\begin{claim}
\label{c38}
[${\gs}$ is equivalence-closed]. 

\noindent
Assume that $p,q \in {\cS}^{\bs}(M)$ are not weakly orthogonal
(e.g. see \ref{g29}).
\Then \, for some $\bar a \in {}^{\omega >} M$ we have: $p,q$ are
definable over $\bar a$ (works without being stationary) and for some
${\gk}_{\gs}$-definable function $\mathbf F$, 
for each $c \in q({\gC}),\ortp_{\gs}(\mathbf F(c,\bar a),\bar a,{\gC}) 
\in {\cS}^{\bs}_{\gs}(\bar a,{\gC})$ and is explicitly $(p,n)$-simple
for some $n$, (if, e.g., $M$ is $(\lambda,*)$-brimmed then $n = w_p(q)$.  
\end{claim}

\begin{PROOF}{\ref{c38}}
We can find $n$ and $c_1,b_0,\dotsc,b_{-1} \in
{\gC}$ with $c$ realizing $q,b_\ell$ realizing $p,\{b_\ell,c\}$ is
not independent over $M$ and $n$ maximal.  Choose $\bar a \in
{}^{\omega >} M$ such that $\ortp_{\gs}(\langle c,b_0,\dotsc,b_{n-1}
\rangle,M,{\gC} \rangle)$ is definable over $\bar a$.  Define
$E_{\bar a}$, an equivalence relation on $q({\mathfrak C}):c_1 E_{\bar a}
c_2$ iff for every $b \in p{\gC})$, we have $b$ is good over
$(a,c_1) \Rightarrow b$ is good over $(\bar a,c_2)$.
By ``${\gs}$ is eq-closed", we are done.
\end{PROOF}

\begin{claim}
\label{c41}
1) Assume $p,q \in {\cS}^{\bs}_{\gs}(M)$ are weakly orthogonal (see e.g.
\ref{g29}(1)) but not orthogonal.  \Then \, we can
find $\bar a \in {}^{\omega >} M$ over which $p,q$ are definable and
$r_1 \in {\cS}^{\bs}_{\gs}(\bar a,{\gC})$ such that letting
$p_1 = p \restriction \bar a,q_1 = q \restriction \bar a,n =: w_p(q)
\ge 1$ we have:
\mn
\begin{enumerate}
\item[$\circledast^n_{\bar a,p_1,q_1,r_2}(a)$]  $p_1,q_1,r_1 \in
{\cS}^{\bs}_{\gs}(\bar a),\bar a \in{}^{\omega >}{\gC}$
\sn
\item[$(b)$]   $p_1,q_1$ are weakly orthogonal (see e.g. Definition
  \ref{g29}(1))
\sn
\item[$(c)$]  if $\{a^*_n:n < \omega\} \subseteq r_1({\gC})$ is
independent over $\bar a$ and $c$ realizes $q$ \then \ for infinitely
many $m < \omega$ there is $b \in p({\mathfrak C})$ such that $b$ is good
over $(\bar a,a^*_n)$ but not over $(\bar a,a^*_n,c)$
\sn
\item[$(d)$]   in (c) really there are $n$ independent such that 
(but not $n+1$).
\end{enumerate}
\mn

\noindent
2) If $\circledast^n_{\bar a,p_1,q_1,r_1}$ then (see Definition
\ref{c26}(3) for some definable function $\mathbf F$, if $c$ realizes
$q_1,c^* = F(c,\bar a)$ and $\ortp_{\gn}(c^*,\bar a,{\gC})$ is
$(p_1,n)$-simple. 
\end{claim}

\noindent
See proof below.
\begin{claim}
\label{c44}

0) Assuming $ A \subseteq  \mathfrak{C} $  and 
$ a \in \mathfrak{C}  $ is called finitary  when 
it is definable over $ \{ a_0, \dots , a_{n-1}\}  $ 
where each $ a_ {\ell} $ is in $ \mathfrak{C} $  
and is good over $  A $  inside $ \mathfrak{C} $. 

\noindent
1) If $a \in \dcl(\cup\{A_i:i < \alpha\}
\cup A,{\gC})$ and $\ortp(a,A,{\gC})$ is finitary
over (A) and
$\{A_i:i < \alpha\}$ is independent over $A$ \then \, for some finite
$u \subseteq \alpha$ we have $a \in \dcl(\cup\{A_i:i \in u\} \cup A,
{\gC})$. 

\noindent
2) If $\ortp(b,\bar a,{\gC})$ is ${\mathbf P}$-simple, \then \, it is
finitary.

\noindent
3) If $\{A_i:i < \alpha\}$ is independent over $A$ and $a$ is finitary
over $A$ then for some finite $u \subseteq \alpha$ (even $|u| <
\wg(c,A)),\{A_i:i \in \alpha \backslash u\}$ is independent over $A,A
\cup \{c\})$ (or use $(A',A''),(A',A'' \cup \{c\})$).
\end{claim}

\begin{definition}
\label{c50}
1) $\dcl(A) = \{a$: for every automorphism $f$ of ${\gC},f(a) = a\}$. 

\noindent
2) $\dcl_{\fin}(A) = \cup\{\dcl(B):B \subseteq A$ finite$\}$. 

\noindent
3) $a$ is finitary over $A$ \If \, there are $n < \omega$ and
$c_0,\dotsc,c_{n-1} \in$ good$(A)$ such that $a \in \dcl(A \cup
\{c_0,\dotsc,c_{n-1}\}$ (or dcl$_{\text{fin}}$?). 

\noindent
4) For such $A$ let $\wg(a,A)$ be $w(\tp(a,A,\gC))$ when well defined.

\noindent
5) Strongly simple implies simple.
\end{definition}

\begin{claim}
\label{c53}
In Definition \ref{c26}(3), for
some $m,k < \omega$ large enough, for every $c \in q({\gC})$ there
are $b_0,\dotsc,b_{m-1} \in \bigcup\limits_{\ell < n} p_\ell({\gC})$ such
that $c \in \rm dcl(\bar a \cup \{a^*_\ell:\ell < k\} \cup 
\{b_\ell:\ell < m\})$.
\end{claim}

\begin{PROOF}{\ref{c41}}
Let $M_1,M_2 \in K_{\gs(\brim)}$ be
such that $M \le_{\gs} M_1 \le_{\gs} M_2,M_1 \,
(\lambda,*)$-brimmed over $M,p_\ell \in {\cS}^{\bs}_{\gs}(M_\ell)$ a 
non-forking extension of $p,q_\ell \in {\cS}^{\bs}_{\gs}(M_\ell)$ is 
a non-forking extension of $q,c \in M_2$ realizes $q_1$ and 
$(M_1,M_2,c) \in K^{3,\pr}_{\gs(\brim)}$.  Let
$b_\ell \in p_1(M_2)$ for $\ell < n^* =: w_p(q)$ be such that
$\{b_\ell:\ell < n^*\}$ is independent in $(M_1,M_2)$, let $\bar a^*
\in {}^{\omega >}(M_1)$ be such that $\ortp_{\gs}(\langle
c,b_0,\dotsc,b_{n-1} \rangle,M_1,M_2)$ is definable over $\bar a^*$
and $r = \ortp_{\gs}(\bar a^*,M_1,M_2),r^+ = \ortp(\bar
a^* \char 94 \langle b_0,\dotsc,b_{n-1} \rangle,M,M_2)$.

Let $\bar a \in {}^{\omega >} M$ be such that $\ortp_{\gs}(\bar
a^*,\langle c,b_0,\dotsc,b_{n-1} \rangle,M,M_2)$ is definable over
$\bar a$.  As $M_1$ is $(\lambda,*)$-saturated over $M$ there is
$\{\bar a^*_f:f < \omega\} \subseteq r({\gC})$ independent in
$(M,M_1)$ moreover letting $a^*_\omega = \bar a^*$, we have $\langle
a^*_\alpha:\alpha \le \omega \rangle$ is independent in $(M,M_1)$.  Clearly
$\ortp_{\gs}(c \bar a^*_n,M,M_2)$ does not depend on $n$ hence we can
find $\left < \langle b^\alpha_\ell:\ell < n \rangle:\alpha \le \omega
\right>$ such that $b^\alpha_\ell \in M_2,b^\omega_\ell = b_\ell$ and
$\{c \bar a^*_\alpha,b^\alpha_0 \ldots b^\alpha_{n-1}:\alpha \le
\omega\}$ (as usual as index set is independent in $(M_1,M_2)$.

The rest should be clear.
\end{PROOF}

\begin{definition}
\label{c56}
Assume $\bar a \in {}^{\omega >} {\gC},n < \omega,p,q,r \in 
\cS^{\bs}(M)$ are as in the
definition of $p$-simple$^{[-]}$ but $p,q$ are weakly
orthogonal (see e.g. Definition \ref{g29}(1)) let $p$ be a definable related
function such that for any $\bar a^\nu_\ell \in r({\gC}),\ell <
k^*$ independent mapping $c \mapsto \langle \{b \in q({\gC}):R
{\gC} \models R(b,c,\bar a^*_\ell)\}$ is a one-to-one function
from $q({\gC})$ into $\{\langle J_\ell:\ell < k^* \rangle:J_\ell
\subseteq p({\gC})$ is closed under dependence and has $p$-weight
$n^*\}$ 


\noindent
1) We can define $E = E_{p,q,r}$ a two-place relation over 
$r({\gC}):\bar a^*_1 E \bar a^*_2$ \Iff \, $\bar a_1,\bar a_2 \in
r(\gC)$ have the same projection common to $p(\gC)$ and $q(\gC)$. 

\noindent
2) Define unit-less group on $r/E$ and its action on $q({\gC})$.
\end{definition}

\begin{remark}
\label{e.22}
1) A major point is: as $q$ is $p$-simple,
$w_p(-)$ acts ``nicely" on $p({\gC})$ so if $c_1,c_2,c_3
\in q({\gC})$ then $w_p(\langle c_1,c_2,c_3 \rangle \bar a) \le 3n^*$.  This
enables us to define average using a finite sequence seem quite
satisfying.  Alternatively, look more at averages of independent
sets. 

\noindent
2) \underline{Silly Groups}:  Concerning interpreting groups note that in our
present context, for every definable set $P^M$ we can add the group of
finite subsets of $P^M$ with symmetric difference (as addition).

\noindent
3) The axiomatization above has prototype ${\mathfrak s}$ where $K_{\gs}
= \{M:M$ a $\kappa$-saturated model of $T\},\le_{\gs} = \prec
\restriction K_{\gs},\nonfork{}{}_{\gs}$ is non-forking, $T$ a
stable first order theory with $\kappa(T) \le \cf(\kappa)$.
But we may prefer to formalize the pair $({\gt},{\gs}),{\gs}$ as
above, $K_{\gt} =$ models of $T,\le_{\gt} = \prec
\restriction K_{\gt},\nonfork{}{}_{\gt}$ is non-forking.

\relax From ${\gs}$ we can reconstruct a ${\gt}$ by closing 
${\gk}_{\gs}$ under direct limits, but in interesting cases we end up
with a bigger ${\gt}$.
\end{remark}
\newpage

\centerline {Part II Generalizing Stable Classes} \label{part2}

\section {II Introduction} \label{4}

In this part we try to deal with classes like ``$\aleph_1$-saturated
models of a first order theory $T$ and even a stable one" rather
than of ``a model of $T$".  
The parallel problem for ``model of $T$, even superstable one"
is the subject of \cite{Sh:h}.
\bigskip

\centerline {$* \qquad * \qquad *$}

Now some construction goes well by induction on cardinality, say by
dealing with $(\lambda,\cP(n))$-system of models but not all.
E.g. starting with $\aleph_0$ we may consider $\lambda > \aleph_0$, so
can find $F:[\lambda]^{\aleph_0} \rightarrow \lambda$ such that there
is an infinite decreasing sequence of $F$-closed subsets of $\lambda,u
\in [\lambda]^{< \aleph_0} \Rightarrow F(u)=0$ maybe such that $u \in
[\lambda]^{\le \aleph_0} \Rightarrow |c\ell_F(u)| \le \aleph_0$.
Let $\langle u_\alpha:\alpha < \alpha_* \rangle$ list $\{c \ell_F(u):u
\in [\lambda]^{\le \aleph_0}\}$ such that $c \ell_F(u_\alpha)
\subseteq c \ell_F(u_\beta) \Rightarrow \alpha \le \beta$ we try to
choose $M_{u_\alpha}$ by induction on $\alpha$.
\newpage

\section {II Axiomatizing a.e.c. without full continuity} \label{5}

\subsection {a.e.c.} \label{5A} \
\bigskip

Classes like ``$\aleph_1$-saturated models of a first order $T$, which
is not superstable", does not fall under a.e.c., still they are close
and below we suggest a framework for them.  So for increasing sequences of
short length we have weaker demand.
\smallskip

\noindent
We shall 
say more   on primes later. 

We shall lift a $(\mu,\lambda,\kappa)$-a.e.c. to
$(\infty,\lambda,\kappa)$-a.e.c. (see below), so actually 
$k_\lambda$ suffice, but for our main objects, good frames, this is
more complicated as its properties (e.g., the amalgamation property)
are not necessarily preserved by the lifting.

This section generalizes \cite[\S1]{Sh:600}, in some cases the
differences are minor, sometimes the whole point is the difference.

\begin{convention}
\label{f0}
1) In this section $\gk$ will denote a directed a.e.c., see Definition
\ref{f2}, may write d.a.e.c. (the $d$ stands for directed).

\noindent
2) We shall write (outside the definitions)
$\mu_{\gk},\lambda_{\gk},\kappa_{\gk}$. 
\end{convention}

\begin{definition}
\label{f2}  
Assume $\lambda < \mu,\lambda^{<
\kappa} = \lambda$ (for notational simplicity) and $\alpha < \mu \Rightarrow 
|\alpha|^{< \kappa} < \mu$ and $\kappa$ is regular.

We say that ${\gk}$ is a $(\mu,\lambda,\kappa)-1$-d.a.e.c. 
(we may omit or add the ``$(\mu,\lambda,\kappa)$" by
Ax(0)(d) below, similarly in similar definitions;
if $\kappa = \aleph_0$ we may omit it, 
instead $\mu = \mu^+_1$ we may write $\le \mu_1$)
\when \, the axioms below hold; we 
write d.a.e.c. or 0-d.a.e.c. when we omit Ax(III)(b),(IV)(b)
 Ax(0) ${\gk}$ consists of
\mn
\begin{enumerate}
\item[$(a)$]  $\tau_{\gk}$, a vocabulary with each predicate
and function symbol of arity $\le \lambda$
\sn
\item[$(b)$]    $K$, a class of $\tau$-models 
\sn
\item[$(c)$]   a two-place relation $\le_{\gk}$ on $K$
\sn
\item[$(d)$]    the cardinals $\mu = \mu_{\gk},\mu(\gk),\lambda =
\lambda_{\gk} = \lambda(\gk)$ 
and $\kappa = \kappa_{\gk} = \kappa(\gk)$
(so $\mu > \lambda =
\lambda^{< \kappa} \ge \kappa = \cf(\kappa)$ and $\alpha < \mu
\Rightarrow |\alpha|^{< \kappa} < \mu$)
\end{enumerate}
\mn
such that
\mn
\begin{enumerate}
\item[$(e)$]  if $M_1 \cong M_2$ then $M_1 \in K \Leftrightarrow
M_2 \in K$
\sn
\item[$(f)$]  if $(N_1,M_1) \cong (N_2,M_2)$ \then \, $M_1
\le_{\gk} N_1 \Rightarrow M_2 \le_{\gk} N_2$
\sn
\item[$(g)$]  every $M \in K$ has cardinality $\ge \lambda$ but $< \mu$
\sn
\item[$(Ax(I)(a))$]  $M \le_{\gk} N \Rightarrow M \subseteq N$
\sn
\item[$(Ax(II)(a))$]  $\le_{\gk}$ is a partial order
\sn
\item[$Ax(III)$]   assume that $\langle M_i:i < \delta \rangle$ is a
$\le_{\gk}$-increasing sequence and $\| \cup\{M_i:i < \delta\}\| <
\mu$ \then \,
\sn
\begin{enumerate}
\item[$(a)$]   (existence of unions) if $\cf(\delta) \ge \kappa$ 
then there is $M \in K$ 
such that $i < \delta \Rightarrow M_i \le_{\gk} M$ and
$|M| = \cup\{|M_i|:i < \delta\}$ but not necessarily $M =
\bigcup\limits_{i < \delta} M_i$
\sn
\item[$(b)$]  (existence of limits) there is $M \in K$ such that $i <
\delta \Rightarrow M_i \le_{\gk} M$
\end{enumerate}
\item[$Ax(IV)(a)$]     (weak uniqueness of limit = weak smoothness)
for $\langle M_i:i < \delta\rangle$ as above, 
\sn
\begin{enumerate}
\item[$(a)$]   if $\cf(\delta) \ge \kappa$ and $M$ is as in Ax(III)(a)
and $i <  \delta \Rightarrow M_i \le_{\gk} N$ then $M \le_{\gk} N$
\sn
\item[$(b)$]  if $N_\ell \in K$ and $i < \delta \Rightarrow M_i
\le_{\gk} N_\ell$ for $\ell =1,2$ \then \, there are $N \in K$ and
$f_1,f_2$ such that $f_\ell$ is a $\le_{\gk}$-embedding of
$N_\ell$ into $N$ for $\ell=1,2$ and $i < \delta \Rightarrow f_1
\restriction M_i = f_2 \restriction M_i$
\end{enumerate}
\item[$Ax(V)$]    if $N_\ell \le_{\gk} M$ for $\ell=1,2$ and
$N_1 \subseteq N_2$ \then \, $N_1 \le_{\gk} N_2$
\sn
\item[$Ax(VI)$]   (L.S.T. property)
if $A \subseteq M \in K,|A| \le \lambda$ \then \,
there is $M \le_{\gk} N$ of cardinality $\lambda$ such that $A
\subseteq M$.
\end{enumerate}
\end{definition}

\begin{remark}
There are some more axioms listed in \ref{f4}(5), but
we shall mention them in any claim in which they are used so no need
to memorize, so \ref{f4}(1)-(4) omit them?
\end{remark}

\begin{definition}
\label{f4}
1) We say $\gk$ is a 4-d.a.e.c. or d.a.e.c$^+$ when it is a
$(\lambda,\mu,\kappa)-1$-d.a.e.c. and satisfies Ax(III)(d),Ax(IV)(e)
   below.

\noindent
2) We say $\gk$ is a 2-d.a.e.c. or d.a.e.c.$^\pm$ \when \, is a
$(\lambda,\mu,\kappa)-0$-d.a.e.c. and Ax(III)(d), Ax(IV)(d) below holds.

\noindent 
3) We say $\gk$ is 5-d.a.e.c. \when \, it is 1-d.a.e.c. and Ax(III)(f)
   holds.

\noindent
4) We say $\gk$ is 6-d.a.e.c. \when \, it is a 1-d.a.e.c. and
   Ax(III)(f) + Ax(IV)(f).

\noindent
5) Concerning Definition \ref{f2}, we consider the following axioms:
\medskip

\noindent
Ax(III)(c) \qquad if $I$ is $\kappa$-directed and $\bar M =
\langle M_s:s \in I\rangle$ is $\le_{\gk}$-increasing $s \le_I$

\hskip60pt  $t \Rightarrow M_s \subseteq M_s$ and 
$\Sigma\{\|M_s\|:s \in I\} < \mu$ \then \, $\bar M$ has a 

\hskip60pt $\le_{\gs}$-upper bound, $M$, i.e. $s \in I
\Rightarrow M_s \le_{\gk} M$.  
\medskip

\noindent
Ax(III)(d) \qquad (union of directed system) if $I$ is $\kappa$-directed,
$|I| < \mu,\langle M_t:t \in I \rangle$ 

\hskip60pt is $\le_{\gk}$-increasing
and $\| \cup \{M_t:t \in I\}\| < \mu$ \then \, there is

\hskip60pt one and only one $M$ with universe 

\hskip60pt $\cup\{|M_t|:t \in I\}$ such that 
$M_s \le_{\gk} M$ for every $s \in I$ 

\hskip60pt  we call it the $\le_{\gk}$-union of $\langle M_t:t \in I\rangle$.
\medskip

\noindent
Ax(III)(e) \qquad like Ax(III)(c) but $I$ is just directed
\medskip

\noindent
Ax(III)(f) \qquad If $\bar M = \langle M_i:i < \delta\rangle$ is
$\le_{\gk}$-increasing, $\cf(\delta) < \kappa$ and $|\cup\{M_i:i <$

\hskip60pt  $\delta\}| < \mu$ \then \, 
there is $M$ which is $\le_{\gk}$-prime over $\bar M$, i.e.
\mn
\begin{enumerate}
\item[${{}}$]  $(*) \quad$ if $N \in K_{\gk}$ and $i < \delta \Rightarrow M_i
  \le_{\gk} N$ \then \, there is a $\le_{\gk}$-embedding 

\hskip40pt  of $M$ into $M$ over $\cup\{|M_i|:i < \delta\}$.
\end{enumerate}
\medskip

\noindent
Ax(IV)(c) \qquad If $I$ is $\kappa$-directed and $\bar M = \langle
M_s:s \in I\rangle$ is $\le_{\gk}$-increasing and $N_1,N_2$ 

\hskip60pt  are $\le_{\gs}$-upper bounds of $\bar M$ \then \, for some 
$(N'_2,f)$ we have

\hskip60pt $N_2 \le_{\gk} N'_2$ and 
$f$ is a $\le_{\gk}$-embedding of $N_1$ into $N_2$ 

\hskip60pt  which is the identity on $M_s$ for every $s \in I$

\hskip60pt (this is a weak form of uniqueness)
\medskip

\noindent
Ax(IV)(d) \qquad If $I$ is a $\kappa$-directed partial order, $\bar M = \langle
M_s:s \in I\rangle$ is $\le_{\gk}$-increasing, 

\hskip60pt  $s \in I \Rightarrow M_s
\le_{\gk} M$ and $|M| = \cup|M_s|:s \in I\}$, \then \,

\hskip60pt  $\bigwedge\limits_{s} M_s \le_{\gk} N \Rightarrow M \le_{\gk} N$.
\medskip

\noindent
Ax(IV)(e) \qquad Like Ax(IV)(c) but $I$ is just directed.
\medskip

\noindent
Ax(IV)(f) \qquad If $I$ is directed and $\bar M = \langle M_s:s \in
I\rangle$ is $\le_{\gk}$-increasing \then \, there is 

\hskip60pt  $M$ which is a
$\le_{\gk}$-prime over $\bar M$, defined as in Ax(III)(f).
\end{definition}

\begin{claim}
\label{f6}
Assume\footnote{By \ref{f0} no need to say this} $\gk$ is a d.a.e.c.

\noindent
1) Ax(III)(d) implies Ax(III)(c).

\noindent
2) Ax(III)(e) implies Ax(III)(c) and it implies Ax(III)(b).

\noindent
3) Ax(IV)(d) implies Ax(IV)(c).

\noindent
4) Ax(IV)(e) implies Ax(IV)(c) and implies Ax(III)(b).

\noindent
5) In all the axioms in Definition \ref{f4} it is necessary that
$|\cup\{M_s:s \in I\}| < \mu_{\gk}$.

\noindent
6) Ax(IV)(b) implies that $\gk$ has amalgamation.
\end{claim}

\begin{definition}
\label{f11}
We say $\langle M_i:i < \alpha \rangle$ is 
$\le_{\gk}$-increasing $(\ge \kappa)$-continuous
\when \, it is $\le_{\gk}$-increasing and $\delta < \alpha$ and 
$\cf(\delta) \ge \kappa \Rightarrow |M_i| = \cup\{|M_j|:j < \delta\}$.
\end{definition}

\noindent
As an exercise we consider directed systems with mappings.
\begin{definition}
\label{f14}
1) We say that $\bar M = \langle M_t,h_{s,t}:s \le_I t\rangle$ is 
a $\le_{\gk}$-directed system \when \, $I$ is a directed partial order and
if $t_0 \le_I t_2 \le_I t_2$ then
$h_{t_2,t_0} = h_{t_2,t_1} \circ h_{t_1,t_0}$.

\noindent
1A) We say that $\bar M = \langle M_t,h_{s,t}:s \le_I t \rangle$ is a
$\le_{\gk}-\theta$-directed system when in addition $I$ is
$\theta$-directed.

\noindent
2) We omit $h_{s,t}$ when $s \le_I t \Rightarrow h_{s,t} = 
\id_{M_s}$ and write $\bar M = \langle M_t:t \in I\rangle$.

\noindent
3) We say $(M,\bar h)$ is a $\le_{\gk}$-limit of 
$\bar M$ when $\bar h = \langle
h_s:s \in I\rangle,h_s$ is a $\le_{\gk}$-embedding of $M_s$ into
$M_s$ and $s \le_I t \Rightarrow h_s = h_t \circ h_{t,s}$.

\noindent
4) We say $\bar M = \langle M_\alpha:\alpha < \alpha^*\rangle$ is
 $\le_{\gk}$-semi-continuous \when \,: (see Ax(III)(f) in \ref{f4})
\mn
\begin{enumerate}
\item[$(a)$]   $\bar M$ is $\le_{\gk}$-increasing
\sn
\item[$(b)$]   if $\alpha < \alpha^*$ has cofinality $\ge \kappa$
 then $M_\alpha = \cup\{M_\beta:\beta < \alpha\}$
\sn
\item[$(c)$]   if $\alpha < \alpha^*$ has cofinality $<\kappa$ then
$M_\delta$ is $\le_{\gk}$-prime over $\bar M \restriction \alpha$.
\end{enumerate}
\end{definition}

\begin{observation}
\label{f17}
[${\gk}$ is an d.a.e.c.]

\noindent
1) If $\bar M = \langle M_t,h_{s,t}:s \le_I t \rangle$ is a 
$\le_{\gk}$-directed system, \then \, we can
find a $\le_{\gk}$-directed system $\langle M'_t:t \in I\rangle$
(so $s \le_I t \Rightarrow M'_s \le_{\gk} M'_t$) and $\bar g =
\langle g_t:t \in I\rangle$ such that:
\mn
\begin{enumerate}
\item[$(a)$]   $g_t$ is an isomorphism from $M_t$ onto $M'_t$
\sn
\item[$(b)$]   if $s \le_I t$ then $g_s = g_t \circ h_{s,t}$.
\end{enumerate}
\mn
2) So in the axioms (III)(a),(b)(IV)(a) from Definition \ref{f2} as
well as those of \ref{f4} we can use $\le_{\gk}$-directed system
$\langle M_s,h_{s,t}:s \le_I t \rangle$ with $I$ as there.

\noindent
3) If $\gk$ is an ess-$(\mu,\lambda)$-a.e.c., see \S1 \then \, $\gk$
   is a $(\mu,\lambda,\aleph_0)$-d.a.e.c. and satisfies all the axioms
   from \ref{f4}.

\noindent
4) If $(M,\bar h)$ is prime over $\bar M = \langle
M_t,h_{s,t}:s \le_I t\rangle$ and $\chi = \Sigma\{\|M_t\|:t \in I\}$
then $\|M\| \le \chi^{< \kappa}$.
\end{observation}

\begin{PROOF}{\ref{f17}}
Straightforward, e.g. we can use ``${\gk}$ has $(\chi^{< \kappa})$-LST",
i.e. Observation \ref{f19}.
\end{PROOF}

\noindent
More serious is proving the LST theorem in our content (recall that in
the axioms, see Ax(VI), we demand it only down to $\lambda$).

\begin{claim}
\label{f19}
[$\gk$ is a $(\mu,\lambda,\kappa)-2$-d.a.e.c.,
see Definition \ref{f4}.]

If $\lambda_{\gk} \le \chi = \chi^{< \kappa} < \mu_{\gk},A \subseteq N
\in \gk$ and $|A| \le \chi \le \|N\|$ \then \, there is $M \le_{\gk}
N$ of cardinality $\chi$ such that $\|M\| = \chi$.
\end{claim}

\begin{PROOF}{\ref{f19}}
Let $\langle u_\alpha:\alpha < \alpha(*)\rangle$ list 
$[A]^{< \kappa({\gk})}$, let $I$ be the following partial order:
\mn
\begin{enumerate}
\item[$(*)_1$]   $(\alpha) \quad$ set of elements is $\{\alpha < \chi$: for no
$\beta < \alpha$ do we have $u_\alpha \subseteq u_\beta\}$
\sn
\item[${{}}$]   $(\beta) \quad \alpha \le_I \beta$ iff $u_\alpha \subseteq
u_\beta$ (hence $\alpha \le \beta$).
\end{enumerate}
\mn
Easily
\mn
\begin{enumerate}
\item[$(*)_2$]   $(a) \quad I$ is $\kappa$-directed
\sn
\item[${{}}$]  $(b) \quad$ for every $\alpha < \alpha(*)$ for some $\beta
< \alpha(*)$ we have $u_\alpha \subseteq u_\beta \wedge \beta \in I$
\sn
\item[${{}}$]   $(c) \quad \cup\{u_\alpha:\alpha \in I\} = A$.
\end{enumerate}
\mn
Now we choose $M_\alpha$ by induction on $\alpha < \chi$ such that
\mn
\begin{enumerate}
\item[$(*)_3$]   $(a) \quad M_\alpha \le_{\gk} N$
\sn
\item[${{}}$]   $(b) \quad \|M_\alpha\| = \lambda_{\gk}$
\sn
\item[${{}}$]   $(c) \quad M_\alpha$ include $\cup\{M_\beta:\beta
<_I \alpha\} \cup u_\alpha$.
\end{enumerate}
\mn
Note that $|\{\beta \in I:\beta <_I \alpha\}| \le |\{u:u \subseteq
u_\alpha\}| = 2^{|u_\alpha|} \le 2^{< \kappa({\gk})} \le
\lambda_{\gk}$ and by the
induction hypothesis $\beta < \alpha \Rightarrow \|M_\beta\| \le
\lambda_{\gk}$ and recall $|u_\alpha| < \kappa({\gk}) \le 
\lambda_{\gk}$ hence the set $\cup\{M_\beta:\beta < \alpha\} 
\cup u_\alpha$ is a subset of $N$ of cardinality $\le \lambda$ 
hence by Ax(VI) there is $M_\alpha$ as required.

Having chosen $\langle M_\alpha:\alpha \in I\rangle$ clearly by Ax(V) it is a
$\le_{\gk}$-directed system hence by Ax(III)(d), $M =
\cup\{M_\alpha:\alpha \in I\}$ is well defined 
with universe $\cup\{|M_\alpha|:\alpha \in I\}$ and by Ax(IV)(d) we
have $M \le_{\gk} N$.

Clearly $\|M\| \le \Sigma\{\|M_\alpha\|:\alpha \in I\} \le |I| \cdot
\lambda_{\gk} = \chi$, and by $(*)_2(c) + (*)_3(c)$ we have $A \subseteq
\cup\{u_\alpha:\alpha < \chi\} = \cup\{u_\alpha:\alpha \in I\}
\subseteq \cup\{|M_\alpha|:\alpha \in I\} = M$ and so
$M$ is as required.  
\end{PROOF}

\begin{notation}
\label{f22}
1) For $\chi \in [\lambda_{\gk},\mu_{\gk})$ let $K_\chi = K^{\gk}_\chi
  = \{M \in K:\|M\| = \chi\}$ and
$K_{< \chi} = \bigcup\limits_{\mu < \chi} K_\mu$. 

\noindent
2) ${\gk}_{\chi} = (K_\chi,\le_{\gk} \restriction K_\chi)$.

\noindent
3) If $\lambda_{\gk} \le \lambda_1 < \mu_1 \le \mu_{\gk},\lambda_1 =
\lambda^{<\kappa}_1$ and $(\forall \alpha < \mu_1)(|\alpha|^{< \kappa}
< \mu_1)$, \then \, we define $K_{[\lambda_1,\mu_1]} 
= K^{\gs}_{[\lambda_1,\mu_1)}$ and
  $\gk_1 = \gk_{[\lambda_1,\mu_1]}$ similarly, i.e. $K_{\gk} = \{M \in
  K_{\gk}:\|M\| \in [\lambda_1,\mu_1)\}$ and $\le_{\gk_1} = \le_{\gk}
  \rest K_{\gk_1}$ and $\lambda_{\gk_1} =
  \lambda_1,\mu_{\gk_1} = \mu_1,\kappa_{\gk_1} = \kappa_{\gk}$.
\end{notation}

\begin{definition}
\label{f24}
The embedding $f:N \rightarrow M$ is a ${\gk}$-embedding or a 
$\le_{\gk}$-embedding if its range is the universe of a model 
$N' \le_{\gk} M$, (so $f:N \rightarrow N'$ is an isomorphism onto).
\end{definition}

\begin{claim}
\label{f27}
[$\gk$ is a 2-d.a.e.c.]

\noindent
1) For every $N \in K$ there is a $\kappa_{\gk}$-directed partial order $I$ 
of cardinality $\le \|N\|$ and $\bar M = \langle M_t:t \in
I \rangle$ such that $t \in I \Rightarrow M_t \le_{\gk} N,
\|M_t\| \le \LST({\gk}),I \models s < t \Rightarrow M_s \le_{\gk} M_t$ 
and $N = \bigcup\limits_{t \in I} M_t$. 

\noindent
2) For every $N_1 \le_{\gk} N_2$ we can find $\langle M^\ell_t:t \in
I^* \rangle$ as in part (1) for $N_\ell$ such that $I_1 \subseteq I_2$ and
$t \in I_1 \Rightarrow M^2_t = M^1_t$. 
\end{claim}

\begin{PROOF}{\ref{f27}}
1) As in the proof of \ref{f19}.

\noindent
2) Similarly.
\end{PROOF}

\begin{claim}
\label{f28}
Assume $\lambda_{\gk} \le \lambda_1 = \lambda^{< \kappa}_1 < \mu_1 \le
\mu_{\gk}$ and $(\forall\alpha < \mu_1)(|\alpha|^{< \kappa} < \mu_1)$.

\noindent
1) \Then \, $\gk^*_1 := \gk_{[\lambda_1,\mu_1)}$ as defined in
\ref{f22}(3) is a $(\lambda_1,\mu_1,\kappa_{\gk})$-d.a.e.c.

\noindent
2) For each of the folllowing axioms if $\gk$ satisfies it then so
does $\gk_1$: Ax(III)(d),(IV)(b),(IV)(c),(IV)(d).

\noindent
3) If if addition $\gk$ satisfies Ax(III)(d),(IV)(d), this part (2)
apply to all the axioms in \ref{f2}, \ref{f4}.
\end{claim}

\begin{claim}
\label{f29}
1) If $\gk$ satisfies Ax(IV)(e) \then \, $\gk$ satisfies Ax(III)(e)
provided that $\mu_{\gk}$ is regular or at least the relevant $I$ has
cardinality $< \cf(\mu_{\gk})$. 

\noindent
2) If Ax(III)(d),(IV)(d) we can waive $\mu_{\gk}$ is regular.
\end{claim}

\begin{PROOF}{\ref{f29}}
We prove this by induction on $|I|$.
\bigskip

\noindent
\underline{Case 1}:  $I$ is finite.

So there is $t^* \in I$ such that $t \in I \Rightarrow t \le_I t^*$,
so this is trivial.
\bigskip

\noindent
\underline{Case 2}:  $I$ is countable.

So we can find a sequence $\langle t_n:n < \omega \rangle$ such that
$t_n \in I,t_n \le_I t_{n+1}$ and $s \in I \Rightarrow 
\bigvee\limits_{n < \omega} s \le_I t_n$.  Now we can apply the axiom to
$\langle M_{t_n},h_{t_{n,t_m}}:m < n< \omega \rangle$.
\bigskip

\noindent
\underline{Case 3}: $I$ uncountable.

First, we can find an increasing continuous sequence 
$\langle I_\alpha:\alpha < |I| \rangle$ such that
$I_\alpha \subseteq I$ is directed of cardinality $\le |\alpha| +
\aleph_0$ and let $I_{|I|} = I = \cup\{I_\alpha:\alpha < |I|\}$.

Second, by the induction hypothesis for each $\alpha < |I|$ we 
choose $N_\alpha,\bar h^\alpha = \langle h_{\alpha,t}:t \in
I_\alpha \rangle$ such that:
\mn
\begin{enumerate}
\item[$(a)$]   $N_\alpha \in {\gk}^{\gs}_{\le \chi}$
\sn
\item[$(b)$]   $h_{\alpha,t}$ is a $\le_{\gk}$-embedding of
$M_t$ into $N_\alpha$
\sn
\item[$(c)$]   if $s <_I t$ are in $I_\alpha$ then $h_{\alpha,s} =
h_{\alpha,t} \circ h_{t,s}$
\sn
\item[$(d)$]  if $\beta < \alpha$ then $N_\beta \le_{\gk}
N_\alpha$ and $t \in I_\beta \Rightarrow h_{\alpha,t} = h_{\beta,t}$.
\end{enumerate}
\mn
For $\alpha = 0$ use the induction hypothesis.

For $\alpha$ a limit ordinal by Ax(III)(a) there is $N_\alpha$ s
required as $I_\alpha = \cup\{I_\beta:\beta < \alpha\}$ there are no
new $h_t$'s; well we have to check $\Sigma\{\|N_\beta\|:\beta <
\alpha\} < \mu_{\gk}$ but as we assume $\mu_{\gk}$ is regular this
holds.

For $\alpha = \beta +1$, by the induction hypothesis there is
$(N'_\alpha,\bar g^\alpha)$ which is a limit of $\langle M_s,h_{s,t}:s
\le_{I_\alpha} t\rangle$.  Now apply Ax(IV)(e); well is the directed
system version with $\langle M_s,h_{s,t}:s \le_{I_\beta}
t\rangle,(N'_\alpha,\bar g_\alpha),(N_\beta,\langle h_s:s \in
I_\beta\rangle$ here standing for $\bar M,N_1,N_2$ there.

So there are $N_\alpha,f^\alpha_s(s \in I_\beta)$ such that $N_\beta
\le_{\gk} N_\alpha$ and $s \in I_\beta \Rightarrow f^\alpha_s \cdot
g_s = h_s$.  Lastly, for $s \in I_\alpha \backslash I_\beta$ we choose
$h_s = f^\alpha_s \circ g_s$, so we are clearly done.

\noindent
2) Easy by \ref{f19} or \ref{f28}.
\end{PROOF}
\bigskip

\subsection {Basic Notions} \label{5B}\
\bigskip

As in \cite[\S1]{Sh:600}, we now recall the definition of orbital types
(note that it is natural to look at types only over models which are
amalgamation bases recalling Ax(IV)(b) implies every $M \in K_{\gk}$ is). 
\begin{definition}
\label{f31}

\noindent
1) For $\chi \in [\lambda_{\gk},\mu_{\gk})$ and 
$M \in K_\chi$ we define ${\cS}(M)$ as $\ortp(a,M,N):M \le_{\gk} 
N \in K_\chi \text{ and } a \in N\}$ where
$\ortp(a,M,N) = (M,N,a)/{\cE}_M$ where ${\cE}_M$ is the 
transitive closure of ${\cE}^{\at}_M$, and the 
two-place relation ${\cE}^{\at}_M$ is defined by:

\begin{equation*}
\begin{array}{clcr}
(M,N_1,a_1) {\cE}^{\at}_M&(M,N_2,a_2) \text{\Iff \, } M \le_{\gk}
N_\ell,a_\ell \in N_\ell,\|M\| \le \|N_\ell\| = \|M\|^{< \kappa} \\
  &\text{ for } \ell=1,2 \text{ and } \text{ there is }
N \in K_\chi \text{ and } \le_{\gk} \text{-embeddings} \\
  &f_\ell:N_\ell \rightarrow N \text{ for } \ell = 1,2 \text{ such that:} \\
  &f_1 \restriction M = \id_M = f_2 \restriction M \text{ and }
f_1(a_1) = f_2(a_2).
\end{array}
\end{equation*}

\mn
$(\text{of course } M \le_{\gk} N_1,M \le_{\gk} N_2 \text{ and }
a_1 \in N_1,a_2 \in N_2)$
\medskip

\noindent
2) We say ``$a$ realizes $p$ in $N$" \when \, $a \in N,p \in {\cS}(M)$
and letting $\chi = \|M\|^{< \kappa}$
for some $N' \in K_\chi$ we have $M \le_{\gk} N' \le_{\gk} N$ and 
$a \in N'$ and $p = \ortp(a,M,N')$; so $M,N' \in K_\chi$ but possibly
$N \notin K_\chi$. 

\noindent
3) We say ``$a_2$ strongly 
\footnote{note that ${\cE}^{\at}_M$ is not an
equivalence relation and certainly in general is not $\cE_M$}
realizes $(M,N^1,a_1)/{\cE}^{\text{at}}_M$ in $N$" \when \, for some
$N^2$ we have $M \le_{\gk} N^2 \le_{\gk} N$ and
$a_2 \in N^2$ and $(M,N^1,a_1)\,{\cE}^{\at}_M \,(M,N^2,a_2)$. 

\noindent
4) We say $M_0$ is a $\le_{\gk[\chi_0,\chi_1)}$-amalgamation
base if this holds in ${\gk}_{[\chi_0,\chi_1)}$, see below.

\noindent
4A) We say $M_0 \in {\gk}$ is an amalgamation base or 
$\le_{\gk}$-amalgamation base \when \,: for every
$M_1,M_2 \in {\gk}$ and $\le_{\gk}$-embeddings
$f_\ell:M_0 \rightarrow M_\ell$ (for $\ell = 1,2$) there is $M_3 \in
{\gk}_\lambda$ and 
$\le_{\gk}$-embeddings $g_\ell:M_\ell \rightarrow
M_3$ (for $\ell=1,2$) such that $g_1 \circ f_1 = g_2 \circ f_2$. 

\noindent
5) We say ${\gk}$ is stable in $\chi$ \when \,:
\mn
\begin{enumerate}
\item[$(a)$]  $\lambda_{\gk} \le \chi < \mu_{\gk}$
\sn
\item[$(b)$]   $M \in K_\chi \Rightarrow |{\cS}(M)| \le \chi$
\sn
\item[$(c)$]  $\chi \in \Car_{\gk}$ which means
$\chi = \chi^{< \kappa}$ or the conclusion of \ref{f19} holds
\sn
\item[$(d)$]  $\gk_\chi$ has amalgamation.
\end{enumerate}
\mn
6) We say $p=q \restriction M$ if $p \in {\cS}(M),q \in {\cS}(N),
M \le_{\gk} N$ and for some $N^+,N \le_{\gk} N^+$ and $a \in N^+$ we
have $p = \ortp(a,M,N^+),q = \ortp(a,N,N^+)$; note that $p
\restriction M$ is well defined if $M \le_{\gk} N,p \in {\cS}(N)$. 

\noindent
7) For finite $m$, for $M \le_{\gk} N,\bar a \in {}^m N$ we can define
$\ortp(\bar a,N,N)$ and ${\cS}^m(M)$ similarly and ${\cS}^{< \omega}(M)
= \bigcup\limits_{m < \omega} {\cS}^m(M)$, (but we shall not use this in any
essential way, hence we choose ${\cS}(M) = {\cS}^1(M)$.)
\end{definition}

\begin{definition}
\label{f34}
1) We say \underline{$N$ is $\lambda$-universal above or over $M$} if
for every $M',M \le_{\gk} M' \in K^{\gk}_\lambda$, 
there is a $\le_{\gk}$-embedding of $M'$
into $N$ over $M$.  If we omit $\lambda$ we mean $\rnd_{\gk}(\lambda)
= \min(\Car)_{\gk} \backslash \lambda$) so 
$\le \|N\|^{< \kappa({\gk})}$; clearly this implies that $M$ is a 
$\le_{\gk_{[\chi_0,\chi_1]}}$-amalgamation base where $\chi_0 =
\|M\|,\chi_1 = (\|N\|^{< \kappa})^+$.
 
\noindent
2)  $K^3_{\gk} = \{(M,N,a):M \le_{\gk} N,a \in N \backslash M$ and
$M,N \in K_{\gk}\}$, with the partial order $\le = \le_{\gk}$ defined by
$(M,N,a) \le (M',N',a')$ iff $a = a',M \le_{\gk} M'$ and $N
\le_{\gk} N'$.  We say $(M,N,a)$ is minimal if $(M,N,a) \le (M',N_\ell,a)
\in K^3_{\gk}$ for $\ell =1,2$ implies $\ortp(a,M',N_1) = 
\ortp(a,M',N_2)$ moreover, $(M',N_1,a) {\cE}^{\at}_\lambda (M',N_2,a)$,
(not needed if every $M' \in K_\lambda$ is an amalgamation basis).

\noindent
2A) $K^{3,\gk}_\lambda$ is defined similarly using
$\gk_{[\lambda,\rnd(\lambda)]}$.  
\end{definition}

\noindent
Generalizing superlimit, we have more than one reasonable choice.
\begin{definition}
\label{f37}
1) For $\ell=1,2$ we say $M^* \in K^{\gk}_\lambda$ is 
\underline{superlimit}$_\ell$ or $(\lambda,\ge \kappa)$-superlimit$_\ell$
\when \, clause (c) of \ref{f31}(5)  
and:
\mn
\begin{enumerate}
\item[$(a)$]   it is universal, (i.e., every $M \in K^{\gk}_\lambda$ 
can be properly $\le_{\gk}$-embedded into $M^*$), and
\sn
\item[$(b)$]   \underline{Case 1}:  $\ell=1$ if 
$\langle M_i:i \le \delta \rangle$ is $\le_{\gk}$-increasing, 
$\cf(\delta) \ge \kappa,\delta < \lambda^+$ 
and $i < \delta \Rightarrow M_i \cong M^*$ then $M_\delta \cong M^*$
\sn
\item[${{}}$]   \underline{Case 2}:  $\ell=2$ if $I$ is a 
$(< \kappa)$-directed partial order of cardinality $\le \chi$, 
$\langle M_t:t \in I\rangle$ is
$\le_{\gk}$-increasing and $t \in I \Rightarrow M_t \cong M^*$
then $\cup\{M_t:t \in I\} \cong M^*$.
\end{enumerate}
\mn
2) $M$ is $\lambda$-saturated above $\mu$ \If \, $\|M\| \ge \lambda >
\mu \ge \LST({\gk})$ and: $N \le_{\gk} M,\mu \le \|N\| < \lambda,p \in
{\cS}_{\gk}(N)$ implies $p$ is strongly realized in $M$.  Let ``$M$ is
$\lambda^+$-saturated" mean that ``$M$ is $\lambda^+$-saturated above
$\lambda$" and $K(\lambda^+$-saturated) 
$= \{M \in K:M$ is $\lambda^+$-saturated$\}$ 
and ``$M$ is saturated
$\ldots$" mean ``$M$ is $\|M\|$-saturated $\ldots$". 
\end{definition}

\begin{definition}
\label{f40}
1) We say \underline{$N$ is $(\lambda,\sigma)$-brimmed over $M$} if we
can find a sequence $\langle
M_i:i < \sigma \rangle$ which is $\le_{\gk}$-increasing
semi-continuous, $M_i \in K_\lambda,M_0 = M,M_{i+1}$ is 
$\le_{\gk}$-universal over $M_i$ and $\bigcup\limits_{i < \sigma} M_i
= N$.  We say $N$ is $(\lambda,\sigma)$-brimmed
over $A$ if $A \subseteq N \in K_\lambda$ and we can find $\langle M_i:i <
\sigma \rangle$ as in part (1) such that $A \subseteq M_0$; if $A =
\emptyset$ we may omit ``over $A$".

\noindent
2) We say $N$ is $(\lambda,*)$-brimmed over $M$ if for some $\sigma
\in [\kappa,\lambda),N$ is $(\lambda,\sigma)$-brimmed over $M$.  We say $N$ is
$(\lambda,*)$-brimmed if for some $M,N$ is $(\lambda,*)$-brimmed over $M$. 

\noindent
3) If $\alpha < \lambda^+$ let ``$N$ is $(\lambda,\alpha)$-brimmed over
$M$" mean $M \le_{\gk} N$ are from $K_\lambda$ and $\cf(\alpha) \ge
\kappa \Rightarrow N$ is $(\lambda,\cf(\alpha))$-brimmed over $M$.
\end{definition}

\noindent
Recall
\begin{claim}
\label{f43}
1) If ${\gk}$ is a $(\mu,\lambda,\kappa)$-{\rm d.a.e.c.}
with amalgamation, stable in $\chi$ and $\sigma = \cf(\sigma)$ so $\chi 
\in [\lambda,\mu)$, \then \, for every $M \in K^{\gk}_\chi$ there is 
$N \in K^{\gk}_\lambda$ universal over $M$
which is $(\chi,\sigma)$-brimmed over $M$ 
(hence is $S^\chi_\sigma$-limit, see \cite{Sh:88r}, not used).

\noindent
2) If $N_\ell$ is $(\chi,\theta)$-brimmed over $M$ for $\ell
   =1,2$, and $\kappa \le \theta = \cf(\theta) \le \chi^+$
\then \, $N_1,N_2$ are isomorphic over $M$. 

\noindent
3) If $M_2$ is $(\chi,\theta)$-brimmed over $M$ and $M_0 \le_{\gs}
   M_1$ then $M_2$ is $(\chi,\theta)$-brimmed over $M_0$.
\end{claim}

\begin{PROOF}{\ref{f43}}  
Straightforward for part (1); recall clause (c) of Definition
\ref{f43}(5).  

\noindent
2),3) As in \cite{Sh:600}.
\end{PROOF}
\bn
\centerline {$* \qquad * \qquad *$}
\bigskip

\subsection {Liftings} \label{5C}\
\bigskip

Here we deal with lifting, there are two aspects.  First, if
$\gk^1,\gk^1$ agree in $\lambda$ they agree in every higher cardinal.
Second, given $\gk$ we can find $\gk_1$ with $\mu_{\gk_1} =
\infty,(\gk_1)_\lambda = \gk_\lambda$.

\begin{theorem}
\label{f46}
1) If ${\gk}^\ell$ is a $(\mu,\lambda,\kappa)$-a.e.c. for $\ell=1,2$ 
and ${\gk}^1_\lambda = {\gk}^2_\lambda$ then ${\gk}^1 = \gk^2$.

\noindent
2) If ${\gk}_\ell$ is a $(\mu_\ell,\lambda,\kappa)$-d.a.e.c. for $\ell=1,2$
and $\gk^1$ satisfies Ax(IV)(d) and $\mu_1 \le \mu_2$ and
${\gk}^1_\lambda = \gk^2_\lambda$ \then \, $\gk_1 = \gk_2[\lambda,\mu_1)$.
\end{theorem}

\begin{PROOF}{\ref{f46}}
By \ref{f27}.
\end{PROOF}

\begin{theorem}
\label{f49}
\underline{The lifting-up Theorem}

\noindent
1) If ${\gk}_\lambda$ is a $(\lambda^+,\lambda,\kappa)$-{\rm
  d.a.e.c.}$^\pm$ \then \, the pair $(K',\le_{{\gk}'})$ defined below is an 
$(\infty,\lambda,\kappa)$-{\rm d.a.e.c.}$^+$ where we define
\mn
\begin{enumerate}
\item[$(A)$]   $K'$ is the class of $M$ such that $M$ is a 
$\tau_{{\gk}_\lambda}$-model, and for some $I$ and $\bar M$ we have
\sn
\begin{enumerate}
\item[$(a)$]   $I$ is a $\kappa$-directed partial order
\sn
\item[$(b)$]   $\bar M = \langle M_s:s \in I \rangle$
\sn
\item[$(c)$]  $M_s \in K_\lambda$
\sn
\item[$(d)$]  $I \models s < t \Rightarrow M_s \le_{{\gk}_\lambda} M_t$
\sn
\item[$(e)$]   if $J \subseteq I$ has cardinality $\le \lambda$
and is $\kappa$-directed and $M_J$ is the $\gk$-union of
$\langle M_t:t \in J\rangle$, see Definition \ref{f4}, 
\then \, $M_J$ is a submodel of $M$
\sn
\item[$(f)$]   $M = \cup\{M_J:J \subseteq I$ is $\kappa$-directed of
  cardinality $\le \lambda\}$, i.e. both for the universe and for the
  relations and functions.
\end{enumerate}
\sn
\item[$(A)'$]  we call such $\langle M_s:s \in I \rangle$ a witness 
for $M \in K'$, we call it reasonable if $|I| \le \|M\|^{< \kappa}$
\sn
\item[$(B)$]   $M \le_{{\gk}'} N$ iff for some $I,J,\bar M$ we have
\sn
\begin{enumerate}
\item[$(a)$]   $J$ is a $\kappa$-directed partial order
\sn
\item[$(b)$]   $I \subseteq J$ is $\kappa$-directed
\sn
\item[$(c)$]   $\bar M =  \langle M_s:s \in J \rangle$ and is a
$\le_{{\gk}_\lambda}$-increasing
\sn
\item[$(d)$]   $\langle M_s:s \in J\rangle$ is a witness for $N \in K'$
\sn
\item[$(e)$]   $\langle M_s:s \in I\rangle$ is a witness for $M \in K'$.
\end{enumerate}
\sn
\item[$(B)'$]  We call such $I, \langle M_s:s \in J \rangle$ witnesses for $M 
\le_{{\gk}'} N$ or say $(I,J,\langle M_s:s \in J \rangle)$ witness
$M \le_{{\gk}'} N$. 
\end{enumerate}
\mn
2) For the other axioms we have implications.
\end{theorem}

\begin{PROOF}{\ref{f49}}
The proof of part (2) is straightforward
so we concentrate on part (1).  So let us check the axioms one by one.
\bigskip

\noindent
\underline{AxO(a),(b),(c) and (d)}:  $K'$ is a class of 
$\tau$-models, $\le_{{\gk}'}$ a two-place relation on $K,
K'$ and $\le_{{\gk}'}$ are closed under
isomorphisms and $M \in K' \Rightarrow \|M\| \ge \lambda$, etc.

\noindent
[Why?  trivially.]
\bigskip

\noindent
\underline{AxI(a)}:  If $M \le_{{\gk}'} N$ then $M \subseteq N$. 

\noindent
[Why?  We use smoothness for $\kappa$-directed unions, i.e. Ax(IV)(x).]
\bigskip

\noindent
\underline{AxII(a),(b),(c)}:  

We prove the first, the others are easier.
\bigskip

\noindent
\underline{Ax II(a)}:  $M_0 \le_{{\gk}'} M_1 \le_{{\gk}'} M_2$ 
implies $M_0 \le_{{\gk}'} M_2$ and $M \in K' \Rightarrow M \le_{{\gk}'} M$. 

\noindent
[Why?  The second phrase is trivial.  For the first phrase let for $\ell \in
\{1,2\}$ the $\kappa$-directed partial orders $I_\ell \subseteq J_\ell$ and
$\bar M^\ell = \langle M^\ell_s:s \in J_\ell \rangle$ witness 
$M_{\ell-1} \le_{{\gk}'} M_\ell$.

We first observe
\mn
\begin{enumerate}
\item[$\boxdot$]    if $I$ is a $\kappa$-directed partial order,
$\langle M^\ell_t:t \in I\rangle$ is a 
$\le_{{\gk}_\lambda}$-system witnessing 
$M_\ell \in K'$ for $\ell=1,2$ and $t \in I \Rightarrow 
M^1_t \le_{{\gk}_\lambda} M^2_t$ \then \ $M_1 \le_{\gk} M_2$.
\end{enumerate}
\mn
[Why?  Let $I_1$ be the partial order with set of elements $I \times
\{1\}$ ordered by $(s,1) \le_{I_1} (t,1) \Leftrightarrow s \le_I t$.
Let $I_2$ be the partial order with set of elements $I \times \{1,2\}$
ordered by $(s_1,\ell_1) \le_{I_2} (s_2,\ell_2) \Leftrightarrow s_1
\le_I s_2 \wedge \ell_1 \le \ell_2$.  Clearly $I_1 \subseteq I_2$ are
both $\kappa$-directed.

Let $M_{(s,1)} = M^1_s,M_{(s,2)} = M^2_s$, so clearly $\bar M =
\langle M_t:t \in I_\ell\rangle$ is a $\le_{\gk}-\kappa$-directed system
witnessing $M_\ell \in K'$ for $\ell=1,2$ and $(I_1,I_2,\bar M)$
witness $M_1 \le_{{\gk}'} M_2$, so we are done.]

Without loss of generality $J_1,J_2$ are pairwise disjoint.  Let $\chi
= (|J_1| + |J_2|)^{< \kappa}$ so $\lambda \le \chi < \mu$ and let

\begin{equation*}
\begin{array}{clcr}
{\cU} := \{u:&u \subseteq J_1 \cup J_2 \text{ has cardinality } \le
\lambda \text{ and } u \cap I_\ell \\
  &\text{is } \kappa\text{-directed under } \le_{I_\ell} \text{ for }
  \ell= 1,2 \text{ and } u \cap J_\ell \\
  &\text{is } \kappa\text{-directed under } \le_{J_\ell} \text{ for }
  \ell=1,2 \text{ and} \\
  &\cup\{|M^2_t|:t \in I_2\} = \cup\{|M^1_t:t \in J_1\}\}.
\end{array}
\end{equation*}

\mn
Let $\langle u_\alpha:\alpha < \alpha^*\rangle$ list ${\cU}$, and
we define a partial order $I$:
\mn
\begin{enumerate}
\item[$(a)'$]   its set of elements is $\{\alpha < \alpha^*$: for no
$\beta < \alpha$ do we have $u_\alpha \subseteq u_\alpha\}$
\sn
\item[$(b)'$]   $\alpha \le_I \beta$ iff $u_\alpha \subseteq u_\beta
\wedge \alpha \in I \wedge \beta \in I$.
\end{enumerate}
\mn
Note that the set $I$ may have $\card(\sum\limits_{i < \delta}
\|M_i\|)^\lambda$ which may be $> \mu_{\gk}$.

As in the proof of \ref{f19}, $I$ is $\kappa$-directed.

For $\ell = 0,1,2$ and $\alpha \in I$ let $M_{\ell,\alpha}$ be
\mn
\begin{enumerate}
\item[$(a)$]    $\le_{\gk}$-union of $\langle M^\ell_t:t \in u_\alpha
\cap I_1\rangle$ if $\ell=0$
\sn
\item[$(b)$]    the $\le_{\gk}$-union of the
  $\le_{\gk_\lambda}$-directed system $\langle M^1_t:t \in J_1\rangle$,
equivalently the $\le_{{\gk}_\lambda}$-directed system of $\langle
M^2_t:t \in I_2\rangle$, when $\ell=1$
\sn
\item[$(c)$]    the $\le_{\gk}$-union of the 
$\le_{{\gk}_\lambda}$-directed system $\langle M^2_t:t \in J_2\rangle$ when
$\ell=2$.
\end{enumerate}
\mn
Now
\mn
\begin{enumerate}
\item[$(*)_1$]    if $\ell=0,1,2$ and $\alpha \le_I \beta$ then
$M^\ell_\alpha \le_{{\gk}_\lambda} M^\ell_\beta$
\sn
\item[$(*)_2$]   if $\alpha \in I$ then $M^0_\alpha 
\le_{{\gk}_\lambda} M^1_\alpha \le_{{\gk}_\lambda} M^1_\alpha$
\sn
\item[$(*)_3$]    $\langle M_{\ell,\alpha}:\alpha \in I\rangle$ is a
witness for $M_\ell \in K'$
\sn
\item[$(*)_4$]     $M_{0,\alpha} \le_{{\gk}_\lambda}
M_{2,\alpha}$ for $\alpha \in I$.
\end{enumerate}
\mn
Together by $\boxdot$ we get that $M_0 \le_{{\mathfrak k}'} M_2$ as required.
\bigskip

\noindent
\underline{Ax III(a)}:  In general.

Let $(I_{i,j},J_{i,j},\bar M^{i,j})$ witness $M_i \le_{{\gk}'}
M_j$ when $i \le j < \delta$ and \wilog \, $\langle J_{i,j}:i < j <
\delta\rangle$ are pairwise disjoint.  Let ${\cU}$ be the family of
$u$ such that for some $v \in [\delta]^{\le \lambda}$,
\mn
\begin{enumerate}
\item[$(a)$]    $v \subseteq \delta$ has cardinality $\le \lambda$ and
  has order type of cofinality $\ge \kappa$
\sn
\item[$(b)$]  $u \subseteq \cup\{J_{i,j}:i<j$ are from $v\}$ has
cardinality $\le \lambda$ and
\sn
\item[$(b)$]    for $i<j$ from $v$ the set $u \cap J_{i,j}$ is
$\kappa$-directed under $\le_{J_{i,j}}$ and $u \cap I_{i,j}$ is
$\kappa$-directed under $\le_{I_{i,j}}$
\sn
\item[$(c)$]    if $i_0 \le i_1 \le i_2$ then
  $\cup\{M^{i(0),i(1)}_{t,s}:s \in u \cap J_{i(0),i(1)}\} =
  \cup\{M^{i(1),i(2)}_s:s \in u \cap I_{i(1),i(2)}\}$
\sn
\item[$(d)$]    if $i(0) \le i(1) \le i(2)$ are from $v$ then
  $\cup\{M^{i(0),i(1)}_{t,s}:s \in u \cap J_{i(0),i(1)}\} =
  \cup\{M^{i(1),i(2)}_s:s \in u \cap I_{i(1),i(2)}\}$
\sn
\item[$(e)$]    if $i(0) \le k(0) \le j(1)$ and $i(1) \le j(1)$ are
  from $u$ then $\cup\{M^{i(0),j(0)}_s:s \in u \cap J^{i(0),j(0)}_s\} =
\cup\{M^{i(1),j(1)}_s:s \in u \cap J_{i(1),j(1)}\}$.
\end{enumerate}
\mn
Let the rest of the proof be as before.
\bigskip

\noindent
\underline{Ax(IV)(a)}:

Similar, but $\cU = \{u \subseteq I:u$ has cardinality $\le \lambda$
and is $\kappa$-directed$\}$.
\bigskip

\noindent
\underline{Ax(III)(d)}:

Assuming $\gk$ satisfies Ax(III)(d).  Similar.
\bigskip

\noindent
\underline{Ax(IV)(d)}:

Assuming $\gk$ satisfiefs Ax(IV)(d). Similar.
\bigskip

\noindent
\underline{Axiom V}:  Assume $N_0 \le_{{\gk}'} M$ and $N_1 \le_{{\gk}'} M$.

If $N_0 \subseteq N_1$, then $N_0 \le_{{\gk}'} N_1$. 

\noindent
[Why?  Let $(I_\ell,J_\ell,\langle M^\ell_s:s \in J_\ell \rangle)$ 
witness $N_\ell \le_{\gk} M$ for $\ell=0,1$; \wilog \, $J_0,J_1$ are
disjoint.

Let 

\begin{equation*}
\begin{array}{clcr}
\cU := \{ u \subseteq J_0 \cup J_1:& |u| \le \lambda \text{ and } u \cap
  J_\ell \text{ is } \kappa\text{-directed} \\
  &\text{ and } u \cap I_\ell \text{ is } \kappa\text{-directed for }
  \ell=0,1 \text{ and} \\
  &\cup\{|M^0_s|:s \in u \cap J_0\} = \cup\{|M^0_s|:s \in u \in
  J_1\}\}.
\end{array}
\end{equation*}

\mn
For $u \in \cU$ let
\mn
\begin{enumerate}
\item[$\bullet$]  $M_u = M \rest \cup\{(M^\ell_s:s \in u \cap
  J_\ell\}$ for $i=0,1$
\sn
\item[$\bullet$]  $N_{\ell,u} = N_\ell \rest \{(M^\ell_s):s \in u \cap
  I_\ell\}$.
\end{enumerate}
\mn
Let
\mn
\begin{enumerate}
\item[$(*)$]  $(a) \quad (\cU,\subseteq)$ is $\kappa$-directed
\sn
\item[${{}}$]  $(b) \quad N_{\ell,u} \le_{\gk} M$
\sn
\item[${{}}$]  $(c) \quad M_{\ell,u} \le_{\gk} M_{\ell,v}$ when $u
  \subseteq v$ are from $\cU$ and $\ell=0,1$
\sn
\item[${{}}$]  $(d) \quad M_{0,u} \le_{\gk} M_{1,u}$
\sn
\item[${{}}$]  $(e) \quad N_|ell = \cup\{N_{\ell,u}:u \in \cU\}$; as
  in ?
\end{enumerate}
\mn
By $\boxdot$ above we are done.
\bigskip

\noindent
\underline{Axiom VI}:  LST$({\gk}') = \lambda$. 

\noindent
[Why?  Let $M \in K',A \subseteq M,|A| + \lambda \le \chi < \|M\|$ and let
$\langle M_s:s \in I \rangle$ witness $M \in K'$; \wilog \, $|A| =
  \chi^{< \kappa}$.  
Now choose a directed $I \subseteq J$ of cardinality $\le |A| =
\chi^{< \kappa}$ such that $A \subseteq M' =: 
\bigcup\limits_{s \in I} M_s$ and so $(I,J,\langle M_s:s \in J
\rangle)$ witnesses $M' \le_{{\gk}'} M$, so 
as $A \subseteq M'$ and $\|M'\| \le |A| + \mu$ we are done.] 
\end{PROOF}

\noindent
Also if two such d.a.e.c.'s have some cardinal in common then we can
put them together.
\begin{claim}
\label{f52}
Let $\iota \in \{1,2,3\}$ and assume $\lambda_1 < \lambda_2 <
\lambda_3$ and
\mn
\begin{enumerate}
\item[$(a)$]   ${\gk}^1$ is an
  $(\lambda^+_2,\lambda,\kappa)-2$-d.a.e.c., $K^1 = K^1_{\ge \lambda}$
\sn
\item[$(b)$]   ${\gk}^2$ is a $(\lambda_3,\lambda_2,\kappa)-\iota$-d.a.e.c.
\sn
\item[$(c)$]   $K^{\gk^1}_{\lambda_2} = K^{\gk^2}_{\lambda_2}$ 
and $\le_{{\gk}^2} \rest K^{\gk^2}_{\lambda_2} = 
\le_{{\gk}^1} \restriction K^{\gk^1}_{\lambda_2}$
\sn
\item[$(d)$]   we define ${\gk}$ as follows: 
$K_{\gk} = K_{\gk} \cup K_{\gk_2},M \le_{\gk} N$
iff $M \le_{{\gk}^1} N$ or $M \le_{{\gk}^2} N$ or for some $M',M 
\le_{{\gk}^1} M' \le_{{\gk}^2} N$.
\end{enumerate}
\mn
\Then \ ${\gk}$ is an $(\lambda_3,\lambda_1,\kappa)-\iota$-d.a.e.c.
\end{claim}

\begin{PROOF}{\ref{f52}}
Straightforward.  E.g. 
\bigskip

\noindent
\underline{Ax(III)(d)}:  So $\langle M_s:s \in I\rangle$ is
$\le_{\gs}-\kappa$-directed system.

If $\|M_s\| \ge \lambda_2$ for some $\lambda$, use $\langle M_s:s \le
t \in I\rangle$ and clause (b) of the assumption.  If $\cup\{M_s:s \in
I\}$ has cardinality $\le \lambda_2$ use clause (a) in the
assumption.  If neither one of them holds, recall $\lambda_2 =
\lambda^{< \kappa}_2$ by clause (b) of the assumption, and let

\[
\\cU = \{u \subseteq I:|u| \ le \lambda_2,u \text{ is }
\kappa\text{-directed (in I), and } \cup\{M_s:s \in u\} \text{ has
  cardinality } \lambda\}.
\]

\mn
Easily $(\cU,\subseteq)$ is $\lambda_2$-directed, for $u \in J$ let
$M_u$ be the $\le_{\gs}$-union of $\langle M_s:s \in u\rangle$.  Now
by clause (a) of the assumption
\mn
\begin{enumerate}
\item[$(*)_1$]  $M_u \in K^{\gk^1}_{\lambda_2} = K^{k^*}_{\lambda_2}$
\sn
\item[$(*)_2$]  if $u_1 \subseteq v$ are from $\cU$ then 
$M_u \le_{\gk^1} M_v,M_u \le_{\gk^2} M_v$.
\end{enumerate}
\mn
Now use clause (b) of the assumption.
\bigskip

\noindent
\underline{Axiom V}:  We shall use freely
\mn
\begin{enumerate}
\item[$(*)$]   ${\gk}^2_\lambda = {\gk}^2$ and ${\gk}^1_\lambda 
= {\gk}^1$.
\end{enumerate}
\mn
So assume $N_0 \le_{\gk} M,N_1 \le_{\gk} M,N_0 \subseteq N_1$. 

Now if $\|N_0\| \ge \lambda_1$ use assumption (b), so we can assume
$\|N_0\| < \lambda_1$. If $\|M\| \le \lambda_1$ we can use assumption (a) so we
can assume $\|M\| > \lambda_1$ and by the definition of $\le_{\gk}$ there
is $M'_0 \in K^{\gk^1}_{\lambda_1} = K^{\gk^2}_{\lambda_1}$ such that 
$N_0 \le_{\gk^1} M'_0 \le_{\gk^2} M$.  First assume 
$\|N_1\| \le \lambda_1$, so we can find 
$M'_1 \in K^{\gk^1}_{\lambda_1}$ such that
$N_1 \le_{{\gk}^1} M'_1 \le_{{\gk}^2} M$ (why?  if $N_1 \in 
K^{\gk^1}_{< \lambda_1}$, by the definition of 
$\le_{\gk}$ and if $N_1 \in K^{\gk^1}_{\lambda_1}$ just 
choose $M'_1 = N_1$).  Now we can by assumption (b) 
find $M'' \in K^{\gk^1}_{\lambda_1}$ such that $M'_0 \cup M'_1 \subseteq M'' 
\le_{\gk} M$, hence by assumption (b) (i.e. AxV for ${\gk}^2$)
we have $M'_0 \le_{\gk} M'',M'_1 \le_{\gk} M''$.
As $N_0 \le_{\gk} M'_0 \le_{\gk} M'' \in K^{\gk}_{\le \lambda_1}$ 
by assumption (a) we have $N_0 \le_{\gk} M''$, and similarly 
we have $N_1 \le_{\gk} M''$.  So
$N_0 \subseteq N_1,N_0 \le_{\gk} M'',N_1 \le_{\gk} M'$ so by 
assumption (b) we have $N_0 \le_{\gk} N_1$.

We are left with the case $\|N_1\| > \lambda$, by assumption (b) there is
$N'_1 \in K_{\lambda_1}$ such that $N_0 \subseteq N'_1 \le_{\gk^2} N_2$.  By
assumption (b) we have $N'_1 \le_{\gk} M$, so by the previous paragraph
we get $N_0 \le_{\gk} N'_1$, together with the previous sentence we have
$N_0 \le_{{\gk}^1} N'_1 \le_{{\gk}^2} N_1$ so by the definition of
$\le_{\gk}$ we are done. 
\end{PROOF}

\begin{definition}
\label{f55}
If $M \in K_\chi$ is $(\chi,\ge \kappa)$-superlimit$_1$ let
$K^{[M]}_\chi = \{N \in K_\chi:N \cong M\},{\gK}^{[M]}_\chi =
(K^{[M]}_\chi,\le_{\gk} \restriction K^{[M]}_\chi)$ and ${\gk}
^{[M]}$ is the ${\gk}'$ we get in \ref{f49}(1) for ${\gk}' =
{\gk}^{[M]}_\chi$.
\end{definition}

\begin{claim}
\label{0.34}
1) If ${\gk}$ is an $(\mu,\lambda,\kappa)$-{\rm a.e.c.}, 
$\lambda \le \chi < \mu,
M \in K_\chi$ is $(\chi,\ge \kappa)$-superlimit$_1$ \then \,
${\gk}^{[M]}_\chi$ is a $(\chi^+,\chi,\kappa)$-{\rm d.a.e.c.}

\noindent
2) If in addition ${\gk}$ is a 
$(\mu,\lambda,\kappa)$-{\rm d.a.e.c.}$^\pm$ \then \, ${\gk}^{[M]}_\chi$ is a
  $(\chi^+,\chi,\kappa)$-d.a.e.c.$^\pm$.
\end{claim}
\newpage

\section {II pr frames} \label{6}

\begin{definition}
\label{g2}
For $\iota = 1,2,3,4$.  We say that ${\gs}$ is a good
$(\mu,\lambda,\kappa)-\iota$-frame \when \, ${\gs}$ 
consists of the following objects
satisfying the following condition: $\mu,\lambda,\kappa$ (so we may
write $\mu_{\gs},\lambda_{\gs},\kappa_{\gs}$ but we
usually ignore them defining ${\gs}$) and
\mn
\begin{enumerate}
\item[$(A)$]    ${\gk} = {\gk}_{\gs}$ is a
$(\mu,\lambda,\kappa)-6$-d.a.e.c., so we may write ${\gs}$ instead of
${\gk}$, e.g. $\le_{\gs}$-increasing, etc. and $\chi \in
[\lambda,\mu) \Rightarrow \LST(\chi^{< \kappa})$
\sn
\item[$(B)$]   ${\gk}$ has a $(\lambda,\ge \kappa)$-superlimit model
$M^*$ which \footnote{follows by (C) in fact}  is
not $<_{\gk}$-maximal, i.e.,
\sn
\begin{enumerate}
\item[$(a)$]  $M^* \in K^{\gs}_\lambda$
\sn
\item[$(b)$]  if $M_1 \in K^{\gs}_\lambda$ \then \, for some
$M_2,M_1 <_{\gs} M_2 \in K^{\gs}_\lambda$ and $M_2$ is
isomorphic to $M^*$
\sn
\item[$(c)$] if $\langle M_i:i < \delta \rangle$ is 
$\le_{\gs}$-increasing, $i < \delta \Rightarrow M_i 
\cong M$ and $\cf(\delta) \ge \kappa,\delta < \lambda^+$ \then \, 
$\cup\{M_i:i < \delta\}$ is isomorphic to $M^*$
\end{enumerate}
\sn
\item[$(C)$]    ${\gk}$ has the amalgamation property, the
JEP (joint embedding property), and has no $\le_{\gk}$-maximal 
member; if of $\iota \ge 2, {\gk}$ has primes$^-$ and if $\iota \ge 4,{\gk}$
has primes$^+$ 
\sn
\item[$(D)$]  $(a) \quad {\cS}^{\bs} = {\cS}^{\bs}_{\gs}$ 
(the class of basic types for ${\gk}_{\gs}$) 
is included in 

\hskip25pt  $\bigcup\{{\cS}(M):M \in K_{\gs}\}$ and is 
closed under isomorphisms including 

\hskip25pt  automorphisms; 
for $M \in K_\lambda$ let ${\cS}^{\bs}(M) = {\cS}^{\bs}_{\gs} 
\cap {\cS}(M)$; 

\hskip25pt  no harm in allowing types of finite sequences.
\sn
\item[${{}}$]  $(b) \quad$ if $p \in {\cS}^{\bs}_{\gs}(M)$, 
\then \, $p$ is non-algebraic (i.e., not realized by any 

\hskip25pt $a \in M$).
\sn
\item[${{}}$]  $(c) \quad$ \underline{(density)} 

\hskip25pt if $M \le_{\gk} N$ are from $K_{\gs}$ and 
$M \ne N$, \then \, for some $a \in N \backslash M$ 

\hskip25pt we have $\ortp(a,M,N) \in {\cS}^{\bs}$

\hskip25pt  [intention: examples are: minimal types in 
\cite{Sh:576}, regular types

\hskip25pt  for superstable theories]
\sn
\item[${{}}$]  $(d) \quad$ \underline{bs-stability} 

\hskip25pt ${\cS}^{\bs}(M)$ has cardinality $\le \|M\|^{< \kappa}$ 
for $M \in K_{\gs}$. 
\sn
\item[$(E)$]  $(a) \quad \nonfork{}{}_{} = \nonfork{}{}_{\gs}$ 
is a four place relation  called nonforking with 
$\nonfork{}{}_{}(M_0,M_1,a,M_3)$ implying $M_0 \le_{\gk} M_1 
\le_{\gk} M_3$ are from $K_{\gs},a \in M_3 \backslash M_1$ 
and $\ortp(a,M_0,M_3) \in {\cS}^{\bs}_{\gs}(M_0)$ 
and $\ortp(a,M_1,M_3) \in {\cS}^{\bs}(M_1)$.  
Also $\nonfork{}{}_{}$ is preserved under isomorphisms. 

We also write $\nonforkin{M_1}{a}_{M_0}^{M_3}$ and demand: if $M_0 = M_1 
\le_{\gk} M_3$ both in $K_\lambda$ then: 
$\nonfork{}{}_{}(M_0,M_1,a,M_3)$ is equivalent to 
``$\ortp(a,M_0,M_3) \in {\cS}^{\bs}(M_0)$".  Also we may 
state $\nonforkin{M_1}{a}_{M_0}^{M_3}$ 
``$\ortp(a,M_1,M_3)$ does not fork over $M_0$ (inside $M_3$)" 
(this is justified by clause (b) below).

\hskip25pt  [Explanation: The 
intention is to axiomatize non-forking of types, but we allow
dealing with basic types only.  Note that in
\cite{Sh:576} we know something on minimal types but other types are
something else.]
\sn
\item[${{}}$]  $(b) \quad$ \underline{(monotonicity)}: 

\hskip25pt  if $M_0 \le_{\gk} M'_0 \le_{\gk} M'_1 \le_{\gk} M_1 \le_{\gk} 
M_3 \le_{\gk} M'_3,M_1 \cup \{a\} \subseteq M''_3 \le_{\gk} M'_3$ 
all of them in $K_\lambda$, \then \,
$\nonfork{}{}_{}(M_0,M_1,a,M_3) \Rightarrow 
\nonfork{}{}_{}(M'_0,M'_1,a,M'_3) \Leftrightarrow \nonfork{}{}_{}
(M'_0,M'_1,a,M''_3)$, \underline{so} it is legitimate to 
just say ``$\ortp(a,M_1,M_3)$ does not fork over $M_0$". 

\hskip25pt  [Explanation: non-forking is preserved by decreasing the type,
increasing the basis (= the set over which it does not fork) and
increasing or decreasing 
the model inside which all this occurs.  The same holds for
stable theories only here we restrict ourselves to ``legitimate" types.]
\sn
\item[${{}}$]  $(c) \quad$  \underline{(local character)}: 

\hskip25pt \underline{Case 1}: $\iota=1,2,3$.

If $\langle M_i:i \le \delta\rangle$ is $\le_{\gs}$-semi-continuous 
and $p \in {\cS}^{\bs}(M_\delta)$ and $\cf(\delta) \ge \kappa$ \then
\, for every $\alpha < \delta$ large enough, $p$ does not fork over $M_\alpha$.

\hskip25pt \underline{Case 2}:  $\iota=4$.  

If $I$ is a $\kappa$-directed partial order and $\bar M = \langle M_t:t
\in I\rangle$ is a $\le_{\gs}$-directed system and $M$ is its
$\le_{\gk}$-union and $M \le_{\gs} N$ and
$\ortp(a,M,N) \in {\cS}^{\bs}(M_\delta)$ \then \,
for every $s \in I$ large enough $\ortp(a,M,N)$ does not fork over $M_s$. 

\hskip25pt [Explanation: This is a replacement for $\kappa \ge \kappa_r(T)$; 
if $p \in {\cS}(A)$ then there is a $B \subseteq A$ of cardinality
$< \kappa$ such that $p$ does not fork over $A$.  The case $\iota=2$?
is a very strong demand even for stable first order theories.]  It
means dimensional continuity, i.e. $M_\delta$ is minimal over
$\cup\{M_\alpha:\alpha < \delta\}$ and $\kappa$-saturated models.]
\sn
\item[${{}}$]  $(d) \quad$  \underline{(transitivity)}: 

if $M_0 \le_{\gk} M'_0 \le_{\gk} M''_0 \le_{\gk} M_3$ and
$a \in M_3$ and $\ortp(a,M''_0,M_3)$ does not fork over $M'_0$ and
$\ortp(a,M'_0,M_3)$ does not fork over $M_0$ (all models are in 
$K_\lambda$, of course, and necessarily the three relevant types are in 
${\cS}^{\bs}$), \then \, $\ortp(a,M''_0,M_3)$ does not fork over $M_0$
\sn
\item[${{}}$]  $(e) \quad$  \underline{uniqueness}: 

if $p,q \in {\cS}^{\bs}(M_1)$ do not fork over $M_0 \le_{\gk} M_1$ 
(all in $K_{\mathfrak s}$) and 

$p \restriction M_0 = q \restriction M_0$ 
\then \, $p = q$
\sn
\item[${{}}$]  $(f) \quad$  \underline{symmetry}:  

\underline{Case 1}: $\ell \ge 3$.

If $M_0 \le_{\gs} M_\ell \le_{\gs} M_3$ and
$(M_0,M_\ell,a_\ell) \in K^{3,\pr}_{\gs}$, see clause (j) below for $\ell=1,2$
\then \, $\ortp_{\gs}(a_2,M_1,M_3)$ does not fork over $M_0$ iff
$\ortp_{\gs}(a_1,M_2,M_3)$ does not fork over $M_0$.

\underline{Case 2}:  $\iota=1,2$.

If $M_0 \le_{\gk} M_3$ are in ${\gk}_\lambda$ and for $\ell = 1,2$
we have $a_\ell \in M_3$ 
and $\ortp(a_\ell,M_0,M_3) \in {\cS}^{\bs}(M_0)$, 
\then \, the following are equivalent:
\sn
\begin{enumerate}
\item[${{}}$]   $(\alpha) \quad$ there are $M_1,M'_3$ in $K_{\gs}$ such that
$M_0 \le_{\gk} M_1 \le_{\gK} M'_3$, 

\hskip25pt $a_1 \in M_1,M_3 \le_{\gk} M'_3$ and 
$\ortp(a_2,M_1,M'_3)$ does not fork over $M_0$
\sn
\item[${{}}$]   $(\beta) \quad$ there are $M_2,M'_3$ in $K_\lambda$ such that
$M_0 \le_{\gk} M_2 \le_{\gk} M'_3$,

\hskip5pt $a_2 \in M_2,M_3 \le_{\gk} M'_3$ and $\ortp(a_1,M_2,M'_3)$ 
does not fork over $M_0$. 

[Explanation: this is a replacement to ``$\ortp(a_1,M_0 \cup
\{a_2\},M_3)$ forks over $M_0$ iff $\ortp(a_2,M_0 \cup \{a_1\},M_3)$
forks over $M_0$"; which is not well defined in out context]
\end{enumerate}
\item[${{}}$]  $(g) \quad$  [existence] if $M \le_{\gs} N,p 
\in {\cS}^{\bs}(M)$ \then \, there is $q \in {\cS}^{\bs}(N)$

\hskip25pt a non-forking extension of $p$
\sn
\item[${{}}$]  $(h) \quad$ [continuity] \underline{Case 1}:  $\iota \ge 1$.

If $\langle M_\alpha:\alpha \le \delta\rangle$ is 
$\le_{\gs}$-increasing, $\le_{\gs}$-semi-continuity, $M_\delta = 
\bigcup\limits_{\alpha < \delta} M_\alpha$ which holds if $\cf(\delta)
\ge \kappa$ and $p \in {\cS}(M_\delta)$ and $p \restriction M_\alpha$ 
does not fork over $M_0$ for $\alpha < \delta$ \then \, $p \in 
{\cS}^{\bs}(M_\delta)$ and it does not fork over $M_0$
\sn
\item[${{}}$]   \underline{Case 2}:  $\iota=4$.

Similarly for $\bar M = \langle M_t:t \in I\rangle,I$ directed, $M =
\cup\{M_t:t \in I\}$ is a $\le_{\gs}$-upper bound of $\bar M$
\sn
\item[${{}}$]  $(j) \quad {\gs}$ has $K^{3,\pr}_{\gs}$-primes, see \ref{g26}
\sn
\item[${{}}$]  $(k) \quad$  \underline{Case 1}:  $\iota \ge 1$. 

If $p \in {\cS}^{\bs}_{\gs}(N)$ \then \, $p$ does not fork over $M$
for some $M \le_{\gs} N$ from $K_\lambda$
\sn
\item[${{}}$]   \underline{Case 2}:  $\iota=3,4$.

If $M_\ell (\ell \le 3),a_\ell,p_\ell(\ell=1,2)$ are as in (E)(f). 

\end{enumerate}
\end{definition}

\begin{discussion}
\label{g5} Consider using:
semi-continuous + $\cf(\delta) \ge \kappa$ for
$(E)(c),(E)(x): \cf(\delta) \ge \kappa$ stable  
only if $\chi =
\chi^{< \kappa}$. 
\end{discussion}

\begin{claim}
\label{g8}
1) If $\langle M_i:i < \delta \rangle$ 
is $\le_{\gk}$-increasing, $(\Sigma\{\|M_i\|:i < \delta\}) < \mu$ 
and $p_i \in {\cS}^{\bs}_{\gs}(M_i)$ does 
not fork over $M_0$ for $i < \delta$ and $[i < j
\Rightarrow p_j \restriction M_i = p_i]$ \then \,:
\mn
\begin{enumerate}
\item[$(a)$]    we can find $M_\delta$ such that $i < \delta 
\Rightarrow M_i \le_{\mathfrak k} M_\delta$
\sn
\item[$(b)$]   for any such $M_\delta$, we can find 
$p \in {\cS}_{\gs}(M_\delta)$ such that $\bigwedge\limits_{i < \delta} 
p \restriction M_i = p_i$ and $p$ does not fork over $M_0$
\sn
\item[$(c)$]    $p_\delta$ is unique in clause (b)
\sn
\item[$(d)$]   if $\ell \ge \kappa \wedge \cf(\delta) \ge \kappa$
we can add $M = \cup\{M_\alpha:\alpha < \delta\}$.
\end{enumerate}
\mn
2) Similarly for $\bar M = \langle M_t:t \in I\rangle,I$ directed.
\end{claim}

\begin{PROOF}{\ref{g8}}
1) First choose $M_\delta$ by \ref{g2}, Clause (A).
Second choose $p_\delta \in {\cS}^{\bs}_{\gs}(M_\delta)$, a
non-forking extension of $p_0$, exist by Ax(g) of (E) of \ref{g2}.   
Now $p_\delta \restriction M_i \in {\cS}^{\bs}_{\gs}(M_i)$ does 
not fork over $M_0$ by (b) of (E) of
\ref{g2} and extend $p_0$ so is equal to $p_i$ by (e) of (E).
Third, $p_\delta$ is unique by (E)(e).

\noindent
2) Should be clear, too.  
\end{PROOF}

\begin{definition}
\label{g11}
0) 
We say $\gs$ is full \when \, $\cS^{\bs}_{\gs}(M) = \{p:p \in
   \cS^\varepsilon_{\gk[\gs]}(M)$ is not algebraic for some
   $\varepsilon < \kappa_{\gs}\}$ for every $M \in K_{\gs}$ [compare
   with (2)].  

\noindent
1) Assume $M_\ell \le_{\gs} N$ for $\ell=1,2$ and $p_\ell \in 
{\cS}^{\bs}_{\gs} (M_\ell)$ for $\ell =1,2$.  We say that
$p_1,p_2$ are parallel \when \, some 
$p \in {\cS}^{\bs}_{\gs}(N)$ is a non-forking 
extension of $p_\ell$ for $\ell=1,2$.

\noindent
2) We say ${\gs}$ is type-full when ${\cS}^{\bs}_{\gs}
(M) = {\cS}^{\na}_{{\gk}_{\gs}}(M)$ for $M \in K_{\gs}$.

\noindent
3) We say $p \in {\cS}^{\bs}_{\gs}(M)$ is based on $\bar{\mathbf a}$ when:
\mn
\begin{enumerate}
\item[$(a)$]   $\bar{\mathbf a}$ is a sequence from $M$
\sn
\item[$(b)$]   if $M \le_{\gs} N$ and $q \in {\cS}^{\bs}_{\gs}(N)$ 
is a non-forking extension of $p$
and $\pi$ is an automorphism of $N$ over $\bar{\mathbf a}$ then $\pi(q) = q$
(by Ax(E)(k) there is such $\bar{\mathbf a} \in {}^\lambda(M)$).
\end{enumerate}
\mn
4) We say ${\gs}$ is $(< \theta)$-based when (3) there is such $\bar{\mathbf a}
\in {}^{\theta >}M$.
\end{definition}

\begin{definition}
\label{g14}
1) We say that NF is a non-forking relation on a
$(\mu,\lambda,\kappa)-1$-d.a.e.c. $\gk$ \when \, 
in addition to \ref{g2}(A)-(C)
\mn
\begin{enumerate}
\item[$(F)$]  $(a) \quad$  NF is a four-place relation on
 ${\gk}_{\gs},\NF_{\gs}(M_0,M_1,M_2,M_3)$ implies

\hskip25pt  $M_0 \le_{\gk} M_\ell
\le_{\gk} M_1$ and $\NF_{\gs}$ is preserved by isomorphisms
\sn
\item[${{}}$]  $(b)_1 \quad$  monotonicity: if $\NF_{\gs}(M_0,M_1,M_2,M_3),M_0
\le_{\gs} M'_\ell \le_{\gs} M_\ell$ for 

\hskip25pt  $\ell=1,2,
M'_1 \cup M'_2 \subseteq M'_3 \le_{\gs} M,M'_3 \le_{\gs} M$ 
then $\NF_{\gs}(M_0,M'_,M'_2,M'_3)$
\item[${{}}$]  $(c) \quad$  symmetry: NF$_{\mathfrak s}(M_0,M_1,M_2,M_3)$ implies
NF$_{\mathfrak s}(M_0,M_2,M_1,M_3)$
\sn
\item[${{}}$]  $(d)_1 \quad$ transitivity: if $\NF_{\gs}(M_{2 \ell},M_{2 \ell
+1},M_{2 \ell + 3},M_{2 \ell +4})$ for $\ell = 0,1$ then 

\hskip25pt  $\NF_{\gs}(M_0,M_1,M_4,M_5)$
\sn
\item[${{}}$]  $(d)_2 \quad$ long transitvity: if 
$\langle (N_i,M_i):i < \delta \rangle$
is an $\NF_{\gs}$-sequence (i.e., $M_i$ is

\hskip25pt  $\le_{\gs}$-increasing, 
$N_i$ is $\le_{\gs}$-increasing, $M_i \le_{\gs} M_i,i < j < 
\delta \Rightarrow$

\hskip25pt  $\NF_{\gs}(M_i,N_i,M_j,N_j)$ and 
$\Sigma\{\|N_i\|:i < \delta\} < \mu$
\then \, we can find 

\hskip25pt $(N_\delta,M_\delta)$ such that 

\hskip25pt  $\langle (M_i,N_i):
i \le \delta \rangle$ is an $\NF$-sequence.  \underline{But} what
about pr-continuity?
\sn
\item[${{}}$]  $(d)^+_2 \quad$ like $(d)_2$ for directed systems
\sn
\item[${{}}$]  $(e) \quad$ continuity 
\end{enumerate}
\end{definition}

\begin{definition}
\label{g17}
1) Let ${\gs}$ be a good $\lambda$-frame and
NF a non-forking relation on ${\gk}$.  We say NF respects ${\gs}$
\when \,: if $\NF_{\gs}(M_0,M_1,M_2,M_3)$ and $a \in M_2,
\ortp_{\gs}(a,M_0,M_3) \in {\cS}^{\bs}_{\gs}(M_0)$ \then \,
$\ortp_{\gs}(a,M_1,M_3)$ is a non-forking extension of 
$\ortp_{\gs}(a,M_0,M_2)$.

\noindent
2) We say $\gs$ is a good $(\lambda,\mu,\kappa)-\NF$-frame \when \,
it is a good $(\lambda,\mu,\chi)$-frame and $\NF_{\gs}$ is a
non-forking relation on $\gk_{\gs}$ which respects $\gs$.
\end{definition}

\begin{definition}
\label{g23}
We say that ${\gs}$ is a very good $(\mu,\lambda,\kappa)-\NF$-frame
\If \, it is a good $(\mu,\lambda,\kappa)-\NF$-frame and
\mn
\begin{enumerate}
\item[$(G)$]  $(a) \quad {\gk}_{\gs}$ has primes for chains and
even directed systems, see 

\hskip25pt Definition \ref{f4}(4)
\sn
\item[${{}}$]  $(b) \quad$ if $\NF_{\gs}(M_0,M_1,M_2,M_3)$ \then \, 
there is $M^*_3 \le_{\gs} M_3$ which is prime

\hskip25pt  over $M_1 \cup M_2$ that is:
\sn
\begin{enumerate}
\item[${{}}$]   $(*) \quad$ if $\NF_{\gs}(M'_0,M'_1,M'_2,M'_3)$ 
and $f_\ell$ is an isomorphism from $M_\ell$ onto $M'_\ell$ for 
$\ell =0,1,2$ such that $f_0 \subseteq f_1,f_0 \subseteq f_2$ \then \, 
there is a $\le_{\gs}$-embedding $f_3$ of $M^*_3$ into $M'_3$
extending $f_1 \cup f_2$
\end{enumerate}
\item[${{}}$]  $(c) \quad {\gk}_{\gs}$ has primes (see \ref{g26}(2) below).
\end{enumerate}
\end{definition}

\begin{definition}
\label{g26}
0) $K^{3,\bs}_{\gs} = \{(M,N,a):M \le_{\gs} N$ and $a \in N$ and 
$\ortp_{\gs}(a,M,N) \in {\cS}^{\bs}_{\gs}(M)\}$.

\noindent
1) $K^{3,\pr}_{\gs} = \{(M,N,a) \in K^{3,\bs}_{\gs}$: 
if $M \le N',a' \in N',\ortp_{\gs}(a',M,N') 
= \ortp(a,M,N)$ \then \, there is a $\le_{\gk}$-embedding of
$N$ into $N'$ extending id$_M$ and mapping $a$ to $a'$.

\noindent
2) ${\gk}_{\gs}$ has $K^{3,\pr}_{\gs}$-primes \If \, for 
every $M \in K_{\gs}$ and $p \in {\cS}^{\bs}_{\gs}(M)$ 
there are $(N,a)$ such that $(M,N,a) \in K^{3,\pr}_{\gs}$ 
and $\ortp_{\gs}(a,M,N) = p$. 
\end{definition}

\begin{definition}
\label{g29}
1) [$\iota \ge 3$]

Assume $p_1,p_2 \in {\cS}^{\bs}(M)$.  We say $p_1,p_2$ are
weakly orthogonal, $p_1 {\underset \wk \bot} p_2$ when: if
$M_0 \le_{\gs} M_\ell \le_{\gs} M_3,(M_0,M_\ell,a_\ell) \in
K^{3,\pr}_{\gs}$ and $\ortp_{\gs}(a_\ell,M_0,M_\ell) =
p_\ell$ for $\ell=1,2$ \then \, $\ortp_{\gs}(a_2,M_1,M_3)$ does not
fork over $M_0$ (symmetric by Ax(E)(f)).

\noindent
2) We say $p_1,p_2$ are orthogonal, $p_1 \bot p_2$ when: if
$M \le_{\gs} M_2,M_1 \le_{\gs} M_2$ and $q_\ell \in {\cS}^{\bs}(M_2)$
is a non-forking extension of $p_\ell$ and
$q_\ell$ does not fork over $M_1$ \then \, $q_1 {\underset \wk \bot} q_2$.

\noindent
3) We say that $\{a_t:t \in I\}$ is independent in
$(M_0,M_1,M_2)$ when
\mn
\begin{enumerate}
\item[$(a)$]   $a_t \in M_2 \backslash M_1$
\sn
\item[$(b)$]    $\ortp_{\gs}(a_t,M_1,M_2)$ does not fork over $M_0$
\sn
\item[$(c)$]   there is a list $\langle t(\alpha):\alpha <
\alpha(*)\rangle$ with no repetitions of $I$ and is a 
$\le_{\gs}$-increasing sequence $\langle M_{1,\alpha}:\alpha \le
\alpha(*)+1\rangle$ such that $M_1 \le_{\gs} M_{1,0},M_2 \le 
M_{1,\alpha(*)+1}$ such that $a_{t(\alpha)} \in M_{1,\alpha +1}$ and
$\ortp_{\gs}(a_{t(\alpha)},M_{1,\alpha},M_{1,\alpha +1})$ does not
fork over $M_0$.
\end{enumerate}
\mn
4) Let $(M,N,\mathbf J) \in K^{3,\bs}_{\gs}$ \If \, $M
\le_{\gs} N$ and $\mathbf J$ is independent in $(M,N)$.

\noindent
5) Let $(M,N,\mathbf J) \in K^{3,\qr}_{\gs}$ \If \,:
\mn
\begin{enumerate}
\item[$(a)$]   $M \le_{\gs} N$
\sn
\item[$(b)$]    $\mathbf J$ is independent in $(M,N)$
\sn
\item[$(c)$]   if $M \le_{\gs} N',h$ is a one-to-one function
from $\mathbf J$ into $N'$ such that $(M,N',h''(\mathbf J)) \in
K^{3,\bs}_{\gs}$ \then \, there is a $\le_{\gs}$-embedding $g$ 
of $N$ into $N'$ over $M$ extending $h$.
\end{enumerate}
\end{definition}

\begin{remark}
$\bar M = \langle M_i:i < \alpha\rangle$ is increasing
semi-continuous when it is $\le_{\gs}$-increasing, $M_\delta$
is prime over $\bar M \restriction \delta$ for every limit $\delta < \alpha$.
\end{remark}

\begin{remark}  
We now can imitate relations of the axioms (as in
\cite[\S2]{Sh:600}), and basic properties of the notions introduced in
\ref{g29}. 
\end{remark}

\begin{definition}
\label{g35}
1) We say $p$ is strongly dominated by $\{p_t:t \in I\}$, possibly
with repetitions, so pedantically we should use a sequence and write
$p \le^{\dm}_{\st} \{p_t:t \in I\})$; \when \,:
\mn
\begin{enumerate}  
\item[$(a)$]  $p \in {\cS}^{\bs}_{\gs}(N),p_t \in {\cS}^{\bs}_{\gs}
(N_t),N_t \le_{\gs} N^+ \in K_{\gs},N \le_{\gs} N^+$ and
\sn
\item[$(b)$]   if $N^+ \le_{\gs} N^*$ and $a_t \in N^*$ and
$\ortp(a_t,N^+,N^*) \in {\cS}^{\bs}_{\gs}(N^+)$ is parallel
to $p_t$ and $p' \in {\cS}^{\bs}_{\gs}(N^+)$ is parallel
to $p$, see Definition \ref{g11} and $\{a_t:t \in I\}$ is independent
in $(N^+,N^*)$ \then \, some $a \in N^*$ realizes $p'$.
\end{enumerate}
\mn
2) We say $p$ is weakly dominated by $\{p_t:t \in I\}$ and write 
$p \le_{\wk} \{p_t:t \in I\}$ \when \, for some set $J$
and function $h$ from $J$ onto $I$ we have $p
\le^{\dm}_{\st} \{p_{h(t)}:t \in J\}$.

\noindent
3) Let dominated mean strongly dominated.
\end{definition}

\begin{claim}
\label{g38}
1) If $p$ is strongly dominated by $\{p_t:t \in I\}$ \then \, 
$p$ is weakly dominated by $\{p_t:t \in I\}$.

\noindent
2) If $p$ is strongly dominated by $\{p_t:t \in I\}$ \then \, for some $J
\subseteq I$ of cardinality $< \kappa_{\gs},p$ is strongly
dominated by $\{p_t:t \in I\}$. 

\noindent
3) $p$ is weakly dominated by $\{p_t:t \in I\}$ \Iff \, for some
$\langle i_t:t \in I\rangle,p$ is strongly dominated by $\{p'_s:s \in
\{(t,i):t \in I,i < i_t\}\}$ where $p'_{(t,i)} = p_t,i_t <
\kappa_{\mathfrak s}$ for each $t \in I$.

\noindent
4) In Definition \ref{g35}(2) \wilog \, $(\forall s \in I)
(\exists^{< \kappa} t \in J)(h(t)=s)$. 

\noindent
5) [Preservation by parallelism]
\end{claim}

\begin{PROOF}{\ref{g38}}
Proof should be clear.
\end{PROOF}

The following should be included in very good for \ref{j50}, 
see\footnote{for the case there is $M \in K_\lambda$ brimmed 
over $N \in K_\lambda$ for every such $N$ this (\ref{j11}(1)(c))
work} more \ref{j11}
\begin{claim}
\label{g41}
1) If $p \le^{\dm}_{\wk} \{p_i:i < i^*\}$ and $i < i^* 
\Rightarrow q \perp p_i$ then $q \perp p$ (see Definition \ref{g26}(3)). 

\noindent
1A) If $p \le^{\dm}_{\wk} \{p_i:i < i^*\},(p \in
{\cS}^{\bs}_{\gs}(M))$ 
\then \, $p \pm p_i$ for some $i < i^*$.

\noindent
2) If $M \le_{\gs} N,q \in N,p_i \in {\cS}^{\bs}_{\gs}(M)$ for
$i < i^* < \kappa$, \then \, there is $q' \in {\cS}^{\bs}_{\gs}(M)$
such that for some $N^+,f$ we have:
\mn
\begin{enumerate}
\item[$(a)$]  $N \le_{\gs} N^+,f$ is an automorphism of $N^+$
\sn
\item[$(b)$]  $f$ maps the non-forking extension of $q$ in 
${\cS}^{\bs}_{\gs}(N^+)$ to the non-forking extension of $q'$ in 
${\cS}^{\bs}_{\gs}(N^+)$
\sn
\item[$(c)$]   $f$ maps the non-forking extension of $p_i$ in 
${\cS}^{\bs}_{\gs}(N^+)$ to itself.
\end{enumerate}
\mn
3) If $p \le^{\dm}_{\st} \{p_i:i < \alpha\}$
then $p \le^{\dm}_{\st} \{p_i:i < \alpha,p_i \pm p\}$
[see Def \ref{g35}.] 

\noindent
4) Assume 
\mn
\begin{enumerate}
\item[$(a)$]   $p,p_i \in {\cS}^{\bs}_{\gs}(M)$ for $i < i^*$
\sn
\item[$(b)$]   $p_i$ is weakly dominated by $p$
\sn
\item[$(c)$]    no $q \in {\cS}^{\bs}_{\gs}(M)$ is weakly
dominated by $p$ and orthogonal to $p_i$ for $i < i^*$.
\end{enumerate}
\mn
\Then \, $p \le^{\dim}_{\wk} \{p_i:i < i^*\}$.
\end{claim}

\begin{PROOF}{\ref{g41}}
  Let $p \in {\cS}^{\bs}(N),N \in K_A$.  
If $p \in {\mathbf P}^\perp$ we are done so assume 
$N \le N_1 \in K_{\gs},q \in {\cS}^{\text{bs}}_{\gs}(N_1) 
\cap {\mathbf P}$ and $p \pm q$.  Let $\mathbf{\bar a} \in
{}^{\kappa >} N_1$ be such that $q$ is definable over $\bar{\mathbf a}$,
so we can find $\langle \bar{\mathbf a}_i:i < \kappa \rangle,\bar a_i
\in {}^{\ell g(\bar{\mathbf a})} M$ such that $\langle \bar{\mathbf a}_i:i
< \kappa \rangle \char 94 \langle \bar{\mathbf a} \rangle$ is
indiscernible over $A \cup \bar b$ where $p$ be definable over
$\mathbf{\bar b} \subseteq M$.  

Let $q_i \in {\cS}^{\text{bs}}(N_1)$ be defined over $\bar{\mathbf
a}_i$ as $q$ was defined over $\bar{\mathbf a}$ so easily $q_i \in
{\cP},p \pm q_i$, so we are done.
\end{PROOF}

\begin{claim}
\label{g44}
1) If $\chi = \chi^{< \kappa} \in [\lambda,\mu)$, the following is impossible:
\mn
\begin{enumerate}
\item[$(a)$]    $\langle M_i:i < \chi^+
\rangle$ is $\le_{\gs}$-increasing $\le_{\gs}$-semi-continuous,
\sn
\item[$(b)$]   $\langle N_i:i < \chi^+ \rangle$ is 
$\le_{\gs}$-increasing, $\le_{\gs}$-semi-continuous,
\sn
\item[$(c)$]   $M_i \le_{\gs} N_i \in K_{\le \chi}$,
\sn
\item[$(d)$]  for some stationary $S \subseteq \{\delta <
  \chi^+:\cf(\gamma) \ge \kappa\}$ for every $i \in S$
\sn
\begin{enumerate}
\item[$\bullet$]  there is $a_i \in M_{i+1} \backslash M_i$ such that 
$\ortp(a_i,
N_i,N_{i+1})$ is not the non-forking extension of $\ortp(a_i,M_i,M_{i+1}) 
\in {\cS}^{\bs}_{\gs}(M_i)$.
\end{enumerate}
\end{enumerate}
\mn
2) Like (1) replacing (d) by:
\mn
\begin{enumerate}
\item[$(d)'$]   like (d) replacing $\bullet$ by
\sn
\begin{enumerate}
\item[$\bullet'$]  $b_i \in N_i \backslash M_i,\ortp_{\gs}
(b_i,M_i,N_i) \in {\cS}^{\bs}_{\gs}(M_i)$ and
$\ortp_{\gs}(b_i,M_{i+1},N_{i+1})$ forks over $M_i$.
\end{enumerate}
\end{enumerate}
\end{claim}

\begin{PROOF}{\ref{g44}}  
Should be clear.
\end{PROOF}

\begin{claim}
\label{g50}
If $p,p_i \in {\cS}^{\bs}_{\gs}(M)$ for $i < \kappa_{\gs}$ and $i < j
\Rightarrow p_2 \bot p_j$ \then \, $p \bot p_i$ for every 
$i < \kappa$ large enough.
\end{claim}

\begin{proof} 
Similar to the proof of \ref{j20}.
\end{proof}

\begin{definition}
\label{g53}
1) We say that (a good (frame$_{\gs}$), ${\gs}$ is $\theta$-based$_1$ when:
\mn
\begin{enumerate}
\item[$(a)$]   if $p \in {\cS}^{\bs}_{\gs}(M)$ then for
some $\bar{\mathbf a} \in {}^{\theta >}M,p$ is based on $\bar{\mathbf a}$
(see Definition \ref{g11}(4)).
\end{enumerate}
\mn
2) We say that ${\gs}$ is $\theta$-based$_2$ \when \,:
\mn
\begin{enumerate}
\item[$(a)$]   as in part (1)
\sn
\item[$(b)$]  $\gs$ is full
\sn
\item[$(c)$]   if $M_1 \le_{\gs} M_2$ and $p \in
  {\cS}^{\bs}_{\gs}(M_2)$ \then \, for some $\bar a_\ell \in
{}^{\theta >}(M_\ell)$ the types $p,\ortp_{\gs}(\bar a_2,M_1,M_2)$ 
are based on $\bar{\mathbf a}_2,\bar{\mathbf a}_1$ respectively, 
(or axiomatically!).
\end{enumerate}
\end{definition}
\newpage

\centerline {Part III} \label{part3}

\section {Introduction/Thoughts on the main gap} \label{7}

We address here two problems: type theory (i.e. dimension, orthogonality,
etc.) for strictly stable class and the main gap concerning somewhat
saturated models, the hope always was that advance in the first will
help the second.

Concerning the first order case work started in \cite[Ch.V]{Sh:c},
particularly \cite[Ch.V,\S5]{Sh:c} and \cite{Sh:429} and was much
advanced in Hernandes \cite{He92}; but this was not enough for the main
gap for somewhat saturated models.

We deal here with the type dimension for a general framework.
\bigskip

\centerline {$* \qquad * \qquad *$}
\bigskip

The main gap for $\aleph_1-|T|^+$-saturated model of a countable first
order theory is open.  A priori it has looked easier than the one for
models (which was preferred as ``the original question") because of
the existence of prime models over any, but is still
open (and for uncountable first order, $|T|^+$-saturated model as well).

Why the proof in \cite[Ch.XII]{Sh:c} does not work?
What is missing is (in ${\gC}^{\eq}!$)
\mn
\begin{enumerate}
\item[$\circledast$]   if $M_0 \prec  M_1 \prec M_2$ are
$\aleph_1$-saturated, $a \in M_2 \backslash M_1$ and $(a/M_1) \pm M_0$
  \then \, for some $b \in M_2 \backslash M_1$ we have 
$\nonfork{b}{M_1}_{M_0}$.
\end{enumerate}
\mn
The central case is $a/M_1$ is orthogonal to $q$ if $q \bot M_0$.
\bigskip

\noindent
\underline{Possible Approach 1}:  We use $T$ being first order
countable, stable NDOP (even shallow) to understand types.  
See \cite{Sh:851}.
\bigskip

\noindent
\underline{Possible Approach 2}:  We use the context of part II.  We are
poorer in knowledge \underline{but} we have a richer ${\gC}^{\eq}$ so we may
prove $\circledast$ even if its fails for $T$ in the elementary case
(this is a connection between Part I and Part II of this work).
\bigskip

\noindent
\underline{Possible Approach 3}:  We use the context of part II.  If things
are not O.K. we define a derived such d.a.e.c. (as in \cite{Sh:300f} or
\cite{Sh:600}),
it may have non-structure properties, if not we arrive to the same
place.  Similarly in limit.  If we succeed enough times we shall prove
that all is O.K.
\bigskip

\noindent
\underline{Possible Approach 4}:  Now we have a maximal non-forking tree
$\langle M_\eta,a_\eta:\eta \in {\cT}\rangle$ inside a somewhat
saturated model; for \cite{Sh:c}, e.g. $\|M_\eta\| \le \lambda$, the models
are $\lambda^+$-saturated but we use models from 
\cite[Part II]{Sh:839}.  If $M$ is prime over $\cup\{M_\eta:\eta \in
\cT\}$ we are done, but maybe there is a residue.  
This appears as: for $\eta \in {\cT}$, and 
$p \in \mathbf \cS^{\bs}(M_\eta)$ the dimension of $p$ is not
exhausted by $\langle a_{\eta \char 94<\alpha>}:\eta \char 94 \langle
\alpha \rangle \in {\cT}$ and $(a_{\eta \char 94 <\alpha>}/M_\eta) \pm p\}$
but the lost part is not infinite!  This imposes $\le \lambda$ unary
functions from ${\cT}$ to ${\cT}$.  Now it seems to me that the
question of this possible non-exhaustion arise (essentially: there is a
non-algebraic $p \in (M^\bot)^\bot$ which do not 1-dominate any $q \in
{\cS}(M))$ is not a good dividing line, as though its negation is
informative, it is not clear whether it has any consequence.  However,
there are two candidates for dividing lines (actually their
disjunction seems so)
\mn
\begin{enumerate}
\item[$(A)$]   $(*) \quad$ we can find $M,\langle M_\eta,a_\eta:\eta
\in {\cT}\rangle$ as above and $\eta_* \in {\cT},\ell g(\eta) =
2,\nu_* \in$

\hskip25pt ${\cT},\ell g(\nu_*) = 1,\eta_* \restriction 1 \ne \nu_*$
and $p \in {\cS}^{\bs}(M_{\eta_*}),p \bot M_{\eta \restriction 1}$ 

\hskip25pt with a residue as above 
such that we need $M_{\nu_*}$ to explicate it.
\end{enumerate}
\mn
More explicitly
\mn
\begin{enumerate}
\item[$(*)'$]   if $M' \le_{\gs} M$ is prime over
$\cup\{M_\eta:\eta \in {\cT}\}$ and we can find $a_{\eta_*,\nu_*}
\in M \backslash M'$ such that $\ortp(\cC(a_{\eta^*,\nu^*},M'),
\cup\{M_\eta:\eta \in {\cT}\})$ mark $(M_{\eta_*},M_{\nu_*})$.
\end{enumerate}
\mn
Even in $(*)'$ we have to say more in order to succeed in using it.

\relax From $(*)'$ we can prove a non-structure result: on ${\cT}$ we can
code any two-place relation $R$ on $\{\eta \in {\cT}:\ell
g(\eta)=1,M_\eta,M_{\eta_* \restriction 1}$ isomorphic over
$M_{<>}\}$ which is $\nu_1 R \nu_2 \Leftrightarrow 
(\exists \nu) \bigwedge \limits_{\ell}$
[there is $\eta',\eta_\ell \triangleleft \eta' \in {\cT},\ell
g(\eta')=2$ and $\nu \in T,\ell g(\nu)=1$ and there is $a_{\eta'\nu}$
as above].

More complicated is the case
\mn
\begin{enumerate}
\item[$(B)$]  $(**) \quad$ we can fix $(M,\langle M_\eta,a_\eta:\eta \in
{\cT}\rangle$ as above), $\eta_* \in T,\nu_*,\nu_{**} \in$

\hskip25pt ${\cT},\ell g(\eta_*) = \ell g(\nu_*)=1=\ell g(\nu_{**})$ such
that

\hskip25pt  $(\eta_*,\nu_*),(\eta_*,\nu_{**})$ are as above.
\end{enumerate}
\mn
But whereas for (A) we have to make both $\eta_*$ and $\nu_*$ not
redundant in (B), in order to get non-structure we have to use a case of (B)
which is not ``a faking", e.g. cannot replace
$(M_{\eta_*},a_{\eta_*})$ by two such pairs.

That is, the ``faker" is a case where we can find
$M'_{\eta_*},M''_{\eta_*}$ such that:
\mn
\begin{enumerate}
\item[$(a)$]   $\NF(M_{<>},M'_{\eta_*},M''_{\eta_*},M_{\eta_*})$
\sn
\item[$(b)$]  $M_{\eta_*}$ is prime over $M'_{\eta_*} \cup M''_{\eta_*}$
\sn
\item[$(c)$]   only $(M'_{\eta_*},M_{\nu_*})$ and
$(M''_{\eta_*},M_{\nu_{**}})$ relate.
\end{enumerate}
\mn
[Possibly we have $\langle \nu_t:t \in I\rangle$ we can ``divide" but
not totally, probably a problem]

\noindent
$(C)$ if both (A) and (B) in the right formulation does not appear
then
\mn
\begin{enumerate}
\item[$(\alpha)$]    good possibility; we can prove a structure
theory: for $M,\langle M_\eta,a_\eta:\eta \in {\cT}\rangle$ as
above that is on each suc$_{\cT}(\eta)$ we have a two-place
relation but it is very simple, you have to glue some together or at
most at tree structure.
\end{enumerate}
\mn
If this fails we may fall back to approach (3).

\begin{question}  
\label{z3}
1) For an a.e.c. $\gk$ when does the theory of a model in the logic ${\cL} =
\bbL_{\infty,\kappa} [{\gk}]$ enriched by dimension quantifiers, 
characterize up to isomorphism models of ${\gk}$?  Similarly
enriching also by game quantifiers of length $\le \kappa$.

\noindent
2) Prove the main gap theorem in the version: if ${\gs}$ is
$n$-beautiful (or $n+1$?) \then \, for $K_{\lambda^{+n}}$ the main
gap holds, if in particular: if ${\gs}$ has NDOP, every $M \in
K_{\lambda^{+n}}$ is prime over some non-forking tree of 
$\le_{{\gK}[{\gs}]}$-submodels $\langle M_\eta:\eta \in {\cT}\rangle$,
each $M_\eta$ of cardinality $\le \lambda,{\cT} \subseteq
{}^{\omega >}(\lambda^{+n})$.  If ${\gs}$ is shallow then the tree
has depth $\le \Depth({\gs}) < \lambda^+$ and we can draw
conclusion on the number of models.
\end{question}

\begin{discussion}
\label{m5}  
[Assume stability in $\lambda_{\gs}$]. 

Let $M_0 \in K_{\gs},\lambda^+_{\gs}$-saturated at least for the time
being.

\noindent
1) Assume
\mn
\begin{enumerate}
\item[$\boxplus_1$]  $N_0 \le_{\gs} N_1 \le_{\gs} M,N_\ell \in
  K^{\gs}_\lambda,a \in N_0$ and $(N_0,N_1,a) \in K^{3,\pr}_{\gs}$.
\end{enumerate}
\mn
We choose $(N^+_{1,i},N_{1,i},\mathbf I_i)$ and if possible also
$(M_1,a_i)$ by induction on $i \le \lambda^+_{\gs}$ such that
\mn
\begin{enumerate}
\item[$(*)(a)$]  $(\alpha) \quad N_{0,i} \le_{\gs} N_{1,i} \le_{\gs}
  N^+_{1,i} \le_{\gs} M$
\sn
\item[$(b)$]  $\mathbf I_i \subseteq \{c \in M:\ortp(c,N_{1,i},M_0)
  \perp N_0\}$ is independent in $(N_{1,i},N^+_{1,i},M)$ and minimal
\sn
\item[$(c)$]  $\langle N_j:j \le i\rangle$ is
  $\le_{gs}$-semi-continuous also $\langle N^+_j:j \le i\rangle$
\sn
\item[$(d)$]  if $i=j+1$ then $N^+_{1,i}$ is $\le_{\gs}$-universal
  over $N^+_{1,j}$ and $(N_0,N_{1,i},a) \in K^{3,\pr}_{\gs}$
\sn
\item[$(e)$]  if $j < i$ then $\mathbf I_j \backslash (N_i \cap \mathbf
  I_j) \subseteq \mathbf I_i$
\sn
\item[$(f)$]  if possible
\sn
\item[${{}}$]  $(\alpha) \quad N_i \le_{\gs} M^+_i \le_{\gs} M$
\sn
\item[${{}}$]  $(\beta) \quad (\mathbf I_i \backslash M_i)$ is
  independent in $(M_i,M)$
\sn
\item[${{}}$]  $(\gamma) \quad a_i \in M \backslash (\mathbf I_i)$
\sn
\item[${{}}$]  $(\delta) \quad \ortp(a_i,M^*_1,M) \in
  \cS^{\bs}_{\gs}(N^+_i)$ is $\perp N_i$
\sn
\item[${{}}$]  $(\varepsilon) \quad N^*_i \le N_{1,i+1}$
\sn
\item[$(g)$]  if $i=j+1$ thee are $(b,N^+_*,N_{**})$ such 
that $b \in
  N^+_{1,j} \backslash N_{1,j},N_{1,i} \le_{\gs} N_* \le_{\gs},N_{**}
  \in K^{\gs}_{\lambda_{\gs}},N^+_{1,i} \le_{\gs} N_{**}$ and
  $\ortp_{\gs}(b,N_*,N_{**})$ forks over $N_{1,j}$ then for some $b
  \in N^+_{1,j} \backslash N_{1,j}$ the type
  $\ortp_{\gs}(N_{1,i},N^+_{1,i})$ forks over $N_{1,j}$.
\end{enumerate}
\mn
There is no problem to carry the induction.
\mn
\begin{enumerate}
\item[$\boxplus_2$]  the following subset of $\lambda^+_{\gs}$ is not
  stationary; say disjoint to the club $C$:
\sn
\begin{enumerate}
\item[$\bullet$]  $S = \{i < \lambda^+_{\gs}:\cf(i) \ge \kappa_{\gs}$
  and $(M_i,a_i)$ is well defined$\}$
\sn
\item[$\bullet$]  $S_2 = \{i:\cf(i) \ge \kappa_{\gs}$ and for some $b
  \in N^+_{1,i},\tp(b,N_{1,i},N^+_{1,i}) = N_0$.
\end{enumerate}
\end{enumerate}
\mn
2) Similarly without $(N_0,a)$ hence without $\perp N_0$; just simpler.
\end{discussion}

\begin{definition}
\label{h8}
We say $(\bar N,\bar a,\bar I)$ is a decreasing pair for $M$ \when \,
for some $n$:
\mn
\begin{enumerate}
\item[$(a)$]  $\bar N = \langle N_\ell:\ell \le n \rangle$ is
  $\le_{\gs}$-increasing 
\sn
\item[$(b)$]  $N_\ell \le_{\gs} M,N_   
\ell \in K^{\gs}_{\lambda_{\gs}}$
\item[$(c)$]  $\bar a = \langle a_\ell:\ell < n\rangle$
\sn
\item[$(d)$]  $(N_\ell,N_{i+1},a_\ell) \in K^{3,\pr}_{\gs}$
\sn
\item[$(e)$]  $\bar{\mathbf I} = \langle \mathbf I_\ell:\ell \le n \rangle$
\sn
\item[$(f)$]  $\mathbf I_\ell$ is independent in $(N_\ell,M)$
\sn
\item[$(g)$]  $\mathbf I_\ell \subseteq \{c \in M:\ortp(c,N_\ell,M) \in
  \cS^{\bs}_{\gs}(N_\ell)$ is $\perp N_k$ if $k < \ell\}$
\sn
\item[$(h)$]  if $N_ 
\ell \le_{\gs} N \le_{\gs} M,b \in M \backslash
  N_0 \backslash \mathbf I_\ell,\ortp(b,N,M)$is $\pm N_\ell$ but $N_k$
  for $k < \ell$ \then \, $b$ depends on $\mathbf I_\ell$ in
  $(N_\ell,M)$.
\end{enumerate}
\end{definition}

\noindent
\underline{Attempt to prove decomposition}

We assume dimensional continuity to prove decomposition.  If we like
to get rid of ``$M$ is $\lambda^+_{gs}$-saturated", we assume we have
a somewhat weaker version $\gs_*$ of $\gs$ with $\lambda_{\gs_*} <
\lambda_{\gs}$ and $\le_{\gs_*}$-represent $N_0$ is $\langle N_{0,i}:i
< \lambda_{\gs}\rangle$ and work with it.  Assuming CH, $|T| =
\aleph_0$ fine.  Without dimensional discontinuity we call nice $(\bar
N,\bar a,\bar{\mathbf I})$ of length $\le \kappa_{\gs}$!
\bigskip

\centerline {$* \qquad * \qquad *$}
\bigskip

\begin{definition}
\label{h11}
We say $\mathbf d = (I,N,\bar a,\bar{\mathbf I})) = (I_{\mathbf d},\bar
N_d,\bar a_{\mathbf d},\bar{\mathbf I}_{\mathbf d})$ is a partial
  decomposition of \when \,:
\mn
\begin{enumerate}
\item[$\boxplus$]  $(a) \quad I \subseteq {}^{\omega >}\Ord$ is closed
  under initial segments
\sn
\item[${{}}$]  $(b) \quad \bar N = \langle N_\eta:\eta \in I\rangle$
  so $N_\eta = N_{\mathbf d,\eta}$
\sn
\item[${{}}$]  $(c) \quad \bar a = \langle a_\eta:\eta \in I\backslash
  \{<>\}\rangle$ so $a_\eta = a_{\mathbf d,\eta}$
\sn
\item[${{}}$]  $(d) \quad \bar{\mathbf I} = \langle \mathbf I_\eta:\eta
  \in I\rangle$ so $\mathbf I_\eta = \mathbf I_{\mathbf d,\eta}$ 
\sn
\item[${{}}$]  $(e) \quad$ if $\eta \in I$ then $(\langle N_{\eta
  \rest \ell}:\ell \le \ell g(\eta)\rangle,\langle \bar a_{\eta \rest
  (\ell +1)}$:

\hskip25pt $\ell < \ell g(\eta)\rangle,\langle \mathbf I_{\eta \rest
  \ell}:\ell \le \ell g(\eta)\rangle)$ is nice in $M$
\sn
\item[${{}}$]  $(f) \quad$ if $\eta \in I$ then $\langle a_{\eta \char
  94 \langle \alpha \rangle}:\eta \char 94 \langle \alpha \rangle \in
  I\rangle$ is a sequence of members of $\mathbf I_\eta$

\hskip25pt  with no repetitions.
\end{enumerate}
\end{definition}

\begin{definition}
\label{h14}
Let $\le_\mu$ be the following two-place relation on the set of
decompositions of $M$:

$\bar{\mathbf d}_1 \le_M \mathbf d_2$ \Iff \,
\mn
\begin{enumerate}
\item[$(a)$]  $I_{\mathbf d_1} \subseteq \mathbf I_{d_1}$
\sn
\item[$(b)$]  $\bar N_{\mathbf d_1} = \bar N_{\mathbf d_2} \rest I_{d_1}$
\sn
\item[$(c)$]  $\bar a_{\mathbf d_1} = \bar a_{\mathbf d_2} \rest (I_{\mathbf
  d_1} \backslash \{<j\})$
\sn
\item[$(d)$]  $\bar{\mathbf I}_{\mathbf d_1} = \bar{\mathbf I}_{d_2} \rest
  I_{\mathbf d_1}$ 
\end{enumerate}
\end{definition}
\newpage

\section {III, Analysis of dimension for ${\mathbf P}$} \label{8}


\begin{hypothesis}
\label{j2}
1) ${\gs}$ is a \underline{very good} type-full
$(\mu,\lambda,\kappa)$-NF-frame (see \ref{g23}), \underline{or} is full good
$(\mu,\lambda,\kappa)$-frame$_{\gs}$ - see Definition and we
sometimes use type-full (\ref{g11}) and $\bot =
{\underset \wk \bot}$); semi-continuous will mean
$\le_{\gs}$-semi-continuous; the ``$p$ base on $\bar a$"
should be used to justify $\bot = {\underset \wk \bot}$.

\noindent
2) ${\gC}$ is a ${\gs}$-monster or we use a place $K_{\ga}$ (or $K_A$ and
$K_{(M,\bar a)}$, see Definition \ref{a17}(5)).
\end{hypothesis}

\begin{definition}
\label{j5}
Let $K_{\ga}$ be a place, ${\mathbf P}$ is a
$A$-based family of types (see \ref{a26} closed under parallelism
(add to Definition in \S6), i.e., with the monster version this means:
\mn
\begin{enumerate}
\item[$(a)$]   $A \subseteq {\gC},M$ vary on $M \le_{\gk}
{\gC}$ such that $A \subseteq M$ or just ${\gk}_A$ is well defined
\sn
\item[$(b)$]   ${\mathbf P} \subseteq \cup\{{\cS}^{\bs}(M):M \in
{\gk}_A$, i.e., $A \subseteq M \le_{\gk} {\gC}\}$
\sn
\item[$(c)$]   every automorphism of ${\gC}$ over $A$ maps
${\mathbf P}$ onto itself 
\end{enumerate}
\mn
here we add
\mn
\begin{enumerate}
\item[$(d)$]   if $p_\ell \in {\cS}^{\bs}_{\gs}(M_\ell)$ 
for $\ell =1,2,M_\ell \in K_{\ga}$ and $p_1\|p_2$ then $p_1 \in {\mathbf P}
\Leftrightarrow p_2 \in {\mathbf P}$ (desirable here, not in weight check!).
\end{enumerate}
\mn
1) ${\mathbf P}$ is $\ga$-dense \when \ ($\mathbf P,A$ are as in part (0))
\mn
\begin{enumerate}
\item[$(*)$]  if $M \in K_{\ga}$ and $p \in {\cS}^{\bs}_{\gs}(M)$
  \then \, 
$(\exists q \in {\mathbf P} \cap {\cS}^{\bs}(M))[p {\underset \wk \pm} q]$.
\end{enumerate}
\mn
2) We say $\langle M_i,a_j:i \le \alpha,j < \alpha \rangle$ is
a ${\mathbf P}$-primeness sequence \when \,:
\mn
\begin{enumerate}
\item[$(a)$]   $M_i$ is $\le_{\gs}$-increasing semi-continuous
[i.e., if $i$ is a limit ordinal then $M_i$ is ${\gs}$-prime over
$\cup\{M_j:j < i\}$, see Definition \ref{g26}(5)]
\sn
\item[$(b)$]   $(M_i,M_{i+1},a_i) \in K^{3,\pr}_{\gs}$
\sn
\item[$(c)$]   $\ortp_{\gs}(a_i,M_i,M_{i+1}) \in {\mathbf P}$.
\end{enumerate}
\mn
3) We may say $\langle M_i,a_j,p_j:i \le \alpha,j < \alpha\rangle$ is
a $\mathbf P$-prime sequence when in addition
\mn
\begin{enumerate}
\item[$(d)$]  $p_i = \ortp_{\gs}(a_i,M_i,M_{i+1})$.
\end{enumerate} 
\end{definition}

\begin{definition}
\label{j8}
Let $\mathbf P$ be a $\mathbf a$-based family.

\noindent
1) ${\mathbf P}^\perp = \{p:p \in {\cS}^{\bs}_{\gs}(M),M \in K_{\ga}$ 
and $p \perp \mathbf P$, see below$\}$.

\noindent
2) $p \perp {\mathbf P}$ means that for some $M:p \in
 \cS^{\bs}_{\gs}(M),M \in K_{\ga}$ and if $M \le_{\gs} N,q \in
\cS^{\bs}_{\gs}(N),p \in \cS^{\bs}(M)$ is a non-forking extension of
 $q$ then $p' \perp q$.
\end{definition}

\begin{claim}
\label{j11}
1) If ${\mathbf P}$ is an $\ga$-based family of types \then \,:
\mn
\begin{enumerate}
\item[$(a)$]   ${\mathbf P}^\perp$ is an $\ga$-based family of types
\sn
\item[$(b)$]  ${ \mathbf P 
} \cup \mathbf P^\perp$ is dense.
\end{enumerate}
\mn
2) Any Boolean combination of $\ga$-based families of  
types in an
$\ga$-based family of types.
\end{claim}

\begin{PROOF}{\ref{j11}}
1) \underline{Clause (a)}:  (is immediate).  

Mainly we should check that $\mathbf P^+$ is closed under parallelism.
Let $N_\ell \in K_A,p_\ell \in {\cS}^{\bs}_{\gs}(N_\ell)$
for $\ell=1,2,N_1 \le_{\gs} N_2,p_2$ is a non-forking extension of
$p_1$.  We should prove $p_1 \in \mathbf P^+ \Leftrightarrow p_2 \in
\mathbf P^\bot$.  The direction $\Rightarrow$ is obvious.  For
$\Leftarrow$ use amalgamation, etc. 

Clause (b) should be clear (weak density).

\noindent
2) Easy.
\end{PROOF}

\begin{claim}
\label{j17}
Assume ${\mathbf P}$ is a reasonable $A$-based, dense.

\noindent
1) If $M \in K_{\ga}$ and $q \in {\cS}(M)$, \then \, we can find a
${\mathbf P}$-primeness sequence $\langle M_i,a_j:i \le \alpha,j < \alpha
\rangle$ with $M_0 = M$ such that $q$ is realized in $M_\alpha$. 

\noindent
2) If $M \le_{\gs} N$ and $M \in K_{\ga}$ \then \, for some 
${\mathbf P}$-primeness sequence $\langle M_i,a_j:i \le \alpha,j < \alpha
\rangle$ we have $M_0 = M$ and $N \le_{\gs} M_\alpha$.

\noindent
3) In part (2) if $\|N\| \le \chi \in [\lambda,\mu)$ and $\chi = \chi^{<
\kappa}$ \ then \, we can demand $\alpha \le \chi$.  Similarly in part
  (1) \wilog \, $\|M_\alpha\| \le \|N\|^{< \kappa}$.
\end{claim}

\begin{PROOF}{\ref{j17}}  
1) We can find $N$ such that $M \le_{\gs} N$ and $a \in N$ realizes
 $q$; we let $A = \{a\}$ and apply 2). 

\noindent
2) Let $\chi = \|N\|^{< \kappa} (< \mu)$.  Now try to
choose $(M_i,N_i)$ and then $a_i$ by induction on $i < \chi^+$ 
such that (here we use ``$ \gs$ is full"). 
\mn
\begin{enumerate}
\item[$\circledast$]   $(a) \quad M_0=M$
\sn
\item[${{}}$]  $(b) \quad M_i \in K^{\gs}_{\le \chi}$ is
$\le_{\gs}$-increasing $\le_{\gs}$-semi-continuous with $i$
\sn
\item[${{}}$]  $(c) \quad N_i$ is $\le_{\gs}$-increasing semi-continuous
with $i$
\sn
\item[${{}}$]  $(d) \quad N_0 = N$
\sn
\item[${{}}$]  $(e) \quad M_i \le_{\gs} N_i$
\sn
\item[${{}}$]  $(f) \quad \langle M_\varepsilon,a_\zeta:
\varepsilon \le i,\zeta < i \rangle$ form 
a ${\mathbf P}$-primeness sequence
\sn
\item[${{}}$]  $(g) \quad$ if $i=j+1$ then for some $c_i
\in N_i \backslash M_i$ we have 

\hskip25pt  $\ortp_{\gs}(c_i,M_i,N_i) \in 
{\cS}^{\bs}_{\gs}(M_i)$ and $\ortp_{\gs}(c_i,M_{i+1},N_{i+1})$ 
forks over $M_i$. 
\end{enumerate}
\mn
For $i=0$ trivial.  For $i=j+1$ if $M_i \ne N_i$ by
assumption and Definition \ref{j5}(2)  
let $c_j \in N_i \backslash
M_i$ so $r_j = \ortp_{\gs}(c_j,M_j,N_j) \in {\cS}^{\bs}_{\gs}(M_j)$ 
hence there $p_j \in {\cS}^{\bs}_{\gs}(M_j) \cap \mathbf P$ 
such that $p_j \pm r_j$.
So there are $N_i,a_j$ such that $N_j \le N_i$ and $a_j$ realizes $p_j$,
but $\{a_j,c_j\}$ is dependent.  
Choose $M_i \le_{\gs} N_i$ such that $(M_j,M_i,a_j) 
\in K^{3,\pr}_{\gs}$.  
For $i$ limit choose first $N_i$ (using the existence of primes over
increasing sequences) then $M_i$ (similarly) and \wilog \, $M_i
\le_{\gs} N_i$ using the Definition of prime.  
If we have carried the induction we get
contradiction to \ref{g44}(2).  So for some $i$ we are stuck: $M_i$
is chosen but $N_i = M_i$ so we are done.

\noindent
3) Follows by the proof of part (2). 
\end{PROOF}

\begin{claim}
\label{j20}
In \ref{j17} we can add $\alpha < \kappa_{\gs}$ if: $\gs$ is ufll (and
maybe $\mathbf P$ is dense).
\end{claim}

\begin{PROOF}{\ref{j20}}
 Let $\chi = \|M\|^{< \kappa}$, let N,b be such that $M
\le_{\gs} N \in K^{\gs}_\chi,b \in N$ realizes $q$.  Now we
repeat the proof of \ref{j17}(2) but in $\circledast$ add 
\mn
\begin{enumerate}
\item[$\bullet$]  $\ortp_{\gs}(b,M_i,M_{i+1})$ forks over $M_i$.
\end{enumerate}
\mn
This is possible by the assumption ``${\mathbf P}$ is dense 
$+ {\gs}$ is full" except when $b \in M_i$.  In the end
$\langle M_i:i < \kappa \rangle$ is well defined then it is 
$\le_{\gs}$-increasing continuous, $\ortp_{\gs}(b,M_\kappa,N_\kappa) \in 
{\cS}^{\bs}_{\gs}(M_\kappa)$ and it forks over $M_i$ for every $i <
\kappa$, contradiction to (c) of (E) of \ref{g2}.
\end{PROOF}

\begin{claim}
\label{j23}
\noindent
1) 
If $\langle M_i,a_j,p_j:i \le \alpha,j < \alpha\rangle$
 is a primeness sequence, $u = \{j:p_j \perp M_0$ and $j \in u
 \Rightarrow p_j$ does not fork over $M_0\}$ and $q \in
 \cS^{\bs}_{\gs}(M_0)$ is realized in $M$ \then \, $q \le^{\st}_{\dm}
 \{p_j:j \in u\}$. 

\noindent
2) Like (1) but $u = \{j:p_j \perp q\}$.

\noindent
3) [
In \ref{j17}(1) if we assume
$(**)^q_{M,{\mathbf P}}$ below we can add $(*)^q_{M,{\mathbf P}}$ to the
conclusion where 
\mn
\begin{enumerate}
\item[$(*)^q_{\mathbf P,M}$]  the type $p_i =: \ortp_{\gs}
(a_i,M_i,M_{i+1})$ does not fork over $M_0$ or $q$ is non-orthogonal to
$p_i$ for every $i < \alpha$
\sn
\item[$(**)^q_{{\mathbf P},M}$]   if $M \le_{\gs} N$ and $p \in
{\cS}^{\bs}_{\gs}(N) \cap {\mathbf P}$ and $p$ is not orthogonal to $q$
(e.g., $p\|p_i$) then $p$ does not fork over $M_0$ or $p \perp M_0$. 
\end{enumerate}
\mn
4) In (3) 
we can conclude that $q$ is dominated by $\{p_i \restriction
M_0:i < \alpha\}$.

\noindent   ss 6) 
5) We can replace $(**)^q_{\mathbf P,M}$ in part (3) 
by:
\mn
\begin{enumerate}
\item[$\circledast^q_{\mathbf P,M}$]   if $M \le_{\gs} N_1
<_{\gs} N_2$, then for some $r \in {\cS}^{\bs}_{\gs}(N_1)$  
realized in $M_2$, there is $p {\underset \wk \pm} r$ as in 
$(**)^q_{\mathbf P,M}$.
\end{enumerate}
\end{claim}

\begin{proof}
The same proof using \ref{g41}(3).  
\end{proof}

\begin{discussion}
\label{j26}  
Below we may think of the case $\bar a^i,\bar a$ has length $<
\kappa_{\mathfrak s}$ (if base$_{\gs}$ is well defined),
or we may think $\bar a,\bar a^i$ list the members of some $N_i
\le_{\gs} M,N_i \le_{\gs} M_i$ of cardinality $\lambda_{\gs}$ 
in this case we can replace ``$p_i$ definable over $\bar a_i$" by
``$p_i$ does not fork over $N_i$".
\end{discussion}

The following definition is central here.  In the case $\kappa_{\gs} 
= \aleph_0$ all is clear.  Note that even if $\ell g(\bar{\mathbf a}) =
\|M\|$, we are interested in the case $M$ is brimmed over $\bar{\mathbf
a}$.  Note if we understand $\langle p_i:i < \kappa
\rangle/J^{\bd}_\kappa$ then ${\mathbf P}^{\mathfrak x}_3 = {\mathbf P}^{\gx}_2$ 
by  earlier claim 

\begin{definition}
\label{j29}
1) We say ${\gx} = \langle M,\bar{\mathbf a},
(M_i, N,  
\bar{\mathbf a}^i,p_i,f_i):i < \kappa \rangle$ is a
multi-dimensionality candidate \when \,:
\mn
\begin{enumerate}
\item[$\circledast(a)$]   $(\alpha) \quad M \le_{\gs} M_i$
\sn
\item[${{}}$]  $(\beta) \quad 
M_i 
\le_{\gs} N,\bar{\mathbf a}$ a sequence of
elements of $M$
\sn
\item[${{}}$]  $(\gamma) \quad \bar{\mathbf a}^i$ a sequence of elements of $M_i$
\sn
\item[$(b)$]   $\nonfork{}{}_{M} \{M_i:i < \kappa\}$ (i.e. $\langle
  M_i:i < \kappa\rangle$ is independent over $M$ inside $N$,
 see \S6 and $M \in K   
 _{\ga}$, 
\sn
\item[$(c)$]   $f_i:M_0 \rightarrow M_i$ is an isomorphism onto $M_i$
\sn
\item[$(d)$]   $f_i \restriction M = \id_M,f_i(\bar{\mathbf a}^0) 
= \bar{\mathbf a}^i$
\sn
\item[$(e)$]   $\kappa \ge \kappa^1_{\gs},\kappa = \cf(\kappa)$ or
$\kappa > \kappa^1_{\gs} > \ell g(\bar{\mathbf a}^i)$
\sn
\item[$(f)$]   $p_i \in {\cS}^{\bs}_{\gs}(M_i),f_i(p_0) = p_i$
\sn
\item[$(g)$]   $p_i$ is based over $\bar{\mathbf a}^i$ (see \S6 of Definition
\ref{g11}(3)) 
\sn
\item[$(h)$]   $\ortp(\bar{\mathbf a}_i,M,M_i)$ is based on definable over
$\bar{\mathbf a}$

\end{enumerate}
\mn
1A) In short we say m.d.-candidate.  
 Let $M^{\gx} = M,M^{\gx}_i = M_i$, etc. and the place $\ga =
 \bar{\mathbf a}$ is $(N_{\gx},\bar{\mathbf a}_{\gx})$.

\noindent
2) We let for ${\gx}$ as above
\mn
\begin{enumerate}
\item[$(a)$]   ${\mathbf P}^{\gx}_1 = \{p:p \in {\cS}^{\bs}(N),
N \in K_{(M,\bar{\mathbf a})}$ 
and $(\forall^\infty i < \kappa)(p \perp p_i)\}$, that is if $N_{\gx}
\le_{\gs} N$ and $p' \in \cS^{\bs}_{\gs}(N)$ is parallel to $p$ then
$(\forall^\infty I < \kappa)(p' \perp p)$
where $\forall^* i$ means except $< \kappa({\gs})$ many
\sn
\item[$(b)$]   ${\mathbf P}^{\gx}_2 = \{p:p \in {\cS}^{\bs}(N),
N \in K_{(M,\mathbf{\bar a})}$ and $p \perp {\mathbf P}^{\gx}_1\}$ that is
$({\mathbf P}^{\gx}_1)^\perp$
\sn
\item[$(c)$]   ${\mathbf P}^{\gx} = \mathbf P^{\gx}_0 = 
{\mathbf P}^{\gx}_1 \cup {\mathbf P}^{\gx}_2$
\sn
\item[$(d)$]   ${\mathbf P}^{\gx}_3 = 
\{p:p \in {\cS}^{\bs}_{\gs}(N),N \in K_{\bar{\mathbf a}} 
\text{ and } p \text{ is orthogonal to } {\mathbf P}^{\gx}_2\}$ that
is $(\mathbf P^{\gx}_2)^\perp$.
\end{enumerate}
\mn
3) We call ${\gx}$ non-trivial if $p^{\gx}_0 \pm p^{\gx}_1$.
\end{definition}

\begin{remark}
It is natural to hope $\mathbf P^{\gx}_3 = \mathbf P^{\gx}_1$ but 
at present we have only $\mathbf P^{\gx}_1
\subseteq \mathbf P^{\gx}_3$.
\end{remark}

\begin{observation}
\label{j35}
For a m.d.-candidate $\gx$:
\mn
\begin{enumerate}
\item[$(a)$]   if $p^{\gx}_0 \bot p^{\gx}_1$ then $\mathbf P^{\gx}_1 
= \cup\{{\cS}^{\bs}_{\gs}(N):N \in K_{\ga}\}$ so 
$\mathbf P^{\gx}_2 = \emptyset,\mathbf P^{\gx}_3 = \mathbf P^{\gx}_1$
\sn
\item[$(b)$]   if $i < j < \kappa$ then $p^{\gx}_0 \bot p^{\gx}_1 
\Leftrightarrow p^{\gx}_i \bot p^{\gx}_j$.
\end{enumerate}
\end{observation}

\begin{PROOF}{\ref{j35}}  
Clause (a):  By \ref{g5}  
\end{PROOF}

\begin{claim}
\label{g.5a}
Assume ${\gx}$ is an m.d.-candidate.

\noindent
1) ${\mathbf P}^{\gx}_\ell$ is a $\ga_{\gx}$-based family of 
types for $\ell = 0,1,2,3$.

\noindent
2) $\mathbf P_{\gx} = {\mathbf P}^{\gx}_1 \cup {\mathbf P}^{\gx}_2$ 
is dense  
see \ref{j11}].

\noindent
3) If ${\mathbf P}^{\gx}_2 \ne \emptyset$ \then \, there is $q \in
{\mathbf P}^{\gx}_2 \cap {\cS}^{\bs}_{\gs}(M^{\gx})$. 
[even any $q \in \mathbf P^{\gx}_2$ has a $\bar{\mathbf a}$-conjugate 
$q \in \mathbf P^{\gx}_2 \cap \mathbf S^{\bs}_{\gs}(M^{\gx})$].
\end{claim}

\begin{PROOF}{\ref{g.5a}}
[The reader can concentrate 
on the case we use ${\gs}^{\eq}$].  

\noindent 
1) Note that for $p \in {\cS}^{\bs}_{\gs}(N)$, the
truth value of ``$p \in {\mathbf P}^{\gx}_1$" is definable over
$\cup\{\bar{\mathbf a}^j:j \in [i,\kappa)\}$ for any fixed $i$.  
But if $N \le_{\gs} N_1$ and for $p \in {\cS}^{\bs}_{\gs}(N_1)$ 
for some $i < \kappa$, the set 
$\{\bar{\mathbf a}^j:j \in [i,\kappa)\}$ is independent over 
$(\bar a$,base$(p))$.  By \ref{j11}(2) also $\mathbf P^{\gx} = \mathbf
  P^{\gx}_D$ is based on $\ga$ and lastly ${\mathbf P}^{\gx}_1$ is based on 
$\bar{\mathbf a}$.  As for ${\mathbf P}^{\gx}_2$
it is based on $\bar{\mathbf a}$ by \ref{j11}(1), clause (a) and its
definition.  By clause (a) of \ref{j11} the family $\mathbf P^{\gx}_3$  
are based on $\bar{\mathbf a}$. 

\noindent
2) By clause (b) of \ref{j11}(1).

\noindent
3) By \ref{g41}(2).  
\end{PROOF}

\begin{claim}
\label{j38}
Assume that $M \le_{\gs} N$ and $p \in {\cS}^{\bs}_{\gs}(N)$.  \Then
 \,  we can find a multi-dimensionality candidate ${\gx}$ such that:
\mn
\begin{enumerate}
\item[$(a)$]   $M_{\gx} = M$
\sn 
\item[$(b)$]   $M^{\gx}_0 = N$
\sn
\item[$(c)$]   $p^{\gx}_0 = p$.
\end{enumerate}
\end{claim}

\begin{claim}
\label{j41}  \label{g.6} 
Let ${\gx}$ be a non-trivial multi-dimensional candidate.  
\Then \, ${\mathbf P}^{\gx}_2$ is non-empty (hence $\mathbf P^{\gx}_2 
\cap {\cS}^{\bs}_{\gs}(M) \ne \emptyset$).
\end{claim}

\begin{PROOF}{\ref{g.6}}
 Assume it is empty so by \ref{g.5a}(2) we know that
${\mathbf P}^{\gx}_1$ is dense.

Let $N_{\gx} \le_{\gs} N \in K_{\gs}$ and let $p^+_0$ is a nonforking
extension of $p_0$ in ${\cS}^{\bs}_{\gs}(N)$.  
By \ref{j20} there is a ${\mathbf P}^{\gx}_1$-primeness 
sequence $\langle M_i,a_j:i \le \alpha,j < \alpha \rangle$ 
such that $M_0 = N,\alpha < \kappa_{\gs}$
(actually $\alpha < \kappa$ suffice) and some 
$b \in M_\alpha$ realizes $p^{\gx}_0$; note that $M_i,M^{\gx}_j$ are
not directly related.

Now as $\kappa_{\gx}$ is regular $\ge \kappa_{\gs}$ because
$\alpha < \kappa^{\gx}$; by the definition of ${\mathbf P}^{\gx}_1$ each
$j < \alpha$, for every $\varepsilon < \kappa_{\gx}$ large
enough we have $p^{\gx}_\varepsilon \perp
\ortp_{\gs}(a_j,M_j,M_{j+1})$ say for $\varepsilon \in
[\zeta_j,\kappa_{\gx})$.  As $\kappa_{\gx}$ is regular $> \alpha$, we
  have $\zeta = \sup\{\zeta_j:j < \alpha\}$ is $< \kappa_{\gx}$.  For
$\varepsilon \in [\zeta,\kappa^{\gx})$ let 
$p^+_{\varepsilon,i} \in {\cS}^{\bs}_{\gs}(M_i)$ be the 
nonforking extension of $p^{\gx}_\varepsilon$ in ${\cS}_{\gs}(M_i)$.  Now we
can prove by induction on $i \le \alpha$ that $p^+_{\varepsilon,i}$ is
the unique extension of $p^{\gx}_{\varepsilon,0}$ 
in ${\cS}_{\gs}(M_i)$.  As $b \in M_\alpha$ by \ref{g41}(1),
$\ortp_{\gs}(b,M_0,M_\alpha) \in {\cS}^{\bs}_{\gs}(M_i)$
is orthogonal to $p^{\gx}_\varepsilon$.  As $\ortp_{\gs}
(b,M_0,M_\alpha) = p^+_0$, we 
can conclude that $p^+_0,p^+_{\varepsilon,0}$ are orthogonal hence
$p^{\gx}_0,p^{\gx}_\varepsilon$ are orthogonal which by
\ref{j35}(b) is a contradiction to ``${\gx}$ non-trivial".
\end{PROOF}

\begin{claim}
\label{j44} \label{g.7}  
Let ${\gx}$ be a non-trivial multi-dimensional candidate, 
$\kappa = \kappa_{\gx}$.  

\noindent
1) There is $q \in {\cS}^{\bs}_{\gs}(M_{\gx})$
strongly dominated by $\{p^{\gx}_i:i < \kappa\}$. 

\noindent
2) If $N_{\gx} \le_{\gs} M \in K_{\gs}$ and $q \in {\cS}_{\gs}(M) 
\cap {\mathbf P}^{\gx}_2$ and $p^+_i \in {\cS}_{\gs}(M)$ is a
nonforking extension of $p^{\gx}_i$ for $i < \kappa_{\gx}$ 
\then \, $q \le^{\dm}_{\st} \{p^+_i:i \in u\}$ for 
some $u \subseteq \kappa^{\gx}$ [introduction: connect
\cite[Ch.V,\S5]{Sh:c}, \cite{Sh:429}.]

\noindent
3) \Wilog \, base$(q) \subseteq M_{\gx}$] [We can find $\mathbf b 
= \mathbf{\bar b}_\kappa \subseteq M_{\gx}$ such that $q$ is 
definable over $\bar{\mathbf b}$ and $\mathbf{\bar c} \subseteq M$ such that
$\ortp(\mathbf{\bar b}_\kappa,M^{\gx},M)$ is 
definable over $\mathbf{\bar c}$ and also $\mathbf{\bar b}_i 
\subseteq M$ for $i < \kappa$ such that $\langle \mathbf
b_i:i < \kappa \rangle$ is indiscernible based on $\bar{\mathbf c}$ and $q_i
\in {\cS}^{\bs}_{\gs}(M^{\gx})$ is definable over $\bar{\mathbf b}_i$ as $q$
is definable over $\mathbf{\bar b}$.  From this it follows that 
$q \pm q_i$; moreover, $q$ is strongly dominated by $\{q_i:i < i^*\}$ 
for some $i^* < \kappa$.]

\noindent
4) We can find a m.d.-candidate $\gy$ such that $M_{\gx} \le_{\gs}
M_{\gy},\bar{\mathbf a}_{\gy} = \bar{\mathbf a}_{\gx}$ and $P^{\gy}$ is
parallel to $q$. 
\end{claim}

\begin{PROOF}{\ref{g.7}}  
1) By \ref{g.6} there is $M$ such that $N_{\gx}
\le_{\gs} M$ and $q \in {\cS}^{\bs}_{\gs}(M) \cap \mathbf P^{\gx}_2$.  
By possibly replacing $\gx$ by a neighborhood, \wilog \, $q$ does 
not fork over $M$ so by part (2) we have $q  \le^{\dm}_{\st}
\{p^{\gx}_i:i \in u\}$, where $u \in [\kappa^{\gx}]^{< \kappa({\gs})}$.  

\noindent
2) Let ${\mathbf P}'_3 = \{p \in {\mathbf P}^{\gx}_3$: there
is $p' \in {\cS}^{\bs}(M)$ parallel to $p\}, 
{\mathbf P}''_3 = \{p \in {\mathbf P}^{\gx}_3:p$ is orthogonal to $M\}$.

We may consider using 
the ``not $\kappa_{\mathfrak s}$-forking" version. 

For $\zeta < \kappa = \kappa_{\gx}$ let 
${\mathbf P}^*_\zeta = {\mathbf P}'_3 \cup {\mathbf P}''_3 \cup \{p^+_\varepsilon:
\varepsilon <\kappa$ and $\varepsilon > \zeta\}$ and let $b$ realize
$q$.  

Now 
\mn
\begin{enumerate}
\item[$\circledast_\zeta$]    ${\mathbf P}^*_\zeta$ is $M$-based which
is dense above $M$.
\end{enumerate}
\mn
[Why?  See in the end of the proof].

Fixing $\zeta$ (its value is immaterial), we try
by induction on $i < \kappa_{\gs}$ to choose $M_i,N_i$ and 
$a_j$ for $j<i$ such that:
\mn
\begin{enumerate}
\item[$\boxplus$]  $(a) \quad \langle M_\varepsilon,a_j:\varepsilon \le i,j <i
\rangle$ is a ${\mathbf P}^*_\zeta$-primeness sequence
\sn
\item[${{}}$]   $(b) \quad M_0 = M,M_i \le_{\gs} N_i,N_i$ is 
$\le_{\gs}$-increasing, $\|N_i\| \le \lambda$
\sn
\item[${{}}$]   $(c) \quad q = \ortp_{\gs}(b,M_0,N_0)$
\sn
\item[${{}}$]   $(d) \quad \ortp_{\gs}(b,M_{i+1},N_{i+1})$ forks over $M_i$
\sn
\item[${{}}$]   $(e) \quad$ if $\ortp_{\gs}(a_j,M_j,M_{j+1})$ is a 
non-forking extension of $p^+_\varepsilon$ equivalent of
$p^{\gx}_\varepsilon$

\hskip25pt  for some $\varepsilon \in [\zeta,\kappa)$ then
$\ortp_{\gs}(a_j,M,M_{j+1}) \notin \{\ortp_{\gs}
(a_i,M,M_{i+1}):i < j\}$. 
\end{enumerate}
\mn
So for some $\alpha < \kappa_{\gs}$ we have 
$\langle M_j:j \le
\alpha\rangle$ but we cannot proceed (by the demand of not too long
forking).  Let $\xi = \sup\{\varepsilon +1:\varepsilon < \kappa$
and $\varepsilon = \zeta =1$ for some $j < \alpha,\ortp(a_j,M_j,M_{j+1})$ is a
non-forking extension of $p_\varepsilon\}$.  As ${\mathbf P}^*_\zeta$ is
dense clearly $q$ is realized by some $b \in M$.

Let $u_0 = \{i < \alpha:\ortp_{\gs}(a_i,M_i,M_{i+1})$ is a nonforking
extension of some $p^+_\varepsilon,\varepsilon < \kappa\}$ for $i \in
u_0$ let $\varepsilon(i) < \kappa$ be such that
$a_i$ realizes $p^+_{\varepsilon(i)}$, so by clause (e) above $\langle
\varepsilon(i):i \in u_0 \rangle$ is without repetitions.  Let

\[
u_1 = \{i < \alpha:\ortp_{\gs}(a_i,M_i,M_{i+1}) \in {\mathbf P}'_3\}
\]

\[
u_2 = \{i < \alpha:\ortp_{\gs}(a_i,M_i,M_2) \in {\mathbf P}''_3\}.
\]

\mn
So $\langle u_0,u_1,u_2 \rangle$ is a partition of $\alpha$ and
clearly:
\mn
\begin{enumerate}
\item[$(*)_1$]   $\{a_i:i \in u_0 \cup u_1\}$ is independent in
$(M,M_\alpha)$.
\end{enumerate}
\mn
[Why?  This holds as $i \in u_0 \cup u_1$ implies $p'_i = 
\ortp_{\mathfrak s}(a_i,M_i,M_{i+1})$ does not fork over $M$.  This
implication holds because:
\mn
\begin{enumerate}
\item[$(\alpha)$]   if $i \in u_1$ then $p'_i$ is parallel to
some member of ${\cS}^{\bs}_{\gs}(M)$, hence $p'_i$ does not fork over $M$
\sn
\item[$(\beta)$]   if $i \in u_0$ then $p'_i$ is parallel to
some $p^+_\varepsilon$ but $p_\varepsilon \in 
{\cS}^{\bs}_{\gs}(M)$, so we are done.]
\sn
\item[$(*)_2$]   $i \in u_2 = u \backslash u_0 \cup u_1 \Rightarrow 
\ortp_{\gs}(a_i,M_i,M_{i+1}) \perp M$
\end{enumerate}
\mn
[Why?  By the definition of $u_2$ and $\mathbf P''_3$.]
\mn
\begin{enumerate}
\item[$(*)_3$]   $i \in u_1 \Rightarrow \ortp_{\gs}(a_i,M,M_{i+1}) \perp q$.
\end{enumerate}
\mn
[Why?  As $\ortp_{\gs}(a_i,M,M_{i+1}) \in \mathbf P'_3 \subseteq \mathbf
  P^{\gx}_3$ whereas $q \in \mathbf P^{\gx}_2$, recalling the definition
  of $\mathbf P^{\gx}_3$.]

Using \ref{j23}(d) 
by $(*)_1 + (*)_2$ and (a) above
we have $q \le \{p_i:i \in u_0 \cup u_1\}$. 

By \ref{j23}(2) and $(*)_3$ it follows $q \le^{\dm}_{\st}
\{\ortp_{\gs}(a_i,M,M_{i+1}):i \in u_0\} =
\{p_{\varepsilon(i)}:i \in u_0\}$ but $\langle \varepsilon(i):
\varepsilon(i) \in u_0 \rangle$ is without repetition. 
\end{PROOF}

So we are done except one debt: $\circledast_\zeta$.

\begin{proof}
\underline{Proof of $\circledast_\zeta$}  

Towards contradiction assume $r \in {\cS}_{\gs}(N),M \le_{\gs} 
N \in K_{\gs}$ and $r$ is orthogonal to every
member of ${\mathbf P}^*_\zeta$.  As $r \perp p_i$ for $i \in
[\zeta,\kappa)$ clearly $r \in {\mathbf P}^{\gx}_1$, so by the
definitions of $\mathbf P^{\gx}_1,\mathbf P^{\gx}_2$ we have  $r \perp 
{\mathbf P}^{\gx}_2$ hence $r \in {\mathbf P}^{\gx}_3$.
\end{proof}
\bigskip

\noindent
\underline{Case 1}:  $r \perp M$.

So $r \in {\mathbf P}''_3 \subseteq {\mathbf P}^*_\zeta$.
\bigskip

\noindent
\underline{Case 2}:  $r \pm M$.

We can find $r' \in {\cS}(M)$ which is a good enough ``reflection
of $r$ over $\bar{\mathbf a}_{\gx}$", hence $r \pm r'$ but still 
$r' \in {\mathbf P}^{\gx}_3$.
How?  Recalling that ${\mathbf P}^{\gx}_3$ is $\ga$-based, $A \subseteq M$
small enough, this is proved as in \ref{g41} but we elaborate.

\noindent
3) Again as in \ref{g41}.

\begin{conclusion}
\label{j50}
1) If ${\mathbf P}$ is $\ga$-based, $M \in K_{\ga}$ and ${\mathbf P}$ is dense,
$q_* \in {\cS}_{\gs}(M)$ and $(\forall p \in {\mathbf P})(p \pm q_* 
\Rightarrow p$ does not fork over $M$), \then \, there are 
$\alpha < \kappa_{\gs},p_i \in {\mathbf P} \cap {\cS}_{\gs}(M)$ for $i < \alpha$ 
such that $p_i \pm q^*$ for $i < \alpha$ and $q^*$ is weakly 
dominated by $\{p_i:i < \alpha\}$. 

\noindent
2) Assume ${\mathbf P} \subseteq {\cS}^{\bs}(M)$ is based on some
small $A \subseteq M,q_* \in \cS^{\bs}_{\gs}(M)$ and if 
$p \in {\cS}^{\bs}_{\gs}(M)$ is orthogonal to ${\mathbf P}$ then it is
orthogonal to $q_*$.  \Then \, we can find $\langle p_i:
i < \alpha \rangle,\alpha < \kappa_{\gs}$ as in part (1). 
\end{conclusion}

\begin{PROOF}{\ref{j50}}
1) Let ${\mathbf P}^* =: {\mathbf P}_1 \cup {\mathbf P}_2 \cup {\mathbf P}_3$ where

\[
{\mathbf P}_1 = \{p:p \text{ is parallel to some } p' \in 
{\cS}^{\bs}_{\gs}(M) \text{ which belong to } {\mathbf P}\}
\]

\[
{\mathbf P}_2 = \{p:p \perp {\mathbf P}_1 \text{ and }
p \text{ is parallel to some } p' \in {\cS}^{\bs}_{\gs}(M)\}
\]

\[
{\mathbf P}_3 = \{p:p \perp M\}.
\]

\mn
We would like to apply \ref{j23}(3) or \ref{j23}(1).  
Clearly ${\mathbf P}^*$ is based on $M$ and $(**)$ of \ref{j23}(3) holds.

The main point is to prove that ${\mathbf P}^*$ is dense above $M$.  
So let $M \le_{\gs} N$ and $q \in {\cS}^{\bs}_{\gs}(N)$.  
If $q \bot M$ then $q \in {\mathbf P}_3 \subseteq {\mathbf P}^*$ so O.K.  If $q
\pm {\mathbf P}_1$ then there is $p_0 \in \mathbf P \subseteq 
{\cS}^{\bs}_{\gs}(M)$ and $p \in {\cS}_{\gs}(N)$
 an extension of $p_0$ which does not fork over $M$ such that 
$p \pm q$, clearly $p \in \mathbf P_1 \subseteq {\mathbf P}^*$
so O.K.  Hence we are left with the case $q \pm M,q \perp {\mathbf P}_1$.

If $q \pm {\mathbf P}_2$ then we can find $p' \in
{\cS}^{\bs}_{\gs}(M)$ orthogonal to ${\mathbf P}$ but not to $p$, so $p'
\restriction N \in {\mathbf P}_2$ is not orthogonal to $q$.

So assume $q \perp {\mathbf P}_2$.  We can find $N^+,\langle f_i:i < \kappa
\rangle,\langle N_i:i < \kappa \rangle$ such that $N_0=N,M \le_{\gs} 
N_i \le_{\gs} N^+,\langle N_i:i < \kappa \rangle$ is
independent over $M$ inside $N^+,f_i$ an isomorphism from $N_0$ onto
$N_i$ over $M$.  Let $q_i = f_i(q)$, so for some ${\mathfrak x},M^{\gx} 
= M,p^{\gx}_i = q_i$.  If ${\gx}$ is trivial, then $i < j
\Rightarrow q_i \perp q_j$ then by earlier claim , 
$q_i \perp M$, 
contradiction to a statement above.  So ${\gx}$ is
non-trivial hence by \ref{j44}(1)
there is $q' \in {\cS}^{\bs}_{\gs}(M)$ dominated by $\{q_i:i <
\kappa\}$ not orthogonal to each $q_i$.  For $\ell=1,2$ as 
$i < \kappa \Rightarrow q_i \perp {\mathbf P}_\ell$ by \ref{g41} also 
$q' \perp {\mathbf P}_\ell$.  By the definition of ${\mathbf P}_2$ 
as $q' \in {\cS}_{\gs}(M)$ we have $q' \in {\mathbf P}_2$ but 
$q' \perp {\mathbf P}_2$, contradiction. 

\noindent
2) Let ${\mathbf P}' = \{p:p$ parallel to some $p' \in {\mathbf P} \cap
{\cS}^{\bs}_{\gs}(M)$ \underline{or} 
$p$ is orthogonal to $M\}$.  Now ${\mathbf P}'$ is $M$-based and 
it is dense above $M$ (as if $q \in {\cS}^{\bs}_{\gs}(M),
M \le_{\gs} N$, either $q \perp M$ so $q \in M'$
or by \ref{g41}(2) there is $q' \in {\cS}^{\bs}_{\gs}(M),q' \pm
q_*,[q' \perp {\mathbf P} \cap {\cS}^{\bs}(M) \Leftrightarrow q
\perp {\mathbf P} \cap {\cS}^{\bs}_{\gs}(M)]$.  

\noindent
So applying (1) we are done.  
\end{PROOF}

\begin{claim}
\label{j53}
1) If $p_1,p_2 \in {\cS}^{\bs}_{\gs}(M)$ are not orthogonal 
\then \, some $r \in {\cS}^{\bs}_{\gs}(M)$ 
is weakly dominated by $p_1$ and weakly dominated by $p_2$. 

\noindent
2) If $p,q \in {\cS}^{\bs}_{\gs}(M)$ and $(\forall r \in 
{\cS}^{\bs}_{\gs}(M))(r \perp p \Rightarrow r \perp q)$ \then \,
$p$ weakly dominates $q$ (and, of course, the inverse is trivial).

\noindent
3) If $p,q \in {\cS}^{\bs}_{\gs}(M)$ 
and every $r \in {\cS}^{\bs}_{\gs}(M)$ weakly
dominated by $q$ is not orthogonal to $p$ \then \, $p$ weakly
dominates $q$.
\end{claim}

\begin{PROOF}{\ref{j53}}
1) We shall rely on parts (2) + (3).
For $\ell=1,2$ let $\{q^\ell_i:i < \alpha_\ell\}$ be
a maximal set of pairwise orthogonal types from 
${\cS}^{\bs}_{\gs}(M)$, each weakly dominated by $p_\ell$ 
and orthogonal to $p_{3 - \ell}$.  So
$\alpha_\ell < \kappa_{\gs}$ by \ref{g41}(x).  

Let ${\mathbf P}_0 = \{q:q$ orthogonal to $p_1\}$

\[ 
{\mathbf P}_1 = \{q:q \text{ parallel to some } q^1_i,i < \alpha\}.
\]

\mn
Clearly ${\mathbf P}_0 \cup {\mathbf P}_1$ is $A$-based for some $A
\subseteq M,|A| < \kappa_{\gs}$.  First assume that there is no $r
\in {\cS}^{\bs}_{\gs}(M)$ orthogonal to 
${\mathbf P}_0 \cup {\mathbf P}_1$.  By \ref{j50}(2) 
there are $\alpha < \kappa_{\gs}$ and $r_i \in ({\mathbf P}_0 
\cup {\mathbf P}_1) \cap {\cS}^{\bs}_{\gs}(M)$ for $i < \alpha$ such
that $i < \alpha \Rightarrow r_i \pm p_1$ and $p_1$ is weakly dominated by
$\{r_i:i < \alpha\}$. Necessarily $i < \alpha \Rightarrow r_i \in
\{q^1_i:i < \alpha\}$ hence $p_1$ is weakly dominated by $\{q^1_i:i <
\alpha_1\}$.  But $i < \alpha_1 \Rightarrow q^1_i \perp p_2$
hence by \ref{g41}(x) $p_1 \perp p_2$, a contradiction.

Second, assume that there is $r \in {\cS}^{\bs}_{\gs}(M)$ 
orthogonal to ${\mathbf P}_0 \cup {\mathbf P}_1$.  As 
$r \perp {\mathbf P}_0$, by part (2) below $r$ is weakly 
dominated by $p_1$.  As $r$ is also orthogonal to ${\mathbf P}_1$, it
satisfies the first two demands on $q^1_{\alpha_1}$, so by $\{q^1_i:i
< \alpha_1\}$ maximality, clearly $r$ is not orthogonal to $p_2$.  By
the maximality of $\{q^1_i:i < \alpha_1\}$ clearly if $r_1$ is weakly
dominated by $r$ then $r_1 \pm p_2$.  Hence by part (3), $r$ is weakly
dominated by $p_2$, so $r$ is as required.

\noindent
2) Let ${\mathbf P}_0 =: \{r:r \perp p\},{\mathbf P}_1 = \{p':p'$ parallel to
$p\}$, so trivially ${\mathbf P}_0 \cup {\mathbf P}_1$ is
dense $M$-based.  Hence by \ref{j50}(1) there are 
$\{r_i:i < \alpha\} \subseteq ({\mathbf P}_0 \cup
{\mathbf P}_1) \cap {\cS}^{\bs}_{\gs}(M)$ which weakly 
dominates $q$  and $i < \alpha \Rightarrow r_i \pm q$.  
But every $r \in {\mathbf P}_0$ is orthogonal
to $p$ hence by the assumption of part (2), is orthogonal to 
$q$, hence $i < \alpha \Rightarrow r_i \in
{\mathbf P}_0 \cup {\mathbf P}_1 \backslash {\mathbf P}_0 = {\mathbf P}_1$, so $i
< \alpha \Rightarrow r_i \|p$.  As $i <\alpha \Rightarrow r_i \in
{\cS}^{\bs}_{\gs}(M)$ we get $i < \alpha \Rightarrow 
p_i = p$.  So $q$ is dominated by $p$. 

\noindent
3) Let ${\mathbf P}_0 =:\{r:r \perp q\},{\mathbf P}_1 = \{p':p'\|p\}$.  Now
${\mathbf P}_0 \cup {\mathbf P}_1$ is $M$-based.
Also ${\mathbf P}_0 \cup {\mathbf P}_1$ is dense.  

\noindent
[Why?  Assume $r$ is orthogonal to $\mathbf P_0 \cup \mathbf P_1,r \bot
\mathbf P_0$ which means: if $r'$ is orthogonal to $q$ then it is
orthogonal to $r$ (by the definition of $\mathbf P_0$) hence by part (2)
we know that $r$ is weakly dominated by $q$.  But by the assumption of
part (3) we know $r$ is not orthogonal to $p$ hence $r$ is not
orthogonal to ${\cP}_1$, contradiction.]

By \ref{j50} there are $\alpha < \kappa_{\gs}$ and $r_i \in
({\mathbf P}_0 \cup {\mathbf P}_1) \cap {\cS}^{\bs}_{\gs}(M)$ 
such that $q$ is dominated by $\{r_i:i < \alpha\}$ and $i 
< \alpha \Rightarrow r_i \pm q$.  
So $r_i \notin {\mathbf P}_0$ hence $r_i \in {\mathbf P}_1$, i.e., $r_i
= p$, so $p$ weakly dominates $q$ as required.
\end{PROOF}

\begin{conclusion}
\label{j56}
Assume $M_0 \le_{\gs} M_1 \le_{\gs} N$ and 
$p \in {\cS}^{\bs}_{\gs}(N)$ is orthogonal to $M_0$ but not to
$M_1$.  \Then \, there is $q \in {\cS}^{\bs}_{\gs}(M_1)$ 
orthogonal to $M_0$ but not to $p$.
\end{conclusion}

\begin{PROOF}{\ref{j56}}
Let $\kappa = \kappa_{\gs}$.

Without loss of generality $N$ is brimmed 
over $M_1$, so \wilog \, $N$ is brimmed over $N_0$ such that $p$ does
not fork over $N_0$ and $N_0 \le N$; let $\bar a_0$ list
$N_0$, so there are $\bar a_i \in {}^{\kappa({\gs})>} N$ 
for $i \in [1,\kappa({\gs}))$ such that
$\langle \bar a_i:i < \kappa_{\gs}\rangle$ is an indiscernible sequence
over $M_1$ based on $\ortp_{\gs}(\bar a_0,M_1,N_1)$ and let 
$p_i \in {\cS}^{\bs}_{\gs}(N)$ be
definable over $\bar a_i$ as $p$ was definable over $\bar a$.  As $p_0
= p \perp M_0$ clearly $i<\kappa \Rightarrow p_i \perp M_0$ and as
$p_0 \pm M_1$ clearly ($p_i \pm M$ and) 
$i<j<\kappa_{\gs} \Rightarrow p_i \pm p_j$. So by
\ref{j44}(1) there is $q \in {\cS}^{\bs}_{\gs}(M_1)$ 
such that $q$ is dominated by $\{p_i:i < \kappa_{\gs}\}$.  
Now as $q$ is dominated by a set of types orthogonal to 
$M_0$ by \ref{g41}(x)  also $q$ is orthogonal to $M_0$.  Also as
$q$ is dominated by $\{p_i:i<\kappa_{\gs}\}$ for some $i,q \perp p_i$, but
by the choice of the $\bar a_i,p_i$ this holds for every $i$ in
particular $q \pm p_0 = p$, so we are done.
\end{PROOF}
\newpage

\centerline {Part IV} \label{part4}

\section {For weakly successful to NF-frames} \label{9}

\begin{discussion}
We may like to see that in (G) or \ref{g23}, (b)
is redundant, imitating earlier proofs.

\noindent
Should we add $K^{3,\qr}_{\gs}$?
\end{discussion}

\begin{hypothesis}
\label{k5}
${\gs}$ is a weakly successful good
$(\mu,\lambda,\kappa)$-frame with primes over chains.

Our aim is to define $\NF_{\gs}$ and prove that it satisfies the
relevant properties, as in \cite[\S7]{Sh:600}.  Also ${\gs}^+$ will
have the density for $K^{3,\pr}_{\gs(*)}$.  
\end{hypothesis}

\begin{claim}
\label{k8}
If $M_\ell \le_{\gs} M$ and $p \in {\cS}^{\bs}(M_\ell)$ for 
$\ell =1,2$, \then \, there is $r \in {\cS}^{\bs}(N)$ dominated by both.
\end{claim}
\bigskip

\begin{PROOF}{\ref{k8}}  
For $\ell=1,2$ let ${\mathbf P}_\ell = \{r:r \in {\cS}(N),r$ 
orthogonal to $p_{3 -\ell}$.
\end{PROOF}

\begin{claim}
\label{k11}
1) Let ${\mathbf P}$ be $A$-based and ${\mathbf P}^\perp =
\{p \in {\cS}^{\bs}(M):p \perp {\mathbf P}\}$.
\end{claim}
\newpage 

\section {Strong stability, Weak form of superstability} \label{10}

\bigskip

The reader can think of the first order case.
\begin{definition}
\label{m2}
For stable $T$ let $\kappa_{\ict}(T)$
be the minimal $\kappa$ such that if $\mathbf I \subseteq {\gC}_T$
is independent over $A,\bar b \in {}^{\omega >}{\gC}$ (we allow
$\mathbf I$ to consist of infinite sequences, too) \then \, for some
$\mathbf J \subseteq \mathbf I$ of cardinality $< \kappa$ the set $\mathbf I
\backslash \mathbf J$ is independent over $(A \cup \mathbf J \cup \{\bar b\},A)$.

\noindent
2) For stable $T$ let $\kappa_{\ict}(T)$ is defined similarly except in the end
 we require just $\bar b \in \mathbf I \backslash \mathbf J \Rightarrow
\ortp(\bar c,A+ \bar b)$ does not fork over $A$.
\end{definition}

\begin{remark}
\label{h5}
Are there suitable $K^{\text{bs}}_\kappa$-templates.
The case $K^{\tr}_\omega$ is not clear.
\end{remark}

\begin{definition}
\label{m8}
1) We say $p$ is pseudo regular (or 1-reg) when: there is 
$(M,N,a) \in K^{3,\pr}_{\gs}$ such that:
\mn
\begin{enumerate}
\item[$(a)$]   $\ortp(a,M,N)$ is parallel to $p$
\sn
\item[$(b)$]   $(M,N,a)$ is pseudo regular which means there is
no $\mathbf J$ such that $(M,N,\mathbf J) \in K^{3,\bs}_{\gs}$
and $|\mathbf J| \ge 2$.
\end{enumerate}
\mn
2) We say $p$ is almost regular or 2-regular when for every $(M,N,a)
\in K^{3,\pr}_{\gs}$ representing $p$ and $b \in N \backslash
M$ we have $(M,N,b) \in K^{3,\pr}_{\gs}$.
\end{definition}

\begin{claim}
\label{m11}
In Definition \ref{m8}, the choice of $(M,N,a)$ is immaterial.
\end{claim}

\begin{proof}  
Left to the reader.  
\end{proof}

\begin{claim}
\label{m14}
$[\kappa_{\jct}({\gs}) = \aleph_0]$.

\noindent
1) If $(M,N,a) \in K^{3,\pr}_{\gs}$ \then \, for some $b
\in N \backslash M$ and $N' \le_{\gs} N$ we have $(M,n,b) \in
K^{3,\pr}_{\gs}$ is pseudo-regular.

\noindent
2) Moreover, almost regular.
\end{claim}

\begin{PROOF}{\ref{m14}}
1) We try to choose $(b_n,c_n)$ by induction on $n$ such that:
\mn
\begin{enumerate}
\item[$\circledast$]   $(a) \quad b_n \ne c_n \in N  \backslash M
\backslash \{b_\ell,c_\ell:\ell < n\}$
\sn
\item[${{}}$]   $(b) \quad \{b_0,\dotsc,a_{n-1},b_n,c_n\}$ is
independent in $(M,N)$.
\end{enumerate}
\mn
If we succeed, we get that
\mn
\begin{enumerate}
\item[$\circledast_1$]   $\{a,b_n\} \subseteq N \backslash M$ is not
independent in $N$ over $M$.
\end{enumerate}
\mn
\begin{enumerate}
\item[$\circledast_2$]   $\{b_n:n <\omega\}$ is independent in
  $(M,N)$.
\end{enumerate}
\mn
[Why?  By $\circledast(b) +$ the finite character of independence.  

So it is enough to carry the induction.  If a pair $(b_0,c_0)$ does
not exist thta $(M,N,a)$ is pseudo regular, so $b=a$ is O.K.  If
$b_0,\dotsc,b_{n-1},b_n,c_n$ have been defined let $N^*_s \le_{\gs} N$ 
be such that $(M,N_n,c_n) \in K^{3,\pr}_{\gs}$ 
If this triple is pseudo regular choose $b := c_n$ and we
are done.  Otherwise there is $\mathbf J$ such that $(M,N_n,\mathbf J) \in
K^{3,\bs}_{\gs}$ and $|\mathbf J_n| \ge 2$.

(Note: $\{b_n,c_n\}$ is independent in $(M,N)$ hence $b_\ell \notin N_n$).

Choose $b_{n+1} \ne c_{n+1} \in \mathbf J$, now 
$\{b_0,\dotsc,b_n,b_{n +  
1},c_{n +  
1}\}$ is independent in $(M,N_0)$ so we
are done.  

\noindent
2) We try to choose $(b_n,c_n,N_n,M_n)$ by induction on $n$ (can waive $c_0$)
 such that
\mn
\begin{enumerate}
\item[$\circledast$]   $(a) \quad b_n \in N \backslash M,M
\le_{\gs} M_n \le_{\gs} N$ and $(M,M_n,b_n) \in K^{3,\pr}_{\gs}$
\sn
\item[${{}}$]   $(b) \quad b_0 = a,M_0 = N = N_0$ we may use $c_0 = a$
  but this is immaterial
\sn
\item[${{}}$]   $(c) \quad N \le_{\gs} N_n$ and $\ell < n
\Rightarrow N_\ell \le_{\gs} N_n$
\sn
\item[${{}}$]  $(d) \quad \{c_1,\dotsc,c_n,b_n\} \subseteq N_n$ is
independent in $(M,N_n)$ and is with no repetition
\sn
\item[${{}}$]   $(e) \quad \{a,c_\ell\}$ is dependent in $(M,N_n)$
\sn
\item[${{}}$]   $(f) \quad \ortp(c_{\ell +1},N_\ell,N_{\ell +1})$ does
  not fork over $M$.
\end{enumerate}
\mn
If we succeed to carry the induction, $M,N_\omega$ prime over $\langle
N_n:n < \omega\rangle,a$ and $\langle c_\ell:\ell \in [1,\omega)\rangle$
contradicts the assumption.  [We assume/use: independency has local character.]

For $n=0$ there is no problem by clause (b).  So assume we have chosen
$c_1,\dotsc,c_n,b_n,M_n,N_n$.  If $\ortp(b_n,M,M_n)$ is almost regular we
are done so there is $b \in M_n \backslash M$ contradicting it, call
it $b_{n+1}$ and choose $M_{n+1} \le_{\gs} M_n$ such that
$(M,M_{n+1},b_{n+1}) \in K^{3,\pr}_{\gs}$.  By  
an earlier Claim 
there is a pair $(N^*_n,M^*_n,\mathbf J_n)$ such that
\mn
\begin{enumerate}
\item[$\boxtimes$]   $(a) \quad N_n \le_{\gs} N^*_n$
\sn
\item[${{}}$]   $(b) \quad M \le_{\gs} M_n \le_{\gs} M^*_n \le_{\gs} N^*_n$
\sn
\item[${{}}$]   $(c) \quad (M,M^*_n,\mathbf J_n) \in K^{3,\pr}_{\gs}$ 
\sn
\item[${{}}$]   $(d) \quad b_{n+1} \in \mathbf J$
\sn
\item[${{}}$]   $(e) \quad \{N_n,M^*_n\}$ is independent over $M_n$.
\end{enumerate}
\mn
Let $\mathbf J'_n = \mathbf J_n \backslash \{b_{n+1}\}$.
\bigskip

\noindent
\underline{Case 1}:  $\{a\} \cup \mathbf J'_n$ is independent in $(M,N^*_n)$.

Then: $\mathbf J'_n$ is independent in $(M,N,N^*_n)$ hence $\mathbf J'_n
\cup \{b_n\}$ is independent in $(M,N^*_n)$ but $M_n \le_{\gs} N^*_n$
and $(M,M_n,b_n) \in K^{3,\pr}_{\gs}$  hence $\mathbf J'_n$
is independent in $(M,M_n,M^*_n)$.  Let $M^{**}_n \le_{\gs} M^*_n$
be such that $(M_{n+1},M^{**}_n,\mathbf J'_n) \in K^{3,\pr}_{\gs}$ hence
$\ortp(b_n,M^{**}_n,M^*_n)$ does not fork over $M_{n+1}$ and is
orthogonal to $M$ hence (see 
earlier calim)  
(clear but we can also
change definition \ref{m2}), contradiction to the choice of $b_{n+1}$.
\bigskip

\noindent
\underline{Case 2}:  Not Case 1.

Let $N_{n+1} = N^*_n$.  We can find finite $\mathbf J''_n \subseteq
\mathbf J'_n$ such that $\{a\} \cup \mathbf J'_n$ is not independent in
$(M,N^*_n)$.  We can replace $\mathbf J''_n$ by one element and call it
$c_{n+1}$.  
\end{PROOF}

\begin{claim}
\label{m20}
$[\kappa_{\rct}({\mathfrak s}) = \aleph_0]$.  For every $M \in 
K_{\mathfrak s}$ and $p \in {\cS}(M)$ for some $N,\mathbf J$ we have:
\mn
\begin{enumerate}
\item[$\circledast_1$]   $(a) \quad (M,N,\mathbf J) \in K^{3,\pr}_{\gs}$
\sn
\item[${{}}$]  $(b) \quad c \in \mathbf J \Rightarrow \ortp
(c,M,N)$ is almost regular
\sn
\item[${{}}$]   $(c) \quad \mathbf J$ is finite
\sn
\item[$\circledast_2$]   $p$ is realized in $N$.
\end{enumerate}
\end{claim}

\begin{proof}
By earlier claim.  
\end{proof}

\begin{claim}
\label{h.5}
${\mathbf P}_{2-\text{reg}}$ is a basis.
\end{claim}

\begin{definition}
\label{m23}
1) ${\mathbf P}_{\ell-\text{reg}}= \{p:p \in S(N'),N \le_{\gs} N',p$ 
is almost $\ell$-regular$\}$.

\noindent
2) $p \in {\cS}^{\na}(M)$ is 3-regular if where: for no $q_1
\perp q_2 \in {\cS}^{\na}(M)$ do we have $q_2 \pm p \perp q_2$.

\noindent
3) $P$ is 4-regular if it is 3-regular and 2-regular.
\end{definition}

\begin{claim}
\label{m26}
1) If $q^\infty$ dominates $p^\infty$ and $q$ is 3-regular \then \,
$p$ is 3-regular.

\noindent
2) [$\kappa_{\rct}(\gs) = \aleph$] ${\mathbf P}_{4-\text{reg}}$ is a basis.

\noindent
3) If $p,q \in {\cS}^{\na}_{\gs}(M)$ are 4-regular not orthogonal
\then \, $p^\infty,q^\infty$ are equivalent.
\end{claim}

\begin{PROOF}{\ref{m26}}  
Straightforward.
\end{PROOF}

\begin{remark}
\label{m29}
1) So we have the main gap for when $\kappa_{\ut}(T)$ but this 
is not a natural assumption ??

\noindent
2) Proof of the minimal $J$ (in \ref{m2}) see 
earlier claim,  
etc.

\begin{equation*}
\begin{array}{clcr}
\mathbf J_0 = \cup\{\mathbf J'' \subseteq \mathbf I'':&(\alpha) \quad \bar b
\in \mathbf J' \text{ or } \mathbf J' \cup \{\bar b\} \text{ not
independent} \\
  &(\beta) \quad \mathbf J' \text{ finite} \\
  &(\gamma) \quad \mathbf J' \text{ minimal under this}\}.
\end{array}
\end{equation*}

\mn
$\mathbf J_{n+1}$ like $\mathbf J_0$ for $\mathbf I \backslash \{\mathbf
J_\ell:\ell \le n\},(A,A \cup \{\mathbf J_\ell:\ell \le n\}$.

$\mathbf J = \cup\{\mathbf J_n:n < \omega\}$.
\end{remark}
\newpage

\section {Decomposition} \label{11}

\begin{discussion}
\label{n2}
Is this phenomena possible?  I.e., can we type
$\langle p_{\alpha,i}:\alpha < \lambda^+_{\gs},i < i^* < \kappa
\rangle$ such that $(i^* \ge 2)$
\mn
\begin{enumerate}
\item[$\boxdot$]    $(a) \quad \left< \langle p_{\alpha,i}:i < i^*
\rangle:\alpha < \lambda^+_{\mathfrak s} \right>$ indiscernible  
\sn
\item[${{}}$]  $(b) \quad \langle p_{\alpha,i}:i < i^* \rangle$ are
pairwise orthogonal
\sn
\item[${{}}$]  $(c) \quad p_{\alpha,i} \perp p_{\beta,j}
\Leftrightarrow (\alpha = \beta \wedge i \ne j)$
\sn
\item[${{}}$]  $(d) \quad$ by $\le^{\dom}_{\wk},\bar p_\alpha 
= \langle p_{\alpha,i}:i < i^* \rangle$ are pairwise equivalent
\sn
\item[${{}}$]  $(e) \quad$ for no $q$ do we have
$\bigwedge\limits_{\alpha} q \le^{\dom}_{\wk} p_{\alpha,0}$.
\end{enumerate}
\bigskip

\noindent
A \underline{Conjecture}:  The phenomena $\boxdot$ is possible at least in
our framework.

Still we believe the main gap holds.  At the moment two approaches
seem reasonable.
\medskip

\noindent
The first 

\noindent
\underline{Program B}:  Add imaginary elements to $M \in K^{\gs}$,
getting $K^{\gs^*}$ such that there the phenomena disappears.

Note: after expanding, do we have primes?
\medskip

\noindent
But most natural by our research is: 

\noindent
\underline{Program C}:  Define ${\gs}^+$, derived from ${\gs}$ such
that in ${\gs}^+$ we have more dichotomy and get the main or
repeat in Dp$({\gs})$ times.
\end{discussion}

\noindent
The trees here will be $\subseteq {}^{\kappa({\gs}) \ge}\mu$.
\begin{hypothesis}
\label{n5}
\mn
\begin{enumerate}
\item[$(a)$]   ${\gs}$ is a very good
$(\mu,\lambda,\kappa)-\NF$-frame, full, 
 finite base 
(for $p \in
{\cS}^{\bs}_{\gs}(M),\bas(p) \in M$ is defined)
\sn
\item[$(b)$]   in ${\gs}$ there are NF-prime (though not
necessarily uniqueness so prime among the compatible cases; from this
we should have decomposition and previously: NF-trees)
\sn
\item[$(c)$]   $\lambda_* = \lambda^{\gs}$ as just $\lambda_* =
\lambda^{< \kappa}_* \in [\lambda_{\gs},M_{\gs}),{\gs}$
stable in $\lambda_*$.
\end{enumerate}
\end{hypothesis}

\begin{definition}
\label{n8}
Assume $M_0 \le_{\gs} N,\|M_\ell\| \le \lambda_* < 
\chi \le \|N\|$ and ${\mathbf P} \subseteq
\cup\{{\cS}^{\bs}_{\gs}(M'):M' \subseteq M_0\}$. 

\noindent
1) Let $\mathbf F^1_\chi({\mathbf P},M_0,M_1,N)$ be the set of $p$
satisfying:
\mn
\begin{enumerate}
\item[$(a)$]   $p \in {\cS}^{\bs}_{\gs}(M_1)$
\sn
\item[$(b)$]   $p$ is orthogonal to ${\mathbf P}$ (i.e., $p \perp q$
for every $q \in {\mathbf P}$) and is dominated by $M_0$
\sn
\item[$(c)$]   dim$(p,N) = \chi$
\sn
\item[$(d)$]   if $q \in {\cS}^{\bs}_{\gs} (M_1)$ is dominated by
$p$ 
then $\dim(q,N) = \dim(p,N)$ (recall $\ge$ always holds).
\end{enumerate}
\mn
2) Let $\mathbf F^0_{\lambda_*,\chi}({\mathbf P},M,N)$ be defined
similarly except
\mn
\begin{enumerate}
\item[$(d)^+$]   if $q \in {\mathbf P}_{N,\lambda_*} =: 
\cup\{{\cS}^{\bs}_{\gs}(M');M' \le_{\gs} N,\|M'\| \le \lambda_*\}$ and $q$ is 
dominated by $p$ then $\dim(q,N) = \dim(p,N)$.
\end{enumerate}
\mn
3) $\mathbf F^\ell_{\lambda_*}({\mathbf P},N) = 
\cup\{\mathbf F^\ell_{\lambda_*,\chi}({\mathbf P},M,N):\|M\| < \chi \le \|N\|\}$;
similarly $F^1_\chi({\mathbf P},M_0,M_1,N)$ for $\ell =1$ we omit 
$\lambda_*$. 

\noindent
4) If $M_1 = M_0$ we may omit $M_1$; if $\lambda_* = \lambda^{\gs}$
 we may omit it.  We may omit ${\mathbf P},M_0$, too.
\end{definition}

\begin{claim}
\label{n11}
Assume that $\lambda_* < \chi \le
\|N\|,p \in \mathbf F^0_{\lambda_*,\chi}(N)$. 
\Then \, there is $M_*$ such that
\mn
\begin{enumerate}
\item[$(a)$]    $\Dom(p) \le_{\gt} M_* \le_{\gt} N$
\sn
\item[$(b)$]    $M_* \in K^{\gt}_\chi$
\sn
\item[$(c)$]    $p_* \in {\cS}^{\bs}_{\gt}(M_*)$, the non-forking
extension of $p_*$ in ${\cS}^{\bs}_{\gt}(M_*)$, has a
unique extension in ${\cS}^{\bs}_{\gs}(N)$.
\end{enumerate}
\end{claim}

\begin{proof}  We try to choose $M^*_i,N^*_i,a_i$ by induction on $i <
\lambda^{++}_*$ such that
\mn
\begin{enumerate}
\item[$\circledast$]   $(a) \quad M^*_i \le_{\gt} N$ is
$\le_{\gt}$-increasing
\sn
\item[${{}}$]   $(b) \quad \|M^*_i\| \le \chi$
\sn
\item[${{}}$]   $(c) \quad M \le_{\gt} M^*_0$
\sn
\item[${{}}$]    $(d) \quad N^*_i \le_{\gt} M^*_i,N^*_i
\le_{\gs} N,N^*_i \in K^{\gt}_{\lambda^+_*}$
\sn
\item[${{}}$]  $(e) \quad N^*_i$ is $\le_{\gt}$-increasing
\sn
\item[${{}}$]    $(f) \quad a_i \in N \backslash M^*_i,
\ortp_{\gt}(a_i,M^*_i,N)$ is not orthogonal to $p$
\sn
\item[${{}}$]   $(g) \quad \ortp_{\gt}(a_i,M^*_i,N)$ does
not fork over $N^*_i$ and $a_i \in N^*_{i+1}$
\sn
\item[${{}}$]   $(h) \quad N^*_0 = M$
\sn
\item[${{}}$]    $(i) \quad \text{ if } q \in {\cS}^{\bs}(N^*_i)$ 
and $\dim(q,N) \le \chi$ \then \, there is
$\mathbf J_q$, a maximal set 

$\qquad$ independent in $(N^*_i,N)$ which is included in $q(N)$
\end{enumerate}
\mn
(alternatively:  $N^*_i \le_{\gt} N^{**}_i \le_{\gt}
N^*_{i+1},N^*_i \le M^*_i,N^*_i$ is $\le_{\gt}$-increasing
continuous, $N^{**}_i \in K^{\gs}_{\lambda^+_*}$ depends on choice
of framework).

For $\ell=0$ there is no problem.  If $\langle M^*_j,N^*_j,N^+_j,a_j:j
< i \rangle$ is defined but we cannot choose $(M_i,a_i)$ as above then
$M^* = \cup\{M^*_j:j <i\} \cup M$ is as required.

If we have carried the induction, we have $p_i \in 
{\cS}^{\bs}(N^*_i)$ as $\pm M$ hence there are $\varepsilon <
\lambda^{++}_*$ and $q \in {\cS}^{\bs}(N^*_\varepsilon)$ such
that $q \le^{\dom}_{\wk} \{p_i:i \in [j,\lambda^{++}_*)$
for every $j < \lambda^{++}_*$ [see \S6 find/add citation;
anyhow true for f.o. - making the canonical basis indiscernible]. 

Now easy contradiction.
\end{proof}

\begin{claim}
\label{n17}
Assume $M \le_{\gs} N,N$ is $\|M\|^+$-saturated and 
$\|M\| < \chi_i \le \|N\|$ for $i=0,1,2$ and
$\chi_1 < \chi_2$ and ${\mathbf P} \subseteq \cup 
\{\cS^{\bs}_{\gs}(M'):M' \le_{\gs}M\}$. 

\noindent
1) $\mathbf F^\ell_\chi({\mathbf P},M,N) \subseteq \mathbf F^\ell_\chi({\mathbf
P},M,N)$ for $\ell <2$ and $\mathbf F^0({\mathbf P},M,N) \subseteq 
\mathbf F^1({\mathbf P},M,N)$. 

\noindent
2) If $p_i \in \mathbf F^0_{\chi_i}({\mathbf P},M,N)$ for $i=0,1$ \then \, $p_1
\perp p_2$. 

\noindent
3) If $p \in \mathbf F^0_\chi({\mathbf P},M,N),q \in {\cS}^{\bs}_{\gs}(M)$
 and $\dim(q,N) > \chi$ \then \, $p \perp q$. 

\noindent
4) In (2), (3) we can use $\mathbf F^1$.
\end{claim}

\begin{PROOF}{\ref{n17}}
1) By the definition. 

\noindent
2) By 3). 

\noindent
3) Let $M' \le_{\gs} M,M' \in K^{\gs}_{\lambda_{\gs}}$ be
such that $p,q$ does not fork over $M'$.  If the desired conclusion
fails by \ref{j53}(1) there is $r \in {\cS}^{\bs}_{\gs}
(M')$ satisfying $M' \le_{\gs} N,\|M'\| = \lambda_{\gs}$ 
such that $r$ dominated by both $p$ and $q$.  By the latter 
$\dim(r,N) \ge \dim(q,N) > \chi$ and by the former there is no such
$r$. 

\noindent
4) Because in the proof of part (3), \wilog \, $M'=M$.
\end{PROOF}

\begin{claim}
\label{n20}
Assume
\mn
\begin{enumerate}
\item[$(a)$]  $M_0 \le_{\gs} M_1 \le_{\gs} N$
\sn
\item[$(b)$]   ${\mathbf P} \subseteq \cup\{{\cS}^{\bs}_{\gs}(M'):
M' \le_{\gs} M_0\}$
\sn
\item[$(c)$]    ${\mathbf P}^* = \cup\{p \in {\cS}^{\bs}(M'),M'
\le_{\gs} N, \|M'\| \le \lambda_*\}$.
\end{enumerate}
\mn
\Then
\mn
\begin{enumerate}
\item[$(\alpha)$]   $\mathbf F^0_\chi({\mathbf P},M_0,N)$ is dense in 
${\mathbf P}^*$ 
\sn
\item[$(\beta)$]   moreover if $p \in {\mathbf P}^*$ and $\dim(p,N) 
= \chi > \lambda_*$ \then \, $p$ dominates some $q \in
\mathbf F^0_{\lambda_*,\chi}({\mathbf P},M_0,N)$
\sn
\item[$(\gamma)$]   
$\mathbf F^0_\chi({\mathbf P},N)$ is dominated by
$\mathbf F^0_\chi({\mathbf P},M_0,N),{\mathbf P}_{N,\lambda_*}$ in the sense
that, i.e., if $p \in {\mathbf P}_{N,\lambda_*}$ is $\pm {\mathbf P}^\perp$
then for some $q \in \mathbf F^0_{\lambda_*}({\mathbf P},M,N)$ is $\pm p$.
\end{enumerate}
\end{claim}

\begin{PROOF}{\ref{n20}}

\noindent
\underline{Clause $(\alpha)$}:  By Clause $(\beta)$.
\medskip

\noindent
\underline{Clause $(\beta)$}:  Assume this fails.  Let $\{p_i:i < i^*\}$ be a
maximal family of pairwise orthogonal types from ${\cP}_{N,\lambda_*}$ 
each dominated by $p$ with dim$(p_i,N) > \chi$.

Note that by our assumption toward contradiction, if $q \in 
{\cP}_{N,\lambda_*}$ is dominated by $p$, then there is $r \in 
{\cP}_{N,\lambda^*}$ dominated by $q$ and $\dim(r,N) > \chi$ so $r \perp
p_i$ for some $i$.  By \ref{g41},  
if $q \in {\cS}^{\bs}(M'),
M' \le_{\gs} N' \wedge N \le_{\gs} N,q$
dominated by $p$ then there is $r \in {\cS}^{\bs}(M'),M'' \le_{\gs} 
N'' \wedge N' \le_{\gs} N'',r$ dominated by $q$ and
$\bigvee\limits_{i} r \pm p_i$ hence $\bigvee\limits_{i} q \pm p_i$.

Let $M' \in p_{N,\lambda_*}$ be such that $\Dom(p),\Dom(p_i)
\subseteq M'$, i.e., $p \le^{\dom}_{\wk} \{p_i:i < i^*\}$
but this implies (\ref{g41}) $\dim(p,N) \ge \min\{\dim(p_i,i):
i < i^*\} \ge \chi^+ > \chi$, contradiction to the assumption on $p$.
\end{PROOF}

\begin{definition}
\label{n23}
We say $(M,{\mathbf P},M^*,\mathbf J,N^*)$ is an approximation \when \,:
\mn
\begin{enumerate}
\item[$(a)$]   $M \le_{\gs} M^* \le_{\gs} N^*$
\sn
\item[$(b)$]  $(M,M^*,\mathbf J) \in K^{3,\qr}_{\gs}$
\sn
\item[$(c)$]   ${\mathbf P} \subseteq {\cS}^{\bs}(M),{\mathbf P}_{M,\lambda_*}$
\sn
\item[$(d)$]    $\mathbf J \subseteq \{c \in M^*:\ortp(c,M,N^*)$ is
a non-forking extension of some $p \in {\mathbf P}\}$
\sn
\item[$(e)$]    if $M^* \le_{\gs} N <_{\gs} N^*$ and $c \in
N^* \backslash N$ then $\ortp(c,N,N^*)$ is orthogonal to ${\mathbf P}$.
\end{enumerate}
\end{definition}

\begin{claim}
\label{dcx.2}
Assume
\mn
\begin{enumerate}
\item[$(a)$]    $(M,{\mathbf P}_0,M^*,\mathbf J_0,N^*)$ is an
approximation where ${\mathbf P} = {\cS}^{\bs}(M)$ (or 
${\mathbf P}_0 \cup {\mathbf P}_1$ if you prefer)
\sn
\item[$(b)$]   ${\mathbf P}_1 = \{p \in {\cS}^{\bs}(M):p \perp {\mathbf P}_0\}$
\sn
\item[$(c)$]    $N^*$ is $\lambda^+_*$-saturated.
\end{enumerate}
\mn
\Then \, we can find $\mathbf J_1,M$ such that
\mn
\begin{enumerate}
\item[$(\alpha)$]   $(M,\mathbf P_0 \cup {\mathbf P}_1,M^{**},\mathbf J_0
\cup \mathbf J_1,N^*)$ is an approximation
\sn
\item[$(\beta)$]   $\mathbf J_1 \subseteq \{c \in N:\ortp(c,M,N^*)
\in \mathbf P_1\}$.
\end{enumerate}
\mn
For each $p \in {\mathbf P}_{N,\lambda_*}$ choose a maximal family
$\{q_{p,i}:i < i_p\}$ of pairwise orthogonal types from $\mathbf
F^0_{\lambda_*,\chi}({\mathbf P},M,N)$ each dominated by $p$ so $i_p <
\kappa_{\gs}$; \wilog \, $\langle q_{p,i}:i < i_p \rangle$ depend
just on $p/11$.  Let $N \in K^{\gs}_{\lambda_*}$ be such that $M
\le_{\gs} N \le_{\gs} N^*$ and $p \in {\cS}^{\bs}(M_*),I < i_p 
\Rightarrow \Dom(q_{p,i}) \le_{\gs} N$, so \wilog \, $q_{p,i} \in 
{\cS}^{\bs}_{\gs}(N)$.  For each such $q \in {\cS}^{\bs}(N)$ let 
$\mathbf J_{q,i}$ be a maximal subset of $q(N)$
independent in $(M,N)$ and let $\langle b_{q,\alpha}:\alpha < \chi_p
\rangle$ list $\mathbf J_p$ so $\chi_q = \dim(q,N)$
\mn
If $q \in {\cS}^{\bs}(N) \cap \mathbf F_{\lambda_*,N}(N)$ then
let $N^*_q \in K^{\gt}_\chi$ be such that $N \le_{\gt} N^*_q
\le_{\gt},\chi(q) = \dim(q,N)$ be such that the non-forking
extension $q^\oplus$ of $q$ in $N^*_q,q^\otimes$ has a unique
extension in ${\cS}^{\bs}(N)$. 

Let $\langle N^*_{q,\alpha}:\alpha \le \chi(q) \rangle$ is 
$\le_{\gt}$-increasing continuous with union $N^*_q,\|N^*_{q,\alpha}\| \le
\lambda_* + |\alpha|$.  Let $\langle a_{q,\alpha}:\alpha < \chi_q
\rangle$ list $N^*_q$ and let $\langle a^*_\alpha:\alpha < \lambda_*
\rangle$ list $N$.

Let $\langle p_i:9 < i < i(*) \rangle$ list ${\cS}^{\bs}(N)$
(can use less).  Now we try to choose $M_i,N_i,\mathbf I_i$ by induction
on $\alpha$ such that
\mn
\begin{enumerate}
\item[$\circledast$]   $(a) \quad M_\alpha \le_{\gt} M^*$ is
increasing continuous, $M_\alpha \in K^{\gt}_{\le \lambda_* + |\alpha|}$
\sn
\item[${{}}$]   $(b) \quad N_\alpha \le_{\gt} N^*$ is increasing
continuous $N_\alpha \in K^{\gt}_{\le \lambda_* + |\alpha|}$
\sn
\item[${{}}$]   $(c) \quad \mathbf I_\alpha \subseteq \{c \in
N:\ortp_{\gt}(c,M,N^*) \in {\mathbf P}_1\}$ is increasing
continuous 

\hskip25pt $|I_\alpha| \le \lambda_* + |\alpha|$
\sn
\item[${{}}$]   $(d) \quad (M_\alpha,N_\alpha,\mathbf I_\alpha) \in
K^{3,\vq}_{\gt}$ and even $\in K^{3,\cn}_{\gt}$

\hskip35pt (recall cn stands for constructible, it is unreasonable to say 

\hskip35pt prime as if 
${\gt}$-models of a stable $T$, there are no primes)
\sn
\item[${{}}$]   $(e) \quad$ if $N \nsubseteq N_\alpha$ \then \, if
possible under the present restriction, for some 

\hskip35pt $\zeta < \lambda_*,\ortp(a^*_\zeta,
N_{\alpha +1},N)$ forks over $N_\alpha$
\sn
\item[${{}}$]   $(f) \quad$ if $\alpha = i(*) \times \beta +i,N
\subseteq N_\alpha$ and $i < i(*)$, \then \, if possible under 

\hskip35pt the present restrictions, for some $\zeta < \chi(p_i)$, 

\hskip35pt $\ortp_{\gt}(a_{q,\alpha},N_{\alpha +1},N)$ forks over $N_\alpha$.
\end{enumerate}
\mn
So for some $\alpha(*),(M_\alpha,N_\alpha,\mathbf I_\alpha)$ is defined
iff $\alpha \le\alpha(*)$   hence 
{\rm NF}$(M_\alpha,M^*,N_\alpha,N^*)$.
Let $M^{**} \le_{\gs} N^*$ be prime over $N_\alpha \cup M^*,\mathbf
J_1 = \mathbf I_{\alpha(*)}$ and let $M^\oplus \le_{\gs} N^*$ be
such that $(M^*,M^\otimes,\mathbf J_1) \in K^{3,\qr}_{\gs}$ we
shall show that they are as required.  The main point is to prove the
following is impossible
\mn
\begin{enumerate}
\item[$\boxdot$]   $M^{**} \le_{\gs} N' <_{\gs} N^*,c \in
N^* \backslash N',\ortp_{\gs}(c,N',N^*) \pm M$. 
\end{enumerate}
\mn
We first prove
\mn
\begin{enumerate}
\item[$\boxdot_1$]   if $\alpha < \lambda^+_*$ and $N \nsubseteq
N_\alpha$ then in $\circledast(e)$ there is such $\zeta$.  
\end{enumerate}
\mn
[Why?  Note that $\NF_{\gt}(M_\alpha,M^*,N_\alpha,N^*)$ by
earlier claim.  
Let $N^+_\alpha \le_{\gs} N$ be prime over $M^* \cup N_\alpha$.
If $N \subseteq N^+_\alpha$ we can easily find $(M_{\alpha
+1},N_{\alpha +1})$ such that: $N_{\alpha +1} \cap N \backslash
N_\alpha \ne \emptyset,(M_{\alpha +1},N_{\alpha +1},\mathbf I_\alpha)
\in K^{3,\cn}_{\gt},M_\alpha \le_{\gt} M_{\alpha +1}
\le_{\gt} M^*,N_\alpha \le_{\gt} N_{\alpha +1} \le_{\gt}
N^+_\alpha$, and we are done.  Otherwise, let $c \in N \backslash
N^+_\alpha$, so $\ortp(c,N^+_\alpha,N) \perp {\mathbf P}_0$.  If
$\ortp(c,N^+_\alpha,N) \perp M$ then let $N^\oplus_\alpha \le_{\gs}
N^*$ be such that $(N^+_\alpha,N^\oplus_\alpha,c) \in
K^{3,\pr}_{\gs}$ and proceed as in the case $N \subseteq
N^+_\alpha$.  So assume $\ortp(c,N^+_\alpha,N) \pm M,\ortp
(c,N^+_\alpha,N) \perp {\mathbf P}_0$, hence for some $p \in {\cP}_1,
p \pm \ortp(c,N^+_\alpha,N)$.  As $p \perp {\mathbf P}_0$ the
non-forking extension $p^+$ of $p$ in ${\cS}^{\bs}(N^+_\alpha)$ 
is $\lambda^+_*$-isolated, but $N^*$ is
$\lambda^+_*$-saturated?? hence there is $b \in N$ realizing $p^+$
such that $\{b,c\}$ is not independent over $N^+_\alpha$.  Now we can
proceed as in the case $N \subseteq N^+_\alpha$ getting $b \in
N_{\alpha +1},\ortp_{\mathfrak t}(c,N_{\alpha +1},N^*)$ forks over
$N_\alpha$.  So we are done.]
\mn
\begin{enumerate}
\item[$\boxdot_2$]   $N \subseteq N_{\alpha(*)}$ for some $\alpha(*)
< \lambda^+_*$.
\end{enumerate}
\mn
[Why?  By $\boxdot_1$.]  
\mn
\begin{enumerate}
\item[$\boxdot_3$]   $(a) \quad$ if $q = q_j \in {\cS}^{\bs}(N)$ 
is $\perp {\mathbf P}_0$ and $q \in \mathbf F^0_{**,\chi}(N^*)$ 
and $N \subseteq N_\alpha,\alpha < \chi,\alpha = j \mod i(*)$, \then
\, in clause (f) of $\circledast$ there is such $\zeta$
\sn
\item[${{}}$]   $(b) \quad$ if $\chi(q_i) \le \alpha$ then $N_{q_i}
\subseteq N_\alpha$.
\end{enumerate}
\mn
[Why?  We prove it by induction on $\chi$.  Let us prove clause (a),
then clause (b) follows as in $\boxdot_2$ so let $N^+_\alpha
\le_{\gs} N^*$ be prime over $M^* \cup N_\alpha$, exists as
$\NF_{\gt}(M_\alpha,M^*,N_\alpha,N^*)$ see 
earlier claim;  
if $N_q \subseteq
N^+_\alpha$ we are done as in the proof of $\boxdot_1$, so let $b \in
N_q \backslash N^+_\alpha$.  As there $\ortp_{\gs}(b,N^+_\alpha,N^*)
\perp {\mathbf P}$ and \wilog \, $\ortp_{\gs}(b,N^+_\alpha,N^*) \perp
M$.  So there is $p \in {\mathbf P}_1$ not orthogonal to it.  Hence there
is $q_* \in {\cS}^{\bs}(N)$ dominated by $p$ not orthogonal
to $\ortp_{\gs}(b,N^+_\alpha,N^*)$ such that $q_* \in \mathbf
F^0_{\lambda_*,*}(N)$ and let $\chi_*$ be such that $q_* \in \mathbf
F^0_{\lambda_*,\chi_*}(N)$.  So necessarily $N_{q_*} \nsubseteq
N_\alpha$ hence $\chi_* > \alpha$.]
\end{claim}

\noindent
We first try to analyze the case of quite saturated models.
\begin{claim}
\label{n26}
If $M_0,{\mathbf P},N^*$ satisfies
$\circledast_1$, \then \, $M_1,N_1,\mathbf J$ satisfies $\circledast_2$
where

$\circledast_1 = \circledast^1_{M_0,N} \quad \quad  
M_0 \le_{\gs} N^*,M_0 \in K_\lambda,{\mathbf P} \subseteq 
{\cS}^{\bs}(M)$ and

\hskip74pt $N^*$ is $\lambda^+_*$-saturated 
or just $\lambda^+$-saturated.  
\smallskip

$\circledast_2 = \circledast^2_{M_0,M_1,N_1,\mathbf J,N^*} \qquad$ 
the following holds
\mn
\begin{enumerate}
\item[$(a)$]   $M_0 \le_{\gs} M_1 \le_{\gs} N_1 \le_{\gs} N^*$,
\sn
\item[$(b)$]    $\mathbf J$ is independent in $(M_1,N_1)$
\sn
\item[$(c)$]  if $c \in \mathbf J$ then $p = \ortp(c,M_1,N^*)
\in \mathbf F^0({\mathbf P},M_0,N^*)$
\sn
\item[$(d)$]   $(M_1,N_1,\mathbf J) \in K^{3,\qr}_{\gs}$
\sn
\item[$(e)$]   if $q \in {\cS}^{\bs}(N_1)$ does not fork over
$M_1$ and is $\perp {\mathbf P}$ \then \, $q$ has a unique extension in
${\cS}_{\gs}(N^*)$
\sn
\item[$(f)$]   if $p \in \mathbf F^0({\mathbf P},M_0,N^*)$ then $|\{c \in
\mathbf J:\ortp_{\gs}(c,M_0,N) = p\}| = \dim(p,N^*)$
\sn
\item[$(g)$]    $(M_0,M_1,\mathbf J' \cup \mathbf J'') \in
K^{3,\vq}_{\gs}$ where for $c \in \mathbf J$ we have 
$c \in \mathbf J' \Leftrightarrow \ortp_{\gs}
(c,M_0,N^*) \in \mathbf F^0({\mathbf P},M_0,N^*),c \in \mathbf
J'' \Rightarrow \ortp_{\gs}(c,M_0,N^*) \perp \mathbf F^0({\mathbf P},M_0,N^*)$. 
\end{enumerate}
\end{claim}

\begin{PROOF}{\ref{n26}} 
 We choose by induction on $i < (\lambda^3)^+$ 
 a triple
$(M_i,N_i,\mathbf J_i,\mathbf J'_i,\mathbf J''_i)$ such that: 
\mn
\begin{enumerate}
\item[$(a)$]   $M_i \le_{\gs} N^*$ is from $K_{\lambda^{\gs}}$
\sn
\item[$(b)$]   $M_i$ is $\le_{\gs}$-increasing continuous

$N_i$ is $\le_{\gs}$-increasing continuous, 
$M_i \le_{\gs} N_i \le_{\gs} N^*$
\sn
\item[$(c)$]  if $i=j+1,q \in {\cS}^{\bs}_{\gs}(M_i),q \perp
{\mathbf P}$ and there is $(M',r)$ such that $M' \le_{\gs} N^*$ of
cardinality $\lambda$ and $r \in {\cS}^{\bs}_{\gs}(M')$ is
dominated by $q$ and $\dim(p,N^*) < \dim(q,N^*)$ \then \, there
is such $r \in {\cS}^{\bs}_{\gs}(M_i)$ (hence also if $\cf(i) 
\ge \kappa_{\gs}$ this holds)
\sn
\item[$(d)$]   if $j<i$ then:
\sn
\begin{enumerate}
\item[$(\alpha)$]   $\mathbf J_j \backslash (\mathbf J_j \cap M_i)
\subseteq \mathbf J_i$
\sn
\item[$(\beta)$]   $\NF_{\gs}(M_i \cap _j,M_i,N_j,N^*)$, 
\sn
\item[$(\gamma)$] $(M_j,M_i \cap N_j,\mathbf J_j \cap M_i) \in 
K^{3,\qr}_{\gs}$.  
\end{enumerate}
\item[$(e)$]   $(M_i,N_i,\mathbf J_i) \in K^{3,\vr}_{\gs}$
\sn
\item[$(f)$]   $\mathbf J$ is a maximal subset of $N^* \backslash M_i$
which is independent in $(M_i,N^*)$ and satisfies clause (c)
\sn
\item[$(g)$]   if $\cf(i) = \kappa$ then there is a triple
$(N^-_i,a_i,c_i,q_i)$ such that:
\sn
\begin{enumerate}
\item[$(\alpha)$]  $M_i \le_{\gs} N^-_i \le_{\gs} N_i$
\sn
\item[$(\beta)$]   $N^-_i \in K_\lambda$
\sn
\item[$(\gamma)$]  $b_i,c_i \in N^* \backslash N_i$
\sn
\item[$(\delta)$]   $q_i = \ortp(c_i,N_i,N^*)$ does not fork
over $M_i$ and is as in (c)
\sn
\item[$(\varepsilon)$]   $\ortp_{\gs}(b_i,N_i,N^*) \pm q$
\sn
\item[$(\zeta)$]   $\ortp_{\gs}(b_i,N_i,N^*)$ does not fork over $N^-_i$'
\sn
\item[$(\theta)$]    $N^-_i \subseteq N_{i+1}$
\end{enumerate}
\item[$(h)$]   $\mathbf J'_i \subseteq \{c \in N^*:\ortp_{\gs}(c,M_0,N) 
\in \mathbf F^0,({\mathbf P},M_0,N^*)\}$ increases with $i$
\sn
\item[$(i)$]   $\mathbf J''_i \subseteq \{c \in N^*:\ortp_{\gs}
(c,M_0,N^*)$ orthogonal to $\mathbf F^0({\mathbf P},M_0,N^*)$
\sn
\item[$(j)$]   if $i=j+2$ then $(M_0,M_i,\mathbf J'_i,\mathbf J'_i) \in
K^{3,\uq}_{\gs}$ 
\sn
\item[$(k)$]   $\mathbf J_0$ satisfies clause (f) of the claim.
\end{enumerate}
\mn
So for some $i(*) \le (\lambda^{\gs})^+$, we have defined for $i$ 
iff $i < i(*)$.
\bigskip

\noindent
\underline{Case 1}:  We succeed to carry the induction, i.e., $i(*) 
= (\lambda^{\gs})^*$.

Let $\bar a_i = \langle a^i_\varepsilon:\varepsilon < \lambda \rangle
\in {}^{\lambda +2}(N^*)$ list $N^-_i \cup \{b_i,c_i\},\{a^i_{4
\varepsilon}:\varepsilon < \lambda\}$ list $b_i,c_i \notin
\{a^i_\varepsilon:\varepsilon \ne 1,2\},N^-_i \cap
M_i,(a^i_{\lambda^+},a^i_{\lambda +1}) = (b_i,c_i)$.  By x.x. for some
stationary $S \subseteq \{i < (2^\lambda)^+:\cf(i) = \lambda$) the
sequence $\langle \bar a_i:i \in S \rangle$ is convergent and
indiscernible over $M_{j(*)}$.

Let $j(*)$ be minimal such that otp$(S \cap j(*)) \ge \kappa$.  By
\ref{j56} we get a contradiction to the maximality of $\mathbf J_{j(*)}$.
\bigskip

\noindent
\underline{Case II}:  for some limit $\delta$ we have defined for all $i <
\delta$.
\bigskip

\noindent
\underline{Case III}:  $i(*)=0$.

Trivial.
\bigskip

\noindent
\underline{Case IV}:  $i(*)=j+1,\cf(j) \ne \lambda^+$.

Easy.
\bigskip

\noindent
\underline{Case V}:  $i(*)=j+1,\cf(j) = \lambda^+$.
\end{PROOF}

\noindent
We prove that $(M_j,N_j,\mathbf J_j)$ are as required.
\bigskip

\begin{claim}
\label{n32}
Assume ${\gs}$ has $\NDOP$.

Assume $\circledast^2_{M_0,M_1,N_1,\mathbf J_1,N^*}$ 
and $N^*$ is $\|M_1\|^+$-saturated, $\|M_1\| = 
\|M_1\|^{< \kappa({\gs})}$. 

\Then \, we can find $N_0,\mathbf J_0$ such that:
\mn
\begin{enumerate}
\item[$(\alpha)$]   $\NF_{\gs}(M_0,M_1,N_0,N_1)$
\sn
\item[$(\beta)$]   $(M_0,N_0,\mathbf J_0) \in K^{3,\qr}_{\gs}$
\sn
\item[$(\gamma)$]   $c \in \mathbf J_0 \Rightarrow 
\ortp_{\gs}(c,M_0,N_0) \in \mathbf F^*({\mathbf P},M_0,N^*)$.
\end{enumerate}
\end{claim}

\begin{proof}  
Let $\Theta_\ell = \{\dim(p,N^*):p \in \mathbf F^*({\mathbf P},M_\ell)\}$.
Fix for the time being $\chi \in
\Theta_\ell$, let ${\mathbf P}_\chi = \mathbf F^0_\chi({\mathbf
P},M_0,M_1,N),{\mathbf P}^*_\chi = \mathbf F^0_\chi({\mathbf P},M_0,N),\mathbf
J^1_\chi = \{c \in \mathbf J_1:\ortp_{\gs}(c,M_1,N^*) \in 
{\mathbf P}_\chi\}$.  Let $\mathbf J^1_\chi = 
\cup\{\mathbf J^1_{\chi,\varepsilon}:\varepsilon < \chi\}$ where $\mathbf
J^1_{\chi,\varepsilon}$ is increasing with $\varepsilon$.
\bigskip

\noindent
\underline{First Proof}:  Now we choose by induction on $i$ an element (of
sequences of length $< \kappa_{\gs}$) $c_{\chi,i}$ such that:
\mn
\begin{enumerate}
\item[$(i)$]   $c_{\chi,i} \in N^* \backslash M_0 \backslash
\{c_{\chi,j}:j<i\}$ 
\sn
\item[$(ii)$]   $\ortp_{\gs}(c_{\chi,i},M_0,N^*) \in {\cP}^*_\chi$
\sn
\item[$(iii)$]   $\{c_{\chi,j}:j \le i\}$ is independent in $(M_0,N^*)$
\sn
\item[$(iv)$]   $\varepsilon_\chi(c_{\chi,i}) \le \varepsilon(c)$
for any $c$ satisfying (i), (ii), (iii) where $\varepsilon_\chi(c) =
\varepsilon_\chi(c,N^*) = \text{ Min}\{\varepsilon \le \chi$: if
$\varepsilon < \chi$ then $c$ is not given over $(M_0,M_1 \cup \mathbf
J^1_{\chi,\varepsilon})$
\sn
\begin{enumerate}
\item[$\boxdot_1$]   $\varepsilon_\chi(c_j) \le 
\varepsilon_\chi(c_i)$ for $j<i$.

[Why?  Trivially.]
\sn
\item[$\boxdot_2$]   for every $\varepsilon$ the set
$\{i:\varepsilon_\chi(c_i) \le \varepsilon\}$ is an ordinal
$i[\varepsilon] < (\|M_1\| + |\varepsilon|)^+$ 

[Why?  By one of the basic properties of dimension.
]
\sn 
\item[$\boxdot_3$]   if $N^* \le_{\gs} N^+,c \in N^+
\backslash M_0 \backslash \{c_{\chi,i}:i < \chi\}$ and 
$\ortp_{\gs}(c,M_0,N^+) \in {\mathbf P}^*_\chi$ and $\{c_{\chi,i}:i < \chi\} \cup
\{c\}$ is independent in $(M_0,N^+)$ \then \,
$\varepsilon_\chi(c,N^+) = \chi$.

[Why?  Assume $c$ is a counterexample that $\varepsilon_\chi(c,N^+) =
\zeta < \chi$ and let $i(*) = i[\zeta]$.
\end{enumerate}
\end{enumerate}
\mn
Now 
\bigskip

\noindent
\underline{Second proof of \ref{n32}}:
\medskip

\noindent
\underline{ Problems 
}:

\noindent
1) Try to use \ref{n20} for \ref{n32}. 

\noindent
2) The problem is (a) or (b) where
\mn
\begin{enumerate}
\item[$(a)$]   the existence of $q$ dominated by $M_1$ orthogonal to
$M_0,(M_0,M_1,a) \in K^{3,\uq}_{\gs}$ and no $r \in {\cS}^{\bs}_{\gs}
(M_1)$ is dominated by $q$, so possibly dim$(q,N) >
\sup\{\dom(p,N):p \in {\cS}^{\bs}_{\gs}(M)\}$. 

Moreover, if $M_1 \le M_2,(M_0,M_2,a) \in K^{3,\uq}_{\gs}$ nothing
changes.  There is a hidden independence property, but some cases
seemingly has just one more unary function so there is a decomposition
\sn
\item[$(b)$]   instead $M_0,M$ we have $\langle M_i:i \le \delta
\rangle$ semi continuous, $\delta < \kappa_{\mathfrak s}$ and we look at
$q$ dominated by $M_\delta$ orthogonal to $M_i$ for $i < \delta$
(needed only if we like to have decomposition to trees $\subseteq
{}^{\kappa >} \mu$.
\end{enumerate}
\mn
3) If NDOP and for some $\{M_i:i < \alpha\}$ independent over $M$ the
prime model over $\cup\{M_i:i < \alpha\}$ is not minimal, \then \,
any logic fix-length-game + quantifiers on dimension does not
characterize models up to isomorphism. 

\noindent
4) Have imaginary elements for
\mn
\begin{enumerate}
\item[$(a)$]   $p/$parallelism
\sn
\item[$(b)$]   $p/\mathbf E,p_1 \mathbf E p_2$ iff they always have the
same dimension (i.e., each dominated the other). 
\end{enumerate}
\mn
5) \underline{conjecture}: if $p$ is dominated by $M$ but dominate no one in
${\cS}^{\bs}(M)$ \then \, on some $p/E$ there is a group hence is
of depth $0$ (and even $p^\infty$ is?). 

\noindent
6) Groups: see what I wrote to Lessman.   

\noindent
7) Take $\NDOP$ from 705 in \cite{Sh:705}: if ${\gs}$ is excellent
exactly up to $n+1$ then ${\gs}^+$ is excellent exactly up to $n$
or up to $n+1$.  Help in \S12 but need lower.  But: does $EM(I)$ help?

Define slim if $EM(I,\Phi)$ is without order this as dividing line.
\end{proof}
\newpage


\bibliographystyle{amsalpha}
\bibliography{shlhetal}

\end{document}